\documentclass[a4paper]{amsart}
\pdfoutput=1
\usepackage[colorlinks,citecolor=blue,urlcolor=black,linkcolor=black]{hyperref}
\usepackage{amssymb, latexsym}
\usepackage{amsmath, amsfonts, amsthm}
\usepackage{stmaryrd} 
\usepackage[headings]{fullpage}
\usepackage{cmap}
\usepackage{tikz}
\usepackage[all]{xy}
\usepackage{ifpdf}
\usepackage{enumerate}
\usepackage{verbatim}

\newtheorem{ufact}{Fact}
\newtheorem{ucor}[ufact]{Corollary}
\newtheorem{conj}{Conjecture}
\newtheorem{uthm}{Theorem}
\newtheorem*{THMA}{Theorem A}
\newtheorem*{THMB}{Theorem B}
\newtheorem*{THMC}{Theorem C}
\newtheorem*{THMD}{Theorem D}

\newtheorem{thm}{Theorem}[section]
\newtheorem{lemma}[thm]{Lemma}
\newtheorem{cor}[thm]{Corollary}
\newtheorem{prop}[thm]{Proposition}
\newtheorem{fact}[thm]{Fact}
\newtheorem{claim}{Claim}[thm]

\theoremstyle{definition}
\newtheorem*{udef}{Definition}
\newtheorem{defn}[thm]{Definition}
\newtheorem{notation}[thm]{Notation}
\newtheorem{example}[thm]{Example}

\theoremstyle{remark}
\newtheorem{remark}[thm]{Remark}

\DeclareMathOperator{\reg}{Reg}
\DeclareMathOperator{\cl}{cl}
\DeclareMathOperator{\Tr}{Tr}
\DeclareMathOperator{\cf}{cf}
\DeclareMathOperator{\dom}{dom}
\DeclareMathOperator{\rng}{Im}
\DeclareMathOperator{\im}{Im}
\DeclareMathOperator{\Drop}{Drop}
\DeclareMathOperator{\otp}{otp}
\DeclareMathOperator{\acc}{acc}
\DeclareMathOperator{\nacc}{nacc}
\DeclareMathOperator{\p}{P}

\renewcommand\mid{\mathrel{|}\allowbreak}
\renewcommand{\restriction}{\mathbin\upharpoonright}

\renewcommand\sin{\mathrel{\underline{\in}}}
\newcommand\nsin{\mathrel{\underline{\notin}}}
\newcommand\s{\subseteq}
\newcommand\stree{\subseteq}
\newcommand\sq{\sqsubseteq}
\newcommand\br{\blacktriangleright}
\newcommand\diagonal{\bigtriangleup}
\newcommand\symdiff{\mathbin\triangle}
\newcommand\bks{\setminus}
\newcommand*\axiomfont[1]{\textsf{\textup{#1}}}
\newcommand\zf{\axiomfont{ZF}}
\newcommand\zfc{\axiomfont{ZFC}}
\newcommand\gch{\axiomfont{GCH}}
\newcommand\ch{\textup{CH}}
\newcommand*\pred[1]{#1_\downarrow}
\newcommand\faithful{\Phi^{\textup{faithful}}}
\newcommand*\cvec[1]{\vec{\mathcal#1}}
\newcommand*\acts[2]{\cvec{#2}^{#1}}

\title{Distributive Aronszajn trees}

\author{Ari Meir Brodsky}
\address{Department of Mathematics, Bar-Ilan University, Ramat-Gan 5290002, Israel.}
\curraddr{Department of Mathematics, Ariel University, Ariel 4070000, Israel.}
\urladdr{http://u.math.biu.ac.il/~brodska/}

\author{Assaf Rinot}
\address{Department of Mathematics, Bar-Ilan University, Ramat-Gan 5290002, Israel.}
\urladdr{http://www.assafrinot.com}

\thanks{This work was partially supported by the Israel Science Foundation (Grant~\#1630/14).
The revision of this paper took place when the first author was supported by the Center for Absorption in Science, Ministry of Aliyah and Integration, State of Israel.}


\begin{document}
\begin{abstract} Ben-David and Shelah proved that if $\lambda$ is a singular strong-limit cardinal and $2^\lambda=\lambda^+$,
then $\square^*_\lambda$ entails the existence of a normal $\lambda$-distributive $\lambda^+$-Aronszajn tree.
Here, it is proved that the same conclusion remains valid after replacing the hypothesis $\square^*_\lambda$  by $\square(\lambda^+,{<}\lambda)$.

As $\square(\lambda^+,{<}\lambda)$ does not impose a bound on the order-type of the witnessing clubs, our construction is necessarily different from that of Ben-David and Shelah,
and instead uses walks on ordinals augmented with club guessing.

A major component of this work is the study of postprocessing functions and their effect on square sequences.
A byproduct of this study is the finding that for $\kappa$ regular uncountable, $\square(\kappa)$ entails the existence of a partition of $\kappa$ into $\kappa$ many fat  sets.
When contrasted with a classic model of Magidor, this shows that it is equiconsistent with the existence of a weakly compact cardinal that $\omega_2$ cannot be split into two fat  sets.
\end{abstract}

\maketitle
\section*{Introduction}

Two central themes in Combinatorial Set Theory are uncountable trees and square principles.
A poset $(T,{<_T})$ is a \emph{tree} if the downward cone $\pred{x}:=\{ y\in T \mid y<_T x \}$ of every node $x\in T$ is well-ordered.
For any ordinal $\alpha$, we write $T_\alpha:=\{ x\in T\mid \otp(\pred{x},{<_T})=\alpha\}$
for the \emph{$\alpha^{th}$ level} of the tree $(T,{<_T})$.
For a regular uncountable cardinal $\kappa$, the tree $(T,{<_T})$ is said to be a \emph{$\kappa$-tree}, provided that $|T_\alpha|<\kappa$ for all ordinals $\alpha$,
and $\{\alpha\mid T_\alpha\neq\emptyset\}=\kappa$.
A $\kappa$-tree $(T,{<_T})$ is said to be \emph{normal} if for all $x\in T$, $\{\alpha\mid x\text{ is compatible with some node from }T_\alpha\}=\kappa$.
A \emph{$\kappa$-Aronszajn tree} is a $\kappa$-tree having no chains of size $\kappa$.
A \emph{$\kappa$-Souslin tree} is a $\kappa$-Aronszajn tree having no antichains of size $\kappa$.

A $\lambda^+$-tree is said to be \emph{special} if it may be covered by $\lambda$ many antichains.
A normal $\lambda^+$-Aronszajn tree is said to be \emph{$\lambda$-distributive} if every intersection of $\lambda$ many dense open subsets of the tree is dense.\footnote{Here,
\emph{dense} and \emph{open} are in the forcing sense under the (reverse) order of the tree. That is, $D\s T$ is \emph{dense} if $T=\bigcup_{x\in D} \pred{x}$ and \emph{open} if $x^\uparrow\s D$ for every $x\in D$.
In particular, the tree is \emph{$\lambda$-distributive}
iff forcing with it does not add a new function $f:\lambda\rightarrow V$.}
It is not hard to see that for any infinite cardinal $\lambda$, and any normal $\lambda^+$-Aronszajn tree $\mathcal T=(T,{<_T})$:\footnote{Note that any $\kappa$-Souslin tree admits a normal $\kappa$-Souslin subtree (see, e.g., \cite[Lemma~2.4]{rinot20}).}
$$\mathcal T\text{ is }\lambda^+\text{-Souslin}\implies\mathcal T\text{ is }\lambda\text{-distributive}\implies\mathcal T\text{ is not special}.$$

Three folklore conjectures in the study of $\lambda^+$-trees read as follows:
\begin{conj} Assume $\gch$, and that $\lambda$ is some regular uncountable cardinal.

Then there exists a $\lambda^+$-Souslin tree.
\end{conj}

\begin{conj} Assume $\gch$, and that $\lambda$ is some singular cardinal.

If there exists a special $\lambda^+$-Aronszajn tree, then there exists a $\lambda^+$-Souslin tree.
\end{conj}

\begin{conj} Assume $\gch$, and that $\lambda$ is some singular cardinal.

If there exists a $\lambda^+$-Aronszajn tree, then there exists a normal $\lambda$-distributive $\lambda^+$-Aronszajn tree.
\end{conj}

We shall come back to these conjectures soon. Now, let us touch upon square principles via a concrete example:
\begin{udef}$\square_\xi(\kappa,{<}\mu)$ asserts the existence of a sequence $\langle \mathcal C_\alpha\mid\alpha<\kappa\rangle$ such that for all limit $\alpha<\kappa$:
\begin{itemize}
\item $\mathcal C_\alpha$ is a nonempty collection of clubs in $\alpha$, each of order-type $\le\xi$;
\item $|\mathcal C_\alpha|<\mu$;
\item each $C\in\mathcal C_\alpha$ satisfies $C\cap\bar\alpha\in\mathcal C_{\bar\alpha}$ for every accumulation point  $\bar\alpha$ of $C$;
\item there exists no club $C$ in $\kappa$ such that $C\cap\bar\alpha\in\mathcal C_{\bar\alpha}$ for every accumulation point $\bar\alpha$ of $C$.
\end{itemize}
\end{udef}

It is clear that both Aronszajn trees and square sequences are instances of incompactness,
but there is a deeper connection between the two. To exemplify:
\begin{ufact}\label{fact02}
For every infinite cardinal $\lambda$:\footnote{Note that $\square_\lambda(\lambda^+,{<}\lambda^+)$ is better known as $\square^*_\lambda$,
and that $\square_\xi(\kappa,{<}\mu)$ with $\xi=\kappa$ and $\mu=2$ is better known as $\square(\kappa)$.
For notational simplicity, we shall hereafter omit the subscript $\xi$ whenever $\xi=\kappa$.}
\begin{enumerate}
\item (Jensen, \cite{MR0309729}) $\square_\lambda(\lambda^+,{<}\lambda^+)$ holds iff there exists a special $\lambda^+$-Aronszajn tree;
\item (Ben-David and Shelah, \cite{MR0861900}) If $\lambda$ is a singular strong-limit cardinal and $2^\lambda=\lambda^+$,
then $\square_\lambda(\lambda^+,{<}\lambda^+)$ entails the existence of a normal $\lambda$-distributive $\lambda^+$-Aronszajn tree.
\end{enumerate}

For every regular uncountable cardinal $\kappa$:
\begin{enumerate}
\item[(3)]  (Todorcevic, \cite{MR908147}) $\square(\kappa,{<}\kappa)$ holds iff there exists a $\kappa$-Aronszajn tree;
\item[(4)] (K\"onig, \cite{MR2013395}) If $\square(\kappa,{<}2)$ holds, then there exists a uniformly coherent $\kappa$-Aronszajn tree.\footnote{For the definition of \emph{uniformly coherent}, see \cite[\S3.1]{MR2013395}.}
\end{enumerate}
\end{ufact}
Coming back to the above-mentioned conjectures, as the reader probably expects,
the best known results toward these conjectures are formulated in the language of square principles.

As for Conjecture~1, the best known result may be found in \cite{paper24}.
That paper deals with  $\lambda^+$-Souslin trees for arbitrary uncountable cardinals $\lambda$;
however, for $\lambda$ regular, the arguments of that paper generalize to show the following.\footnote{The details will appear in~\cite{paper35}.}
\begin{ufact}[implicit in \cite{paper24}] Assume $\gch$, and that $\lambda$ is some regular uncountable cardinal.

If $\square(\lambda^+,{<}\lambda)$ holds, then there exists a $\lambda^+$-Souslin tree.
\end{ufact}
Note that this is just one step away from verifying Conjecture~1, since $\square(\lambda^+,{<}\lambda^+)$ for regular $\lambda$ is already a consequence of $\gch$.

The best known results toward Conjecture~2 are as follows.
\begin{ufact} Assume $\gch$, and that $\lambda$ is some singular cardinal.

Then there exists a free $\lambda^+$-Souslin tree, in any of the following cases:
\begin{itemize}
\item (\cite{paper32}) There are a special $\lambda^+$-Aronszajn tree and a non-reflecting stationary subset of $E^{\lambda^+}_{\neq\cf(\lambda)}$.
\item (\cite{paper26}) $V=W^{\mathbb Q}$, where $W$ is an inner model of $\zfc+\gch$
in which $\lambda$ is inaccessible, and $\mathbb Q$ is some $\lambda^{+}$-cc notion of forcing of size $\lambda^+$.
\end{itemize}
\end{ufact}

In this paper, we deal with Conjecture~3. Our result is again just one step away from verifying it. It is proved:
\begin{THMA} Assume $\gch$, and that $\lambda$ is some singular cardinal.

If either
$\square(\lambda^+,{<}\lambda)$ or $\square_\lambda(\lambda^+,{<}\lambda^+)$ holds,
then there exists a normal $\lambda$-distributive $\lambda^+$-Aronszajn tree.
\end{THMA}

Attempting to construct a normal $\lambda$-distributive $\lambda^+$-Aronszajn tree raises the following interesting question:
\emph{How many different constructions of Aronszajn trees are there?}
More specifically, when constructing a $\lambda^+$-tree, what strategies are available
for ensuring that the tree construction can be continued all the way up to height $\lambda^+$
while preventing the birth of a chain of size $\lambda^+$?

Curiously enough, virtually all standard constructions of $\lambda^+$-Aronszajn trees are steered towards getting
either a special $\lambda^+$-tree or a $\lambda^+$-Souslin tree.
The point is that in each of these extreme cases,
there is an abstract combinatorial fact that secures the non-existence of a chain of size $\lambda^+$.
However, normal special $\lambda^+$-trees are not $\lambda$-distributive,
and Souslin trees seem too good to be derived in our desired scenarios.

Now, let us examine the two cases covered by Theorem~A.
The second case has already been established by Ben-David and Shelah in~\cite{MR0861900},\footnote{However, we prove it from a weaker arithmetic hypothesis --- see the remark before Corollary~\ref{cor39}.}
and their proof builds crucially on the fact that for every \emph{singular} cardinal $\lambda$, if $\square_\lambda(\lambda^+,{<}\lambda^+)$ holds,
then it may be witnessed by a sequence $\langle \mathcal C_\alpha\mid\alpha<\lambda^+\rangle$ in which each $C\in\bigcup_{\alpha<\lambda^+}\mathcal C_\alpha$ has order-type strictly smaller than $\lambda$.
Indeed, the latter is the key to constructing a $\lambda$-splitting tree all the way up to $\lambda^+$ while preventing the birth of a chain of size $\lambda^+$.

To tackle the first case of Theorem~A, where the clubs witnessing $\square(\lambda^+,{<}\lambda)$ have unrestricted order-type,
one has to find a fundamentally different way to prevent the outcome tree from admitting a chain of size $\lambda^+$.
Eventually, we ended up constructing the sought normal $\lambda$-distributive $\lambda^+$-Aronszajn tree using the method of \emph{walks on ordinals},
as a tree of the form $\mathcal T(\rho_0^{\vec C})$, and we have used \emph{club guessing} to ensure the non-existence of a chain of size $\lambda^+$.
To the best of our knowledge, this is the first time that club guessing plays a role in constructions of walks-on-ordinals trees.

\medskip

Motivated by this finding, we decided to look for a walks-on-ordinals proof of the following consequence of Fact~\ref{fact02},
Clauses (1) and (2).
\begin{ucor}[\cite{MR0309729}+\cite{MR0861900}]\label{cor05} Suppose that $\lambda$ is a singular strong-limit cardinal and $2^\lambda=\lambda^+$.

If there exists a special $\lambda^+$-Aronszajn tree, then there exists a nonspecial $\lambda^+$-Aronszajn tree.
\end{ucor}

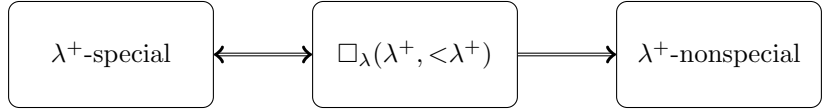
\ifpdf
\begin{center}
\tikzstyle{line} = [draw, -latex']
\tikzstyle{block} = [rectangle, draw, 
    text width=7em, text centered, rounded corners, minimum height=4em]
\begin{tikzpicture}[node distance = 4cm, scale=0.1]
    \node  [block] (atrr1) {$\square_\lambda(\lambda^+,{<}\lambda^+)$};
    \node  [block] (atrr2) [left of=atrr1] {$\lambda^+$-special};
    \node  [block] (atrr3) [right of=atrr1] {$\lambda^+$-nonspecial};
    \draw[<->,double] (atrr2) to  (atrr1);
    \draw[->,double] (atrr1) to  (atrr3);
\end{tikzpicture}
\end{center}
\fi

The point is that Clause~(1) of that fact does have a canonical proof using walks on ordinals:
\begin{ufact}[Todorcevic, \cite{MR908147}] For every infinite cardinal $\lambda$, the following are equivalent:
\begin{itemize}
\item $\square_\lambda(\lambda^+,{<}\lambda^+)$ holds;
\item There exists a sequence $\vec C$ for which $\mathcal T(\rho_0^{\vec C})$ is a special $\lambda^+$-Aronszajn tree.
\end{itemize}
\end{ufact}

In this paper, we indeed obtain a canonical version of Corollary~\ref{cor05}:
\begin{THMB} For every singular cardinal $\lambda$, if $2^\lambda=\lambda^+$, then the following are equivalent:\footnote{Note that $\lambda$ is not assumed to be a strong-limit.}
\begin{itemize}
\item $\square_\lambda(\lambda^+,{<}\lambda^+)$ holds;
\item There exists a special $\lambda^+$-Aronszajn tree;
\item There exists a special $\lambda^+$-Aronszajn tree whose projection is a nonspecial $\lambda^+$-Aronszajn tree;
\item There exists a sequence $\vec C$ for which $\mathcal T(\rho_0^{\vec C})$ is a special $\lambda^+$-Aronszajn tree,
and its projection, $\mathcal T(\rho_1^{\vec C})$, is a nonspecial $\lambda^+$-Aronszajn tree which is normal but not $\lambda$-distributive.
\end{itemize}
\end{THMB}
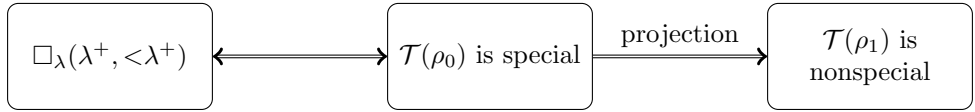
\ifpdf
\begin{center}
\tikzstyle{line} = [draw, -latex']
\tikzstyle{block} = [rectangle, draw, 
    text width=7em, text centered, rounded corners, minimum height=4em]
\begin{tikzpicture}[node distance = 5cm, scale=0.1]
    \node  [block] (atrr1) {$\mathcal T(\rho_0)$ is special};
    \node  [block] (atrr2) [left of=atrr1] {$\square_\lambda(\lambda^+,{<}\lambda^+)$};
    \node  [block] (atrr3) [right of=atrr1] {{$\mathcal T(\rho_1)$ is nonspecial}};
    \draw[<->,double] (atrr2) to  (atrr1);
    \draw[->,double] (atrr1) to node[above] {projection}   (atrr3) ;
\end{tikzpicture}
\end{center}
\fi

We also obtain the following analog of Fact~\ref{fact02}, Clause~(4):\footnote{Compare this, also, with Theorem 1.4 of \cite{rinot20}.}
\begin{THMC} Suppose that $\lambda$ is a strong-limit singular cardinal and $2^\lambda=\lambda^+$.

If $\square(\lambda^+,{<}2)$ holds, then there exists a uniformly coherent $\lambda^+$-Souslin tree.
\end{THMC}

\subsection*{Conventions} Throughout the whole paper,
$\kappa$ stands for an arbitrary regular uncountable cardinal,
$\mu$ is some (possibly finite) cardinal $\le\kappa$,
$\chi$ is some infinite regular cardinal $<\kappa$,
and $\xi$ is some ordinal $\le\kappa$.

\subsection*{Notation}
For infinite cardinals $\theta\le\lambda$,
denote $E^\lambda_\theta:=\{\alpha<\lambda\mid \cf(\alpha)=\theta\}$, and define $E^\lambda_{\neq\theta}$, $E^\lambda_{<\theta}$, $E^\lambda_{>\theta}$, and $E^\lambda_{\ge\theta}$ in a similar fashion.
Write $[\lambda]^\theta$ for the collection of all subsets of $\lambda$ of cardinality $\theta$, and define $[\lambda]^{<\theta}$ similarly.
Write  $\ch_\lambda$ for the assertion that $2^\lambda=\lambda^+$.

Suppose that $C$ and $D$ are sets of ordinals.
Write $\acc(C):=\{\alpha\in C\mid \sup (C\cap\alpha) = \alpha>0 \}$, $\nacc(C) := C \setminus \acc(C)$,
$\acc^+(C) := \{\alpha<\sup(C)\mid \sup (C\cap\alpha) = \alpha>0 \}$, and $\cl(C):=C\cup\acc^+(C)$.
For any  $j < \otp(C)$, denote by $C(j)$ the unique element $\delta\in C$ for which $\otp(C\cap\delta)=j$.
Write $D\sq C$ iff there exists some ordinal $\beta$ such that $D = C \cap \beta$.
Write $D \sq_\chi C$ if either $D \sq C$ or ($\otp(C)< \chi$ and $\nacc(C)$ consists only of successor ordinals).
Write $C=^* D$ iff either $C=D$ or for some ordinal $\beta$, $\{C\setminus\beta,D\setminus\beta\}$ is a singleton distinct from $\{\emptyset\}$.
Let $\reg(\kappa)$ denote the set of all infinite regular cardinals below $\kappa$,
and let $\mathcal K(\kappa):=\{ x\in\mathcal P(\kappa)\mid x\neq\emptyset\ \&\ \acc^+(x)\s x\ \&\ \sup(x) \notin x\}$
denote the collection of all nonempty $x \subseteq \kappa$ such that $x$ is a club subset of $\sup(x)$.

\subsection*{Organization of this paper}
In {Section~\ref{section1}}, we introduce the notions of $\mathcal C$-sequences, $C$-sequences, amenable $C$-sequences, and postprocessing functions.
The $\mathcal C$-sequences are just abstract versions of square sequences,
and $C$-sequences are typically transversals for the former. It is proved that transversals for $\square_\xi(\kappa,{<}\mu)$-sequences are moreover \emph{amenable}, provided that $\min\{\xi,\mu\}<\kappa$.

The collection of postprocessing functions forms a monoid that acts on the class of square sequences.
This means that these functions allow us to move from an arbitrary witness to $\square_\xi(\kappa,{<}\mu)$ to some better witness with additional properties.
For instance, the proof of Theorem~C (which is given in a later section) involves applying a dozen different postprocessing functions in order to obtain a witness good enough for the purpose of constructing  a uniformly coherent Souslin tree.

We believe that the above-mentioned concepts capture fundamental properties of the combinatorics of a given cardinal $\kappa$.
To practice these concepts in a simpler context, we decided to focus the first section on a combinatorial problem of independent interest.
Recall that a subset $F$ of (a regular uncountable cardinal) $\kappa$ is said to be \emph{fat}
if any of the two equivalent conditions hold:\footnote{See \cite[Lemma~1.2]{AbSh:146}.}
\begin{itemize}
\item for every club  $D\s \kappa$ and every $\alpha<\kappa$, there is an increasing and continuous map $\pi:\alpha\rightarrow F\cap D$;
\item for every club $D\s\kappa$ and every $\theta\in\reg(\kappa)$, $F\cap D$ contains a closed copy of $\theta+1$.
\end{itemize}

Clearly, every fat set is stationary.
In \cite{MR0327521}, Friedman proved that a subset of $\omega_1$ is fat iff it is stationary.
In particular, $\omega_1$ may be partitioned into $\omega_1$ many fat sets.
Some 40 years ago (see \cite[Remark~12]{MR521125}), Shelah noticed that in Magidor's model that appeared in \cite[$\S2$]{MR683153} and assumes the consistency of a weakly compact cardinal, $\omega_2$ cannot be partitioned into two fat sets.
In the first section, amenable $C$-sequences and postprocessing functions are used to establish the following.

\begin{THMD} If $\square(\kappa,{<}2)$ holds, then every fat subset of $\kappa$ may be partitioned into $\kappa$ many fat sets.

In particular, the failure to partition $\omega_2$ into two fat sets is equiconsistent with the existence of a weakly compact cardinal.
\end{THMD}

Additional byproducts of this study are an answer to Question~3 from \cite{rinot07}, and a generalization of Theorem~2 from \cite{MR3135494}.

\medskip

In Section~\ref{section2}, we prove the first case of Theorem~A (see Corollary~\ref{first-case-A}).
As said before, the witnessing tree is of the form $\mathcal T(\rho_0^{\vec C})$ for a carefully crafted $C$-sequence $\vec C$.
To obtain $\vec C$, we establish a mixing lemma for postprocessing functions, a postprocessing-function version of a result from \cite{paper24},
and a \emph{wide-club-guessing} lemma for $\square(\kappa,{<}\mu)$-sequences (thus, answering Question~16 from \cite{rinot_s01} in the affirmative).

\medskip

In Section~\ref{section3}, we prove that the \emph{strong-limit} hypothesis  in Fact~\ref{fact02}(2) is surplus.
Of course, this verifies the second case of Theorem~A (see Corollary~\ref{cor39}). More importantly, some of the results leading to it will pave the way for proving Theorem~C.

\medskip

In Section~\ref{section4}, we prove Theorem~C (see Corollary~\ref{thmC}). Remarkably enough, its proof requires almost all the machinery developed in Sections \ref{section1}, \ref{section2}, and~\ref{section3}.

\medskip

Finally, Section~\ref{section5} is a short section, in which we prove Theorem~B (see Corollary~\ref{thm51}).

\medskip

It is worth mentioning that a recurring technical ingredient involved in each of the proofs of Theorems A, B, and C has to do with the study of non-accumulation points of the clubs appearing in a transversal for a square sequence.
In Section~\ref{section3}, the following result is established:

\begin{uthm}\label{main3.2}  For any uncountable cardinal $\lambda$, $\ch_\lambda$ entails that the following are equivalent:
\begin{enumerate}
\item $\square_\lambda(\lambda^+,{<}\lambda^+)$;
\item There exists a $C$-sequence $\langle C_\alpha\mid\alpha<\lambda^+\rangle$ such that:
\begin{itemize}
\item $\otp(C_\alpha)<\lambda$ for all $\alpha\in E^{\lambda^+}_{<\lambda}$;
\item $|\{ C_\alpha\cap\delta\mid \alpha<\lambda^+\}|\le\lambda$ for all $\delta<\lambda^+$;
\item for every sequence $\langle A_i\mid i<\lambda\rangle$ of cofinal subsets of $\lambda^+$, for every $\theta\in\reg(\lambda)$, the following set is stationary:
$$\{\alpha<\lambda^+\mid  \otp(C_\alpha)=\theta\ \&\  \forall i<\theta[C_\alpha(i+1)\in A_i]\}.$$
\end{itemize}
\end{enumerate}
\end{uthm}

In Section~\ref{section5}, the following result is established:

\begin{uthm}\label{thm23} For any singular cardinal $\lambda$, $\ch_\lambda$ entails that the following are equivalent:
\begin{enumerate}
\item $\square_\lambda(\lambda^+,{<}\lambda^+)$;
\item There exists a $C$-sequence $\langle C_\alpha\mid\alpha<\lambda^+\rangle$ such that:
\begin{itemize}
\item $\otp(C_\alpha)\le\lambda$ for all $\alpha<\lambda^+$;
\item $|\{ C_\alpha\cap\delta\mid \alpha<\lambda^+\}|\le\lambda$ for all $\delta<\lambda^+$;
\item for every sequence $\langle A_i\mid i<\lambda\rangle$ of cofinal subsets of $\lambda^+$,
the following set is stationary:
$$\{\alpha<\lambda^+\mid \otp(C_\alpha)=\lambda\ \&\ \forall i<\lambda[C_\alpha(i+1)\in A_i]\}.$$
\end{itemize}
\end{enumerate}
\end{uthm}

In Section~\ref{section4}, the following two results are established:

\begin{uthm}\label{thm3} The following are equivalent:
\begin{enumerate}
\item $\diamondsuit(\omega_1)$ holds;
\item There exists a $C$-sequence $\langle C_\alpha\mid\alpha<\omega_1\rangle$ such that:
\begin{itemize}
\item $|\{ C_\alpha\cap\delta\mid \alpha<\omega_1\ \&\ \sup(C_\alpha\cap\delta)=\delta\}|=1$ for all $\delta<\omega_1$;
\item for every sequence $\langle A_i\mid i<\omega_1\rangle$ of cofinal subsets of $\omega_1$, the following set is stationary:
$$\{\alpha<\omega_1\mid \otp(C_\alpha)=\alpha\ \&\ \forall i<\alpha[C_\alpha(i+1)\in A_i]\}.$$
\end{itemize}
\end{enumerate}
\end{uthm}

\begin{uthm}\label{thm4}
For any singular strong-limit cardinal $\lambda$, $\ch_\lambda$ entails that the following are equivalent:
\begin{enumerate}
\item $\square(\lambda^+,{<}2)$;
\item There exists a $C$-sequence $\langle C_\alpha\mid\alpha<\lambda^+\rangle$ such that:
\begin{itemize}
\item $|\{ C_\alpha\cap\delta\mid \alpha<\lambda^+\ \&\ \sup(C_\alpha\cap\delta)=\delta\}|=1$ for all $\delta<\lambda^+$;
\item for every sequence $\langle A_i\mid i<\lambda^+\rangle$ of cofinal subsets of $\lambda^+$, the following set is stationary:
$$\{\alpha<\lambda^+\mid \forall i<\alpha[\sup(\nacc(C_\alpha)\cap A_i)=\alpha]\}.$$
\end{itemize}
\end{enumerate}
\end{uthm}

In Section~\ref{section2}, the following result is established:

\begin{uthm}\label{thm22}
For any uncountable strong-limit cardinal $\lambda$, $\ch_\lambda$ entails that the following are equivalent:
\begin{enumerate}
\item $\square(\lambda^+,{<}\lambda)$;
\item There exists a $C$-sequence $\langle C_\alpha\mid \alpha<\lambda^+\rangle$
such that:
\begin{itemize}
\item for every club $D\s\lambda^+$, there exists $\delta\in\acc(D)$ such that for all $\alpha<\lambda^+$ either $\sup(C_\alpha\cap\delta)<\delta$ or $\sup(\nacc(C_\alpha\cap\delta)\cap D)=\delta$;
\item $|\{ C_\alpha\cap\delta\mid \alpha<\lambda^+\ \&\ \sup(C_\alpha\cap\delta)=\delta\}|<\lambda$ for all $\delta<\lambda^+$;
\item for every sequence $\langle A_i\mid i<\lambda\rangle$ of cofinal subsets of $\lambda^+$, for every $\theta\in\reg(\lambda)$, the following set is stationary
$$\{\alpha<\lambda^+\mid \forall i<\theta[\sup(\nacc(C_\alpha)\cap A_i)=\alpha]\}.$$
\end{itemize}
\end{enumerate}
\end{uthm}

\section{Partitioning a fat set}\label{section1}
\begin{defn} A \emph{$\mathcal C$-sequence} over $\kappa$ is a sequence $\cvec{C}=\langle\mathcal C_\alpha\mid\alpha<\kappa\rangle$ such that, for all limit $\alpha<\kappa$, $\mathcal C_\alpha$ is a nonempty collection of club subsets of $\alpha$.
It is said to be \emph{$\xi$-bounded} if $\otp(C)\le\xi$ for all $C\in\mathcal C_\alpha$ and $\alpha<\kappa$. Its \emph{support} is defined to be the following set:
\[
\Gamma(\cvec{C}) :=
\{ \alpha \in \acc(\kappa) \mid \forall C \in \mathcal C_\alpha \forall \bar\alpha \in \acc(C)
[C \cap \bar\alpha \in \mathcal C_{\bar\alpha}] \}.
\]
\end{defn}
\begin{defn} A \emph{$C$-sequence} over $\Gamma$ is a sequence $\vec C=\langle C_\alpha\mid\alpha\in\Gamma\rangle$ such that, for all limit ordinals $\alpha\in\Gamma$, $C_\alpha$ is a club subset of $\alpha$.
It is said to be \emph{$\xi$-bounded} if $\otp(C_\alpha)\le\xi$ for all $\alpha\in\Gamma$.
\end{defn}

A key concept of this paper is that of an \emph{amenable $C$-sequence}, which is a strengthening of $\otimes_{\vec C}$ of \cite[p.~134]{Sh:365}.

\begin{defn}\label{amenable} A $C$-sequence $\vec C=\langle C_\alpha\mid\alpha\in\Gamma\rangle$ over a stationary subset $\Gamma\s\kappa$ is said to be \emph{amenable}
if for every club $D\s\kappa$, the set $\{ \alpha\in \Gamma \mid \sup(D\cap\alpha\setminus C_\alpha)<\alpha\}$ is nonstationary in $\kappa$.
\end{defn}
\begin{example}\label{example14} The simplest example of an amenable $C$-sequence is a $\lambda$-bounded $C$-sequence over $\lambda^+$.
Indeed, since $\{\alpha<\kappa\mid \otp(D\cap\alpha)=\omega^\alpha\}$ is a club for any club $D\s\kappa$,
any $C$-sequence $\langle C_\alpha\mid\alpha\in\Gamma\rangle$ over a stationary subset $\Gamma\s\kappa$ for which $\{\alpha\in\Gamma\mid \otp(C_\alpha)=\alpha\}$ is nonstationary, is amenable.
\end{example}
\begin{remark} If $V=L$, then
$\kappa$ carries an amenable $C$-sequence iff $\kappa$ is not ineffable.
In general, every stationary $\Gamma\s\kappa$ admits a stationary subset $\Gamma'\s\Gamma$ that carries an amenable $C$-sequence.
\end{remark}
Let us point out that the concept of Definition~\ref{amenable} is a relative of that of being a  \emph{nontrivial $C$-sequence} in the sense of \cite[Definition 6.3.1]{MR2355670}.
\begin{prop}\label{amenable_vs_trivial} Suppose that $\vec C=\langle C_\alpha\mid\alpha\in\Gamma\rangle$ is a $C$-sequence over a stationary subset $\Gamma\s\kappa$.

Then the following are equivalent:
\begin{enumerate}
\item $\vec C$ is amenable;
\item for every cofinal $A\s\kappa$, the set $\{\alpha\in\Gamma\mid A\cap\alpha\s C_\alpha\}$ is nonstationary.
\end{enumerate}
\end{prop}
\begin{proof} $(1)\implies(2)$: Let $A$ be an arbitrary cofinal subset of $\kappa$. Let $D:=A\cup\acc^+(A)$ denote the closure of $A$.
Then $A \cap \alpha \subseteq C_\alpha \iff D \cap \alpha \subseteq C_\alpha$ for all $\alpha \in \Gamma$, since $C_\alpha$ is club in $\alpha$, so that
$$\{\alpha\in\Gamma\mid A\cap\alpha\s C_\alpha\} = \{\alpha\in\Gamma\mid D\cap\alpha\s C_\alpha\}\s \{ \alpha\in \Gamma \mid \sup(D\cap\alpha\setminus C_\alpha)<\alpha\}.$$

Since $\vec C$ is amenable, the right-hand side of the above inclusion is nonstationary, and then so is the left-hand side.

$\neg(1)\implies\neg(2)$:
Suppose that $D$ is a club in $\kappa$ for which $S:=\{ \alpha\in \Gamma \mid \sup(D\cap\alpha\setminus C_\alpha)<\alpha\}$
is stationary. By Fodor's lemma, pick $\varepsilon<\kappa$ for which $S_\varepsilon:=\{ \alpha\in \Gamma \mid \sup(D\cap\alpha\setminus C_\alpha)=\varepsilon\}$ is stationary.
Then $A:=D\setminus(\varepsilon+1)$ is a cofinal subset of $\kappa$ for which $\{\alpha\in\Gamma\mid A\cap\alpha\s C_\alpha\}$ covers the stationary set $S_\varepsilon$.
\end{proof}

\begin{lemma}\label{lemma12} Suppose that $\vec C=\langle C_\alpha\mid\alpha\in\Gamma\rangle$ is an amenable $C$-sequence over a stationary subset $\Gamma\s\kappa$.
Then for every stationary $\Omega\s\Gamma$, there exists $i<\kappa$ such that $\Omega_{i,\tau}:=\{\beta\in \Omega\mid \otp(C_\beta)>i\ \&\ C_\beta(i)\ge\tau\}$ is stationary for all $\tau<\kappa$.
\end{lemma}
\begin{proof} We follow the proof of \cite[Lemma 3.2]{rinot18}. Suppose not.
Then we can fix a stationary $\Omega \subseteq \Gamma$ and a function $f:\kappa\rightarrow\kappa$ such that $\Omega_{i,{f(i)}}$ is nonstationary for all  $i<\kappa$.
For each $i<\kappa$, let $D_i$ be a club subset of $\kappa\setminus \Omega_{i,{f(i)}}$.
Then $D:=\{\delta\in\diagonal_{i<\kappa}D_i\mid f[\delta]\s\delta\}$ is a club subset of $\kappa$.
Consider the set $S:=\{\beta\in \Omega\mid  \otp(C_\beta)=\beta\}$.

$\br$ Suppose $S$ is stationary. Then $D \cap S$ is a stationary subset of $\Gamma$.
Thus, by amenability of $\vec C$ and Proposition~\ref{amenable_vs_trivial}(2), we may pick some $\beta \in D \cap S$ and $\alpha\in(D\cap\beta)\setminus C_\beta$.
As $\beta\in {S}\cap\diagonal_{i<\kappa}D_i$, we get that $\beta\in {S}\setminus \Omega_{i,{f(i)}}$ for all $i<\beta$.
In particular, $C_\beta(i)<f(i)$ for all $i<\alpha$. Since $\alpha\in D$, we have $f[\alpha]\s\alpha$. Altogether, $C_\beta(i)<\alpha$ for all $i<\alpha$.
Since $\alpha < \beta = \otp(C_\beta)$ and the map $i\mapsto C_\beta(i)$ is increasing and continuous, we then get that $C_\beta(\alpha)=\alpha$, contradicting the fact that $\alpha\notin C_\beta$.

$\br$ Suppose $S$ is nonstationary. Then there are stationarily many $\beta \in \Omega$ such that $\otp(C_\beta) <\beta$, so that by Fodor's lemma,
there exists some $\varepsilon<\kappa$ such that $T_\varepsilon :=\{\beta\in \Omega\cap \diagonal_{i<\kappa}D_i\mid \otp(C_\beta)=\varepsilon\}$ is stationary.
Pick $\beta\in T_\varepsilon$ above $\sup(f[\varepsilon])$.
As $\otp(C_\beta)=\varepsilon$ and $\beta\in \diagonal_{i<\kappa}D_i$, we get that $\beta \notin \Omega_{i,f(i)}$ for all $i<\varepsilon$, and hence $C_\beta(i)<f(i)$ for all $i<\varepsilon$.
So, $\beta=\sup(C_\beta)\le\sup(f[\varepsilon])$, contradicting the choice of $\beta$.
\end{proof}

\begin{defn}
A function $\Phi:\mathcal K(\kappa)\rightarrow\mathcal K(\kappa)$ is a \emph{postprocessing function}  if for every $x\in\mathcal K(\kappa)$:
\begin{itemize}
\item  $\Phi(x)$ is a club in $\sup(x)$;
\item $\acc(\Phi(x)) \s \acc(x)$;
\item $\Phi(x)\cap\bar\alpha=\Phi(x\cap\bar\alpha)$ for every $\bar\alpha\in\acc(\Phi(x))$, meaning that the following diagram commutes:
\[
\xymatrix{
x \ar[rr]^\Phi \ar[d]_{\cap\bar\alpha}  && \Phi(x) \ar[d]^{\cap\bar\alpha}  \\
x \cap \bar\alpha \ar[rr]^\Phi  && *+[F]\txt{$\Phi(x)\cap\bar\alpha$\\${}=\Phi(x\cap\bar\alpha)$}
}
\]
\end{itemize}
\end{defn}
\begin{remark} By the first clause, $\otp(\Phi(x))\ge\cf(\sup(x))$, and by the second clause, $\otp(\Phi(x))\le\otp(x)$.
In particular, if $\otp(x)$ is a regular cardinal, then $\otp(\Phi(x)) = \otp(x)$.
\end{remark}
We say that $\Phi$ is \emph{conservative} provided that $\Phi(x) \subseteq x$ for all $x\in\dom(\Phi)$.
We say that $\Phi$ is \emph{faithful} provided that $(\omega\setminus\{0\})\notin\rng(\Phi)$.
For any function $f$, we say that $\Phi$ is \emph{$f$-preserving} provided that $f(\Phi(x))=f(x)$ for all $x\in\dom(\Phi)$.
For instance, if $\Phi$ is $\acc$-preserving, then it is also $\otp$-preserving.
Note that the composition $\Phi\circ\Phi'$ of two (resp.\ $f$-preserving, conservative) postprocessing functions $\Phi$ and $\Phi'$ is a (resp.\ $f$-preserving, conservative) postprocessing function.
Furthermore, if  $\Phi$ is faithful, then  so is $\Phi \circ \Phi'$.

\begin{example}The simplest example of a postprocessing function is the identity function; it is conservative, $f$-preserving (for any $f$), but not faithful.
\end{example}
\begin{example}\label{faithful_correction} Define $\faithful:\mathcal K(\kappa)\rightarrow\mathcal K(\kappa)$ by stipulating:
$$\faithful(x):=\begin{cases}
x\setminus\{2\},&\text{if }x(1)=2;\\
x,&\text{otherwise}.
\end{cases}$$

Then $\faithful$ is a faithful, conservative, $\acc$-preserving, $\min$-preserving postprocessing function.
\end{example}

\begin{example}\label{chop-bottom}
For an ordinal $j<\kappa$, define $\Phi^{\{j\}} : \mathcal K(\kappa) \to \mathcal K(\kappa)$ by stipulating:
\[\Phi^{\{j\}}(x) := \begin{cases}
x \setminus x(j), &\text{if } \otp(x) > j; \\
x,   &\text{otherwise}.
\end{cases}\]

Then $\Phi^{\{j\}}$ is a conservative postprocessing function. It is $\min$-preserving iff $j=0$, faithful iff $j\in\omega\setminus2$, and $\acc$-preserving iff $j<\omega$.
\end{example}
\begin{example}\label{intersectD} For a given club $D\s\kappa$, define $\Phi:\mathcal K(\kappa)\rightarrow\mathcal K(\kappa)$ by stipulating:
$$\Phi(x):=\begin{cases}
x\cap D,&\text{if }\sup(x\cap D)=\sup(x);\\
x\setminus\sup(x\cap D),&\text{otherwise}.
\end{cases}$$

Then $\Phi$ is a conservative postprocessing function.
\end{example}

\begin{lemma}\label{cons-pp-preserves-amenable} Suppose that $\vec C=\langle C_\alpha\mid\alpha\in\Gamma\rangle$ is a $C$-sequence over a stationary subset $\Gamma\s\acc(\kappa)$.

Then the following are equivalent:
\begin{enumerate}
\item $\vec C$ is amenable;
\item for every postprocessing function $\Phi:\mathcal K(\kappa)\rightarrow\mathcal K(\kappa)$,  $\langle \Phi(C_\alpha)\mid\alpha\in\Gamma \rangle$ is amenable;
\item for every conservative postprocessing function $\Phi:\mathcal K(\kappa)\rightarrow\mathcal K(\kappa)$ and every club $D\s\kappa$,
the set $\{ \alpha\in\Gamma\mid D\cap\alpha=\Phi(C_\alpha)\}$ is nonstationary.
 \end{enumerate}
\end{lemma}
\begin{proof} $\neg(2)\implies\neg(1)$:
Suppose that $\Phi:\mathcal K(\kappa)\rightarrow\mathcal K(\kappa)$ is a postprocessing function
for which $\langle \Phi(C_\alpha)\mid \alpha\in\Gamma\rangle$ is not amenable.
By Proposition~\ref{amenable_vs_trivial}, let us fix a cofinal $A\s\kappa$,
for which the set $S:=\{\alpha\in\Gamma\mid A\cap\alpha\s \Phi(C_\alpha)\}$ is stationary.
Consider the club $D:=\acc^+(A)$.
For each $\alpha\in S\setminus\{0\}$, we have $D\cap\alpha\s\acc(\Phi(C_\alpha))\s\acc(C_\alpha)$, so that $\sup(D\cap\alpha\setminus C_\alpha)=0<\alpha$.
That is $\langle C_\alpha\mid\alpha\in\Gamma\rangle$ is not amenable.

$(2)\implies(3)$: This is an immediate consequence of Proposition~\ref{amenable_vs_trivial}.

$\neg(1)\implies\neg(3)$: Suppose that $\vec C$ is not amenable.
By Proposition~\ref{amenable_vs_trivial}, let us fix a cofinal $A\s\kappa$, for which the set $S:=\{\alpha\in\Gamma\mid A\cap\alpha\s C_\alpha\}$ is stationary.
Consider the conservative postprocessing function given by Example~\ref{intersectD} for the club $D:=\acc^+(A)$.
Put $T:=S\cap\acc(D)$. For each $\alpha\in T$, we have $D\cap\alpha\s\acc(C_\alpha)\s C_\alpha$ and $\sup(D\cap\alpha)=\alpha$, so that $D\cap\alpha=C_\alpha\cap D=\Phi(C_\alpha)$.
Consequently, $\{ \alpha\in\Gamma\mid D\cap\alpha=\Phi(C_\alpha)\}$ covers the stationary set $T$.
\end{proof}

We now arrive at the main lemma of this section:

\begin{lemma}\label{split_amenable}
Suppose that $\vec C=\langle C_\delta\mid\delta\in\Gamma\rangle$ is an amenable $C$-sequence over a stationary subset $\Gamma\s \acc(\kappa)$.
Suppose that $\langle \Omega^\iota\mid\iota<\Lambda\rangle$ is a sequence of stationary subsets of $\Gamma$, with $\Lambda\le\kappa$.

Then there exist a conservative postprocessing function $\Phi:\mathcal K(\kappa)\rightarrow\mathcal K(\kappa)$,
a cofinal subset $B\s\Lambda$, and an injection $h:B\rightarrow\kappa$ such that
$\{ \delta\in\Omega^\iota\mid \min(\Phi(C_\delta))=h(\iota)\}$ is stationary for all $\iota\in B$.
\end{lemma}
\begin{proof}  First, for all $i<\kappa$, define a regressive function $\varphi_{i}:\Gamma\rightarrow\kappa$ by stipulating:
$$\varphi_{i}(\delta):=\begin{cases}
C_\delta(i),&\text{if }\otp(C_\delta)>i;\\
0,&\text{otherwise}.
\end{cases}$$

Now, we consider a few cases.

\underline{Case 1.} Suppose that $\Lambda$ is a successor ordinal, say $\Lambda=\iota+1$. Let $B:=\{\iota\}$, and define $h(\iota)$ by appealing to Fodor's lemma with $\varphi_0\restriction\Omega^\iota$.
That is, we define $h(\iota)$ to ensure that  $\{ \delta\in\Omega^\iota \mid C_\delta(0) = h(\iota) \}$ is stationary.
Then letting $\Phi$ be the identity function does the job.

\underline{Case 2.}  Suppose that $\Lambda$ is a limit ordinal. By passing to a cofinal subset of $\Lambda$, we may assume that $\Lambda$ is an infinite regular cardinal.

Fix an arbitrary $\iota<\Lambda$.
For all $i,\tau<\kappa$, let $\Omega^\iota_{i,\tau}$ denote the corresponding set defined in the statement of Lemma~\ref{lemma12}.
For every $i<\kappa$, $\langle \Omega^\iota_{i,\tau}\mid \tau<\kappa\rangle$ is a $\s$-decreasing sequence,
and hence  $$\zeta^\iota_i:=\{\tau<\kappa\mid\Omega^\iota_{i,\tau}\text{ is stationary in } \kappa \}$$ is an ordinal.

Define $f:\kappa\rightarrow\Lambda+1$ by stipulating:
$$f(i):=\sup\{\iota<\Lambda\mid \zeta_{i}^\iota=\kappa\}.$$

\underline{Case 2.1.} Suppose that there exists some $i^*<\kappa$ such that $f(i^*)=\Lambda$.
Let $B:=\{\iota<\Lambda\mid \zeta_{i^*}^\iota=\kappa\}$, and pick an injection $h:B\rightarrow\kappa$ that satisfies for all $\iota\in B$:
$$\Delta^\iota:=\{\delta\in\Omega^\iota\mid \otp(C_\delta)>i^*\ \&\ C_\delta(i^*)=h(\iota)\}\text{ is stationary}.$$
Let us point out that this is indeed possible. We obtain $h$ by recursion over $\iota\in B$, as follows.
For $\iota=\min(B)$, define $h(\iota)$ by appealing to Fodor's lemma with $\varphi_{i^*}\restriction\Omega^\iota_{i^*,0}$,
and for all nonminimal $\iota\in B$ such that $h\restriction(B\cap\iota)$ has already been defined, define $h(\iota)$ by appealing to Fodor's lemma with
$$\varphi_{i^*}\restriction\Omega^\iota_{i^*,\sup(\im(h\restriction(B\cap\iota)))+1}.$$
Let $\Phi$ be the conservative postprocessing function $\Phi^{\{i^*\}}$ from Example~\ref{chop-bottom}.
Then, for all $\iota\in B$,
$\{ \delta\in\Omega^\iota\mid \min(\Phi(C_\delta))=h(\iota)\}$ covers the stationary set $\Delta^\iota$, so we are done.

\underline{Case 2.2.} Suppose that $\Lambda\notin\im(f)$.
By Lemma~\ref{lemma12}, for every $\iota<\Lambda$, let us pick some $i_\iota<\kappa$ for which $\zeta^\iota_{i_\iota}=\kappa$.
As $\Lambda\notin\im(f)$, the map $\iota\mapsto i_\iota$ is $(<\Lambda)$-to-$1$ over $\Lambda$,
so let us pick some cofinal $A\s\Lambda$ such that $\iota\mapsto i_\iota$ is strictly increasing.
Define $g:\kappa\rightarrow\kappa+1$ by stipulating:
$$g(i):=\sup\{\zeta_{i}^\iota\mid \iota\in A, \iota>f(i)\}.$$

\underline{Case 2.2.1.} Suppose that there exists some $i^*<\kappa$ such that $g(i^*)=\kappa$.
Pick injections $b:\Lambda\rightarrow A$ and $h:\im(b)\rightarrow\kappa$ such that for all $\iota\in\dom(h)$:
$$\Delta^\iota:=\{\delta\in\Omega^\iota\mid \otp(C_\delta)>i^*\ \&\ C_\delta(i^*)=h(\iota)\}\text{ is stationary}.$$
Let us point out that this is indeed possible.
Set $$b(0):=\min\{ \iota\in A \mid \zeta^\iota_{i^*}>0, \iota> f(i^*)\},$$
and then define $h(b(0))$ by appealing to Fodor's lemma with $\varphi_{i^*}\restriction\Omega^{b(0)}_{i^*,0}$.
Next, given a nonzero $\beta<\Lambda$ such that $b\restriction\beta$ and $h\restriction\im(b\restriction\beta)$ has already been defined,
let $$b(\beta):=\min\{ \iota\in A \mid \zeta^\iota_{i^*}>\sup(\im(h\restriction\im(b\restriction\beta))), \iota>\sup(\im(b\restriction\beta)) \},$$
and then define $h(b(\beta))$ by appealing to Fodor's lemma with $$\varphi_{i^*}\restriction\Omega^{b(\beta)}_{i^*,\sup(\im(h\restriction\im(b\restriction \beta)))+1}.$$

Define $\Phi$ as in Case 2.1, and note that $B:=\rng(b)$ and $h$ do the job.

\underline{Case 2.2.2.} Suppose that  $\kappa\notin\im(g)$.
Recalling that $\Lambda\notin\im(f)$, let $B$ be a cofinal subset of $A$ with the property that $\sup\{ f(i_{\iota'})\mid \iota'\in B\cap \iota\}<\iota$ for all $\iota\in B$.
As $\kappa\notin\im(g)$, let us pick an injection $h:B\rightarrow\kappa$ such that for all $\iota\in B$:
\begin{enumerate}
\item $\Delta^\iota:=\{\delta\in\Omega^\iota\mid \otp(C_\delta)>i_\iota\ \&\ C_\delta(i_\iota)=h(\iota)\}$ is stationary;
\item $h(\iota)>g(i_\iota)$.
\end{enumerate}
At this stage, the reader can be probably see that such a function indeed exists.

For all $x\in\mathcal K(\kappa)$, let $I(x):=\{ i<\otp(x)\mid \exists \iota\in B[i=i_\iota\ \&\ x(i_\iota)=h(\iota)]\}$.
Define $\Phi:\mathcal K(\kappa)\rightarrow\mathcal K(\kappa)$ by stipulating:
$$\Phi(x):=\begin{cases}
x\setminus x(\min(I(x))),&\text{if }I(x)\neq\emptyset;\\
x,&\text{otherwise}.
\end{cases}$$

Let $\iota\in B$ be arbitrary.  For all $\iota'\in B\cap\iota$, we have $h(\iota')>g(i_{\iota'})$ and $\iota>f(i_{\iota'})$,
and hence $\zeta^\iota_{i_{\iota'}}<h(\iota')$ so that $\Omega^\iota_{i_{\iota'},h(\iota')}$ is nonstationary.
Let $D_\iota$ be a club disjoint from $\bigcup\{\Omega^\iota_{i_{\iota'},h(\iota')}\mid \iota'\in B\cap\iota\}$.
We claim that $\{ \delta\in\Omega^\iota\mid \min(\Phi(C_\delta))=h(\iota)\}$ covers the stationary set $\Delta^\iota\cap D_\iota$.
To see this, let $\delta\in\Delta^\iota\cap D_\iota$ be arbitrary.
As $\delta\in \Delta^\iota$, we have $\iota\in I(C_\delta)$, so that $\min(C_\delta\setminus C_\delta(i_\iota))=C_\delta(i_\iota)=h(\iota)$.
Towards a contradiction, suppose that $\min(\Phi(C_\delta))\neq h(\iota)$.
Since $\iota\mapsto i_\iota$ is strictly increasing over $B$, this must mean that we may pick $\iota'\in B\cap\iota$.
In particular, $C_\delta(i_{\iota'})=h(\iota')$, so that $\delta\in\Omega^{\iota}_{i_{\iota'},h(\iota')}$, contradicting the fact that $\delta\in D_\iota$.

Thus, we are left with proving the following.
\begin{claim}$\Phi$ is a postprocessing function.
\end{claim}
\begin{proof} Let $x\in\mathcal K(\kappa)$ be arbitrary. As $\Phi(x)$ is a final segment of $x$, we know that $\Phi(x)$ is a club in $\sup(x)$ and $\acc(\Phi(x))\s\acc(x)$.
Evidently, $I(x\cap\bar\alpha)\sq I(x)$.

$\br$ If $I(x)=\emptyset$, then $\Phi(x)=x$ and $I(x\cap\bar\alpha)=\emptyset$, so that $\Phi(x\cap\bar\alpha)=x\cap\bar\alpha=\Phi(x)\cap\bar\alpha$.

$\br$ If $I(x)\neq\emptyset$, then since $\bar\alpha\in\acc(\Phi(x))$ and $I(x\cap\bar\alpha)\sq I(x)$, we infer that $I(x\cap\bar\alpha)\neq\emptyset$ and $\Phi(x\cap\bar\alpha)=\Phi(x)\cap\bar\alpha$.
\end{proof}

It is easy to see that $\Phi(x)$ is a final segment of $x$ for all $x\in\mathcal K(\kappa)$.
In particular, $\Phi$ is conservative. This completes the proof.
\end{proof}

\begin{defn}\label{def115}
The principle $\square_\xi(\kappa,{<}\mu,\mathcal R_0,\mathcal R_1)$ asserts the existence of a $\xi$-bounded $\mathcal C$-sequence over $\kappa$,
$\cvec{C}=\langle\mathcal C_\alpha \mid \alpha < \kappa\rangle$, such that:
\begin{itemize}
\item for every $\alpha < \kappa$, $\left| \mathcal C_\alpha \right| < \mu$ and $|\{ C \in \mathcal C_\alpha\mid \otp(C)=\xi\}|\le1$;
\item for every  $\alpha < \kappa$, every $C \in \mathcal C_\alpha$, and every $\bar\alpha \in \acc(C)$, there exists $D \in \mathcal C_{\bar\alpha}$ such that $D \mathrel{\mathcal R_0} C$;
\item for every cofinal $A\s\kappa$, there exists $\alpha\in\acc^+(A)$ such that $(A\cap\alpha)\mathrel{\mathcal R_1}\mathcal  C_\alpha$;
\item for every $\alpha \in \kappa \setminus \Gamma(\cvec{C})$, $\mathcal C_\alpha$ is a singleton,
say, $\mathcal C_\alpha = \{e_\alpha\}$, with $\otp(e_\alpha) = \cf(\alpha)$;\footnote{In particular, we mean that $\mathcal C_{\alpha+1}=\{\{\alpha\}\}$ for all $\alpha<\kappa$.}
\item $E^\kappa_\omega \subseteq \Gamma(\cvec{C})$.
\end{itemize}
\end{defn}

If we omit $\xi$, then we mean that $\xi=\kappa$.
If we omit $\mu$, then we mean that $\mu = 2$,
and in that case we sometimes say that the principle is witnessed by a corresponding $C$-sequence $\langle C_\alpha \mid \alpha<\kappa \rangle$,
where for every $\alpha<\kappa$, $C_\alpha$ is the unique element of $\mathcal C_\alpha$.
We write $\square_\xi(\kappa,\mu,\mathcal R_0,\mathcal R_1)$ for $\square_\xi(\kappa,{<}\mu^+,\mathcal R_0,\mathcal R_1)$.

We shall sometimes
refer to the \emph{$\mathcal R_0$-coherence} of the sequence, or simply to its \emph{coherence} in the case where $\mathcal R_0 = {\sq}$.
Note that a study of coherence relations weaker than $\sq$ is necessary. For instance, unlike coherent square sequences that are refuted by large cardinals,
$\sq_\chi$-coherent square sequences provide an effective means to obtain optimal incompactness results above large cardinals (cf.~\cite{paper28}).
Nevertheless, on first reading, it will be easier to assume $\chi := \aleph_0$ throughout, in which case $\sq_{\chi}$ coincides with $\sq$,
any $\square_\xi(\kappa,{<}\mu,{\sq_\chi},\mathcal R_1)$-sequence has support $\Gamma = \acc(\kappa)$, and faithfulness of postprocessing functions plays no role.

\begin{example} The binary relations $\mathcal R_1$ used in this paper are  $\notin$, and the always-satisfied relation $V$.\footnote{Another relation, $\nsin$, will be introduced in Section~\ref{section4}.}
\end{example}
\begin{example} The classical axioms $\square_\lambda$ and $\square^*_\lambda$ (cf.~\cite{MR0309729})
correspond to $\square_\lambda(\lambda^+,{<}2, {\sq}, V)$ and $\square_\lambda(\lambda^+,{<}\lambda^+, {\sq}, V)$, respectively.
\end{example}

If we omit $\mathcal R_1$, then we mean that $\mathcal R_1 = {\notin}$.
Notice that $\square_\xi(\kappa,{<}\mu,\mathcal R_0, V)$ is equivalent to $\square_\xi(\kappa,{<}\mu,\mathcal R_0)$ whenever $\xi<\kappa$,
and that $\square_\kappa(\kappa, {<}2, {\sq}, V)$ is a trivial consequence of $\zf$.

If we omit both $\mathcal R_1$ and $\mathcal R_0$, then we mean that $\mathcal R_1={\notin}$ and $\mathcal R_0 = {\sq}$.
In particular,  $\square(\kappa,\mu)$ and $\square(\kappa)$ agree with their classical definitions (cf.~\cite{MR908147}).

\begin{defn} For any $C$-sequence $\vec C=\langle C_\alpha\mid\alpha\in\Gamma\rangle$:
\begin{enumerate}
\item $\vec C$ is said to be a \emph{transversal for $\cvec{C}=\langle\mathcal C_\alpha\mid\alpha<\kappa\rangle$} iff $\vec C\in\prod_{\alpha\in\Gamma}\mathcal C_\alpha$ and $\Gamma=\Gamma(\cvec{C})$;
\item $\vec C$ is said to be a \emph{transversal for $\square_\xi(\kappa,{<}\mu,\mathcal R_0,\mathcal R_1)$} iff it is a transversal for some $\square_\xi(\kappa,{<}\mu,\mathcal R_0,\mathcal R_1)$-sequence.
\end{enumerate}
\end{defn}

The following is obvious.

\begin{prop}\label{transversal-width} Suppose that $\vec C=\langle C_\alpha \mid \alpha \in \Gamma \rangle$ is a $C$-sequence.
For all $\xi,\mu\le\kappa$, we have:
\begin{enumerate}
\item $\vec C$ is a transversal for $\square_\xi(\kappa, {<}\mu, {\sq}, V)$ iff it is $\xi$-bounded, $\Gamma=\acc(\kappa)$,
and for every $\gamma \in \acc(\kappa)$ the set $\{ C_\alpha \cap\gamma \mid \alpha \in \acc(\kappa) \text{ and } \sup(C_\alpha \cap\gamma) = \gamma \}$ has size $<\mu$;
\item If $\vec C$ is a transversal for $\square(\kappa, {<}\kappa, {\sq_\chi}, V)$, then  $| \{ C_\alpha \cap \gamma \mid \alpha \in \Gamma \} | < \kappa$ for all $\gamma < \kappa$.\qed
\end{enumerate}
\end{prop}

\begin{lemma}\label{Gamma-closure} Suppose $\cvec{C} = \langle\mathcal C_\alpha \mid \alpha < \kappa\rangle$ is a $\square_\xi(\kappa,{<}\mu,\mathcal R_0,\mathcal R_1)$-sequence.
Let $\Gamma := \Gamma(\cvec{C})$. Then:
\begin{enumerate}
\item  $\Gamma = \{ \alpha \in \acc(\kappa) \mid \exists C \in \mathcal C_\alpha \forall \bar\alpha \in \acc(C) [C \cap \bar\alpha \in \mathcal C_{\bar\alpha}] \}$;
\item  $\bar\alpha \in \Gamma$, whenever $\bar\alpha \in \acc(C)$ and $C\in\bigcup\{\mathcal C_\alpha\mid \alpha \in \Gamma\}$;
\item  If $\mathcal R_0={\sq_\chi}$, then $E^\kappa_\omega\cup E^\kappa_{\geq\chi}\s\Gamma$, and for every $\theta \in \reg(\kappa)$, $E^\kappa_\theta \cap \Gamma$ is stationary in $\kappa$.
\end{enumerate}
\end{lemma}
\begin{proof}
\begin{enumerate}
\item Fix $\alpha \in \acc(\kappa)$.
If $\alpha \notin \Gamma$ then $\mathcal C_\alpha$ must be a singleton, so that the formulations involving $\forall C \in \mathcal C_\alpha$ and $\exists C \in \mathcal C_\alpha$ are equivalent.
\item Fix arbitrary $\alpha \in \Gamma$, $C \in \mathcal C_\alpha$, and $\bar\alpha \in \acc(C)$. Let $D := C \cap \bar\alpha$.
Since $\alpha \in \Gamma$, we have $D \in \mathcal C_{\bar\alpha}$.
Consider any $\gamma \in \acc(D)$. In particular, $\gamma \in \acc(C)$ and $\gamma < \bar\alpha$, so that using $\alpha \in \Gamma$ and $C \in \mathcal C_\alpha$
we obtain $D \cap \gamma = (C \cap \bar\alpha) \cap \gamma = C \cap \gamma \in \mathcal C_\gamma$.
Thus we have found $D \in \mathcal C_{\bar\alpha}$ such that $D \cap \gamma \in \mathcal C_\gamma$ for all $\gamma \in \acc(D)$,
and the result follows from Clause~(1).
\item We have $E^\kappa_\omega\s\Gamma$ by the last bullet of Definition~\ref{def115}. We have $E^\kappa_{\ge\chi}\s\Gamma$ by $\sq_\chi$-coherence together with the fact that any club in an ordinal of cofinality $\ge\chi$ has order-type $\ge\chi$.

Fix arbitrary $\theta \in \reg(\kappa)$ and club $D \subseteq \kappa$. We must find some $\alpha \in E^\kappa_\theta \cap \Gamma \cap D$.
As both $E^\kappa_\omega$ and $E^\kappa_{\geq\chi}$ are stationary and included in $\Gamma$, we may assume
that $\aleph_0 < \theta < \chi$.
Fix some $\gamma \in E^\kappa_{\geq\chi} \cap \acc(D)$. Pick some $C \in \mathcal C_\gamma$.
Since $\cf(\gamma) \geq\chi > \aleph_0$, it follows that $C \cap D$ is a club in $\gamma$, so that $\otp(C \cap D) \geq \cf(\gamma) \geq \chi > \theta$, and we can let $\alpha := (C \cap D)(\theta)$, so that $\cf(\alpha) = \theta$.
Since $\gamma \in E^\kappa_{\geq\chi} \subseteq \Gamma$ and $\gamma \in \acc(C)$, Clause~(2) gives $\alpha \in \Gamma$.
Altogether, $\alpha \in E^\kappa_\theta \cap \Gamma \cap D$, as required. \qedhere
\end{enumerate}
\end{proof}

\begin{prop}\label{narrow-transversal-coherent}
If $\langle C_\alpha \mid \alpha\in\Gamma \rangle$ is a transversal for $\square_\xi(\kappa,{<}2,\mathcal R_0,\mathcal R_1)$,
then $C_\alpha \cap\bar\alpha = C_{\bar\alpha}$ for every $\alpha\in\Gamma$ and every $\bar\alpha \in \acc(C_\alpha)$.

If, moreover, $\Phi$ is a postprocessing function, then $\Phi(C_\alpha) \cap\bar\alpha = \Phi(C_{\bar\alpha})$ for every $\alpha\in\Gamma$ and every $\bar\alpha \in \acc(\Phi(C_\alpha))$.\qed
\end{prop}

The next lemma is implicit in \cite[Lemma 2.4]{MR3730566} and in \cite[Lemma 4.2]{MR3600760}.
Note that it follows from a footnote on \cite[p.~180]{MR2355670} that the hypothesis ``$\min\{\xi,\mu\}<\kappa$'' cannot be waived.

\begin{lemma}\label{square_is_amenable} Any transversal for  $\square_\xi(\kappa,{<}\mu,{\sq_\chi})$ is amenable, provided that $\min\{\xi,\mu\}<\kappa$.
\end{lemma}
\begin{proof} Suppose that $\langle \mathcal C_\alpha\mid\alpha<\kappa\rangle$ is a $\square_\xi(\kappa,{<}\mu,{\sq_\chi})$-sequence, and $\vec C=\langle C_\alpha\mid\alpha\in\Gamma\rangle$ is a corresponding transversal.
By Lemma~\ref{Gamma-closure}(3), $\Gamma$ is stationary in $\kappa$.
By Proposition~\ref{amenable_vs_trivial}, it suffices to show that for every cofinal $A \subseteq \kappa$, the set $\{\alpha \in \Gamma \mid A \cap \alpha \subseteq C_\alpha \}$ is nonstationary.

If $\xi<\kappa$, then for every cofinal $A\s\kappa$, the set $\{\alpha \in \Gamma \mid A \cap \alpha \subseteq C_\alpha \}$ is bounded by $A(\xi+1)$, and we are done.

Next, suppose that $\xi=\kappa$, so that $\mu<\kappa$.
Towards a contradiction, let us fix a cofinal set $A$ in $\kappa$ for which $S:=\{ \delta\in\Gamma\mid A\cap\delta\s C_\delta\}$ is stationary.
Let $\{\beta_\alpha\mid \alpha<\kappa\}$ denote the increasing enumeration of $(\{0\}\cup\acc^+(A))$.
For all $\alpha<\kappa$, put:
$$T_\alpha:=\{ C_\delta\cap\beta_\alpha\mid \delta\in S, \beta_\alpha<\delta\}.$$

\begin{claim}\label{claim1221} $\mathcal T:=(\bigcup_{\alpha<\kappa}T_\alpha,{\sq})$ is a tree whose $\alpha^{\text{th}}$ level is $T_\alpha$, and $|T_\alpha|<\mu$ for all $\alpha <\kappa$.
\end{claim}
\begin{proof} We commence by pointing out that $T_\alpha\s  \mathcal C_{\beta_\alpha}$ for all $\alpha<\kappa$.
Clearly, $T_0=\{\emptyset\}=\mathcal C_0=\mathcal C_{\beta_0}$.
Thus, consider an arbitrary nonzero $\alpha<\kappa$ along with some $t \in T_\alpha$.
Fix $\delta \in S$ above $\beta_\alpha$ such that $t = C_\delta \cap \beta_\alpha$.
Then $\delta\in\Gamma$ and  $\beta_\alpha \in \acc^+(A) \cap \delta \s\acc(C_\delta)$, so that $C_\delta\cap\beta_\alpha\in \mathcal C_{\beta_\alpha}$. That is, $t\in\mathcal C_{\beta_\alpha}$.

This shows that $|T_\alpha|<\mu$ for all $\alpha<\kappa$.
In addition, this shows that for all $t\in \bigcup_{\alpha<\kappa}T_{\alpha}$: $$t\in T_{\alpha}\text{ iff }\sup(t)=\beta_{\alpha}.$$

Next, consider arbitrary $\alpha<\kappa$ and $t \in T_\alpha$, and let $\pred{t} := \{ s \in \bigcup_{\alpha'<\kappa}T_{\alpha'}\mid s \sq t, s \neq t \}$ be the set of predecessors of $t$ in $\mathcal T$.
Fix $\delta \in S$ above $\beta_\alpha$ such that $t = C_\delta \cap \beta_\alpha$.
We claim that $\pred{t} = \{ C_\delta \cap \beta_{\alpha'} \mid \alpha' < \alpha \}$, from which it follows that $(\pred{t},{\sq})\cong(\alpha,{\in})$.

Consider $\alpha'<\alpha$.
Then $\beta_{\alpha'} < \beta_\alpha < \delta$, so that $s := C_\delta \cap \beta_{\alpha'}$ is in $T_{\alpha'}$, and it is clear that $s$ is a proper initial segment of $t$. That is, $s \in \pred{t}$.

Conversely, consider $s \in \pred{t}$.
Fix $\alpha'<\kappa$ such that $s\in T_{\alpha'}$.
By our earlier observation, $\sup(s) = \beta_{\alpha'}$, so that since $s \sq t$, $s \neq t$, and $\sup(t) = \beta_\alpha$, we must have $\beta_{\alpha'} < \beta_\alpha$, and therefore $\alpha'<\alpha$.
Thus, $s = t\cap\beta_{\alpha'} = (C_\delta\cap\beta_\alpha) \cap \beta_{\alpha'} = C_\delta\cap \beta_{\alpha'}$, as required.
\end{proof}
By the preceding claim, $\mathcal T$ is a tree of height $\kappa$ and width $<\mu$, and so by a lemma of Kurepa (see \cite[Proposition 7.9]{MR1994835}),
it admits a cofinal branch.\footnote{A subset $B$ of a $\kappa$-tree $(T,{<_T})$ is a \emph{cofinal branch} if it is linearly ordered by $<_T$, and $\{\alpha\mid B\cap T_\alpha\neq\emptyset\}=\kappa$.}
Let $B$ be a cofinal branch through $\mathcal T$, so that $C := \bigcup B$ is a club in $\kappa$.
As $\langle \mathcal C_\alpha\mid\alpha<\kappa\rangle$ is a $\square(\kappa,{<}\mu,{\sq_\chi},{\notin})$-sequence, let us pick $\beta\in\acc^+(C)$ such that $C\cap\beta\notin\mathcal C_\beta$.
By definition of $C$, we may pick some $t \in B$ such that $C\cap\beta\sq t$.
Then by definition of $B$, we may pick some $\delta\in S$ above $\sup(t)$ such that $t \sq C_\delta$.
Thus $C \cap \beta \sq C_\delta$.
But $\delta\in\Gamma$ and $\beta \in \acc(C) \cap \sup(t) \subseteq \acc(t) \subseteq \acc(C_\delta)$,  and hence $C\cap\beta=C_\delta\cap\beta\in\mathcal C_\beta$. This is a contradiction.
\end{proof}

With the tools developed up to this point, we can now prove Theorem~D:

\begin{thm}\label{fatsplit} Suppose that $\square(\kappa,{\sq_\chi})$ holds.

For every fat subset $F\s \kappa$, there exists a partition $\langle F_j \mid j<\kappa\rangle$ of $F$ such that:
\begin{itemize}
\item $F_j$ is fat for all $j<\kappa$;
\item For every $J\s\kappa$, there exists no $\delta\in E^\kappa_{\ge\chi}$ of uncountable cofinality such that $(\bigcup_{j\in J}F_j)\cap\delta$ and $(\bigcup_{j\in\kappa\setminus J}F_j)\cap\delta$ are both stationary in $\delta$.
\end{itemize}
\end{thm}
\begin{proof} By Friedman's theorem \cite{MR0327521}, we may assume that $\kappa\ge\aleph_2$.\footnote{Of course, the second bullet holds trivially for $\kappa=\omega_1$.}
Let $\Theta$ be some cofinal subset of $\reg(\kappa)\setminus\max\{\aleph_1,\chi\}$.
Let $F$ be an arbitrary fat subset of $\kappa$.

\begin{claim} For every $\theta\in\Theta$, the following set is stationary $$\Omega_\theta:=\{\alpha\in F\cap E^\kappa_\theta\mid \exists d_\alpha\s F\cap\alpha~[\acc^+(d_\alpha)\s d_\alpha\ \&\ \sup(d_\alpha)=\alpha]\}.$$
\end{claim}
\begin{proof} Let $\theta\in\Theta$ be arbitrary. Let $D$ be an arbitrary club.
As $F$ is fat, let $\pi:\theta+1\rightarrow F\cap D$ be strictly increasing and continuous.
Then $\pi(\theta)\in \Omega_\theta$, as witnessed by $d_\alpha:=\pi[\theta]$.
As $\pi(\theta)\in D$, we have demonstrated that $\Omega_\theta\cap D\neq\emptyset$.
\end{proof}

Let $\vec C=\langle C_\alpha\mid\alpha\in\Gamma\rangle$ be a transversal for $\square(\kappa,{<}\mu,{\sq_\chi})$ with $\mu=2$.
By Lemma~\ref{Gamma-closure}(3), $\Gamma$ is a stationary set covering $E^\kappa_{\geq\chi}$, and
by Lemma~\ref{square_is_amenable} (since $\mu<\kappa$), $\vec C$ is amenable, so that by Lemma~\ref{split_amenable},
there exists a conservative postprocessing function $\Phi:\mathcal K(\kappa)\rightarrow\mathcal K(\kappa)$,
a cofinal subset $\Theta'\s\Theta$, and an injection $h:\Theta'\rightarrow\kappa$ such that for all $\theta\in\Theta'$, $S_\theta:=\{ \alpha\in \Omega_\theta\mid \min(\Phi(C_\alpha))=h(\theta)\}$ is stationary.
For notational simplicity, suppose that $\Theta'=\Theta$.

Denote $C_\alpha^\circ:=\Phi(C_\alpha)$.
By Lemma~\ref{cons-pp-preserves-amenable}, $\langle C_\alpha^\circ\mid \alpha\in\Gamma\rangle$ is an amenable $C$-sequence.

Let $\theta\in\Theta$ be arbitrary.
By Lemma~\ref{lemma12}, let us pick  $i_\theta<\kappa$ such that $S_{\theta,\tau}:=\{\alpha\in S_\theta\mid \otp(C^\circ_\alpha)>i_\theta\allowbreak\ \&\ C^\circ_\alpha(i_\theta)\ge\tau\}$ is stationary for all $\tau<\kappa$.
Denote $S_\theta^\tau:=\{\alpha\in S_\theta\mid \otp(C^\circ_\alpha)>i_\theta\allowbreak\ \&\ C^\circ_\alpha(i_\theta)=\tau\}$.
Pick a strictly increasing function $f_\theta:\kappa\rightarrow\kappa$ such that $S^{f_\theta(j)}_\theta$ is stationary for all $j<\kappa$, and put:
\[
F^j_\theta:=\{\alpha\in\Gamma\mid \min(C_\alpha^\circ)=h(\theta), \otp(C_\alpha^\circ)>i_\theta, C_\alpha^\circ(i_\theta)=f_\theta(j)\}.
\]

Then for every $j<\kappa$, put:
\begin{itemize}
\item $F^j:=\bigcup_{\theta\in\Theta}F_\theta^j$,
\item $G^j:=\{\alpha\in \kappa\setminus\bigcup_{i<\kappa}F^i\mid (\alpha\notin\Gamma\ \&\ \alpha=j)\text{ or }(\alpha\in\Gamma\ \&\ \min(C^\circ_\alpha)=j)\}$, and
\item $F_j:=( F^j\cup G^j)\cap F$.
\end{itemize}

\begin{claim}\label{c1222} If  $j_0,j_1$ are two distinct elements of $\kappa$, then $F^{j_0}\cap F^{j_1} = F_{j_0}\cap F_{j_1} = \emptyset$.
\end{claim}
\begin{proof} It is clear that $G^{j_0}\cap F^{j_1}=F^{j_0}\cap G^{j_1}=G^{j_0}\cap G^{j_1}=\emptyset$.

Towards a contradiction, suppose that $\alpha\in F^{j_0}\cap F^{j_1}$.
For each $n<2$, pick $\theta_n$ such that $\alpha\in F^{j_n}_{\theta_n}$.
Then $h(\theta_0)=\min(C_\alpha^\circ)=h(\theta_1)$. But $h$ is injective, and hence $\theta_0=\theta_1$, say it is some $\theta$.
Then $f_\theta(j_0)=C_\alpha^\circ(i_\theta)=f_\theta(j_1)$, contradicting the fact that $f_\theta$ is injective.
\end{proof}

\begin{claim} For all $J\s\kappa$ and $\delta\in \Gamma\cap E^\kappa_{>\omega}$, if $(\bigcup_{j\in J}F_j)\cap\delta$ is stationary in $\delta$, then $(\bigcup_{j\in\kappa\setminus J}F_j)\cap\delta$ is nonstationary in $\delta$.
\end{claim}
\begin{proof} Suppose not, and let $\delta$ and $J$ be a counterexample.
Then $S_0:=(\bigcup_{j\in J}F_j)\cap\acc(C^\circ_\delta)$ and $S_1:=(\bigcup_{j\in\kappa\setminus J}F_j)\cap\acc(C^\circ_\delta)$ are both stationary in $\delta$.
Put $\eta:=\min(C^\circ_\delta)$. By switching $J$ with its complement, we may assume that $\eta\in \kappa\setminus J$.

Pick ${\alpha_0}\in S_0$ and ${\alpha_1}\in S_1$ above ${\alpha_0}$.
By $\{\alpha_0,\alpha_1\}\s \acc(C_\delta^\circ) \subseteq \acc(C_\delta)$ and $\mu=2$,
Proposition~\ref{narrow-transversal-coherent} gives $C_{\alpha_0}^\circ\sq C_{\alpha_1}^\circ\sq C_\delta^\circ$, and Lemma~\ref{Gamma-closure}(2) gives $\{\alpha_0,\alpha_1\}\s\Gamma$.
Let $j^*\in J$ be such that ${\alpha_0}\in F_{j^*}$.
As $\min(C_{\alpha_0}^\circ)=\eta\notin J$, we infer that $\alpha_0\notin G^{j^*}$, so that $\alpha_0\in F^{j^*}$.
Fix $\theta\in\Theta$ such that $\alpha_0\in F^{j^*}_\theta$.
By $C_{{\alpha_0}}^\circ\sq C_{{\alpha_1}}^\circ$ and the definition of $F^{j^*}_\theta$, we also have $\alpha_1\in F^{j^*}_\theta\s F^{j^*}$.
It follows from $j^* \in J$ and Claim~\ref{c1222} that $\alpha_1 \notin \bigcup_{j\in\kappa\setminus J}(F^j)$.
As $\alpha_1\in S_1\s \bigcup_{j\in\kappa\setminus J}(F^j\cup G^j)$, we conclude that $\alpha_1\in\bigcup_{j\in\kappa\setminus J}G^j$.
However, $G^j$ is disjoint from $F^{j^*}$ for all $j<\kappa$, and hence $\alpha_1$ cannot be an element of both. This is a contradiction.
\end{proof}

Thus, we are left with proving the following:
\begin{claim} $F_j$ is fat for all $j<\kappa$.
\end{claim}
\begin{proof} Fix arbitrary $j<\kappa$, arbitrary club $D\s\kappa$ and an arbitrary regular cardinal $\theta<\kappa$.
We shall show that $F_j\cap D$ contains a closed copy of $\theta+1$.

By increasing $\theta$, we may assume that $\theta\in\Theta$.
Since $S_\theta^{f_\theta(j)}$ is stationary, let us pick $\alpha\in\acc(D)\cap S_\theta^{f_\theta(j)} \subseteq F^j_\theta \cap \Omega_\theta$.
Let $d_\alpha\s F\cap\alpha$ be a club in $\alpha$ witnessing that $\alpha\in \Omega_\theta$.
By passing to a subclub, we may assume that $\otp(d_\alpha) = \theta$.
Put $\epsilon:=C_\alpha^\circ(i_\theta)$.
Since $\cf(\alpha)=\theta>\aleph_0$, $d := d_\alpha\cap D\cap\acc(C_\alpha^\circ\setminus\epsilon)$ is a club in $\alpha$ of order-type $\theta$.
By Proposition~\ref{narrow-transversal-coherent}, we have $C_{\bar\alpha}^\circ \sq C_\alpha^\circ$ for every $\bar\alpha \in \acc(C_\alpha^\circ)$.
Since $\alpha \in F^j_\theta$ and $d \subseteq \acc(C_\alpha^\circ \setminus \epsilon)$, it follows that $d \subseteq F^j_\theta$.
Hence $d \cup\{\alpha\}$ is a closed copy of $\theta+1$ in $F^j_\theta\cap F\cap D$, let alone in $F_j\cap D$.
\end{proof}

Then $\langle F_j\mid j<\kappa\rangle$ is the desired partition of $F$.
\end{proof}

\begin{cor}\label{c122} The following are equiconsistent:
\begin{enumerate}
\item $\omega_2$ cannot be partitioned into $\omega_2$ many fat sets;
\item $\omega_2$ cannot be partitioned into two  fat sets;
\item There exists a weakly compact cardinal.
\end{enumerate}
\end{cor}
\begin{proof} Clause~(2) logically implies Clause~(1).

By Corollary~\ref{fatsplit}, if $\omega_2$ cannot be partitioned into $\omega_2$ many fat sets, then $\square(\omega_2)$ fails, and then by \cite{MR908147}, $\omega_2$ is a weakly compact cardinal in $L$.
That is, Clause~(1) implies the consistency of Clause~(3).

By~\cite[$\S2$]{MR683153}, the existence of a weakly compact cardinal entails the consistency of the following statement.
For every stationary $S\s E^{\omega_2}_\omega$, the set $\{\delta\in E^{\omega_2}_{\omega_1}\mid S\cap\delta\text{ is stationary in }\delta\}$ contains a club relative to $E^{\omega_2}_{\omega_1}$.
Thus, we assume that the statement holds, in order to derive Clause~(2).
Let $F$ and $G$ be two arbitrary fat subsets of $\omega_2$.
Since $F$ is fat, the set $S:=F \cap E^{\omega_2}_\omega$ is stationary.
Thus, by our assumption, we can choose a club $D \subseteq \omega_2$ such that $S\cap\delta$ is stationary in $\delta$ for all $\delta\in D\cap E^{\omega_2}_{\omega_1}$.
Since $G$ is fat, we can find an increasing and continuous function $\pi:\omega_1+1\rightarrow G\cap D$.
Put $\delta:=\pi(\omega_1)$. As $\pi[\omega_1]$ is a club in $\delta$ and $S\cap\delta$ is stationary in $\delta$, we get that $S\cap G\cap\delta$ is nonempty.
In particular, $F$ and $G$ are not disjoint. Thus, we have shown that Clause~(3) implies the consistency of Clause~(2).
\end{proof}

In \cite{rinot07}, the second author introduced the following reflection principle in connection with the study of the validity of $\diamondsuit$ at successors of singular cardinals.
\begin{defn}[\cite{rinot07}]  For regular uncountable cardinals $\mu<\kappa$, $R_2(\kappa,\mu)$ asserts that for every function $f:E^\kappa_{<\mu}\rightarrow\mu$,
there exists some $j<\mu$ such that $\{ \delta\in E^\kappa_\mu\mid f^{-1}[j]\cap\delta\text{ is nonstationary in }\delta\}$ is nonstationary in $\kappa$.
\end{defn}

\begin{prop}\label{prop124} For regular uncountable cardinals $\mu<\kappa$, if $\kappa$ may be partitioned into $\mu$ many fat sets, then $R_2(\kappa,\mu)$ fails.
\end{prop}
\begin{proof} Suppose that $\langle F_j\mid j<\mu\rangle$ is a partition of $\kappa$ into fat sets.
Let $f:E^\kappa_{<\mu}\rightarrow\mu$ be the unique function satisfying $\alpha\in F_{f(\alpha)}$ for every $\alpha \in E^\kappa_{<\mu}$.
Let $j<\mu$ be arbitrary.
Then, for every club $D\s\kappa$, there exists a strictly increasing and continuous map $\pi:\mu+1\rightarrow F_j\cap D$.
In particular, $\{\delta\in E^\kappa_\mu\mid F_j\cap\delta\text{ contains a club in }\delta\}$ is a stationary set which is covered by
$\{ \delta\in E^\kappa_\mu\mid f^{-1}[j]\cap\delta\allowbreak\ \text{is nonstationary in }\delta\}$, since $f^{-1}[j] \subseteq \bigcup_{i<j} F_i$ and the latter is disjoint from $F_j$.
Therefore, $f$ witnesses the failure of $R_2(\kappa,\mu)$.
\end{proof}

In \cite[Question~3]{rinot07}, the author asks about the consistency strength of $R_2(\omega_2,\omega_1)$. Here we provide an answer:

\begin{cor} The following are equiconsistent:
\begin{enumerate}
\item $R_2(\omega_2,\omega_1)$ holds;
\item There exists a weakly compact cardinal.
\end{enumerate}
\end{cor}
\begin{proof} It is trivial to see that $R_2(\omega_2,\omega_1)$ holds in Magidor's model \cite[$\S2$]{MR683153}.
In particular, Clause~(2) implies the consistency of Clause~(1).

Next, by Proposition~\ref{prop124}, if $R_2(\omega_2,\omega_1)$ holds, then there exists no partition of $\omega_2$ into $\omega_1$ many fat sets, so that by Corollary~\ref{c122}, Clause~(1) implies the consistency of Clause~(2).
\end{proof}

To conclude this section, in light of the results of \cite[$\S2$]{MR3135494}, let us point out that the proof of Lemma~\ref{split_amenable} (including Lemma~\ref{lemma12} on which it builds) easily generalizes to yield the following:
\begin{lemma}\label{lemma127} Suppose that $\mathcal F$ is a normal filter over some $\Gamma\s \kappa$, and $\vec C=\langle C_\delta\mid \delta\in\Gamma\rangle$ is a $C$-sequence.
Suppose that $\Lambda\le\kappa$, and $\langle \Omega^\iota\mid\iota<\Lambda\rangle$ is a sequence of $\mathcal F$-positive sets such that $\vec C\restriction \Omega^\iota$ is amenable for each $\iota<\Lambda$.
Then there exists a conservative postprocessing function $\Phi:\mathcal K(\kappa)\rightarrow\mathcal K(\kappa)$, a cofinal subset $B\s\Lambda$, and an injection $h:B\rightarrow\kappa$ such that
$\{ \delta\in\Omega^\iota\cap\acc(\kappa)\mid \min(\Phi(C_\delta))=h(\iota)\}$ is $\mathcal F$-positive for all $\iota\in B$.\qed
\end{lemma}
Thus, we get the following generalization of Theorem~2 of \cite{MR3135494}:
\begin{cor}\label{cor-Stanley}  Suppose that $\mathcal F$ is a normal filter over $\kappa$, and $\square_\xi(\kappa)$ holds.
Then there exists a $\square_\xi(\kappa)$-sequence $\langle C_\alpha\mid\alpha<\kappa\rangle$
and a partition of $\kappa$ into $\kappa$ many $\mathcal F$-positive sets, $\langle B_i\mid i<\kappa\rangle$, such that $\acc(C_\alpha)\s B_i$ for all $i<\kappa$ and all $\alpha\in B_i$.
\end{cor}
\begin{proof} Let $\vec C=\langle C_\delta\mid\delta<\kappa\rangle$ be some witness to $\square_\xi(\kappa)$.
Appeal to Lemma~\ref{lemma127} with $\mathcal F$, $\vec C$, and the constant $\kappa$-sequence whose unique element is $\kappa$, to obtain a postprocessing function $\Phi:\mathcal K(\kappa)\rightarrow\mathcal K(\kappa)$,
a cofinal subset $B\s\kappa$, and an injection $h:B\rightarrow\kappa$ such that $\{\delta\in\acc(\kappa)\mid\min(\Phi(C_\delta))=h(\iota)\}$ is $\mathcal F$-positive for all $\iota\in B$.
Put $C_0^\bullet:=\emptyset$, $C_{\delta+1}^\bullet:=\{\delta\}$ for all $\delta<\kappa$, and $C_\delta^\bullet:=\Phi(C_\delta)$ for all $\delta\in\acc(\kappa)$.

\begin{claim} $\vec{C^\bullet}=\langle C^\bullet_\delta\mid\delta<\kappa\rangle$ is a $\square_\xi(\kappa)$-sequence.
\end{claim}
\begin{proof} First, the required coherence comes from Proposition~\ref{narrow-transversal-coherent}.

Next, towards a contradiction, suppose that  $A$ is a cofinal subset of $\kappa$ such that  $A\cap\delta=C_\delta^\bullet$ for every $\delta\in\acc^+(A)$.
Then for all $\alpha<\beta$ both from $\acc^+(A)$, we have $\alpha\in\acc(C_\beta^\bullet)\s \acc(C_\beta)$.
That is, $\langle C_\alpha\mid \alpha\in\acc^+(A)\rangle$ is an $\sq$-increasing chain converging to some club $C$ satisfying $C\cap\delta=C_\delta$ for all $\delta\in\acc(C)$, contradicting the choice of $\vec C$.
\end{proof}

For all $i<\kappa$, set $B_i:=\{\delta<\kappa\mid(\delta=i=0)\text{ or }(\otp(\min(C^\bullet_\delta)\cap \rng(h))=i)\}$.
Then $\vec{C^\bullet}$ and $\langle B_i\mid i<\kappa\rangle$ are as sought.
\end{proof}

\section{$C$-sequences of unrestricted order-type}\label{section2}

By waiving the conservativity requirement of Lemma~\ref{split_amenable}, we arrive at the following extremely useful lemma.
\begin{lemma}[mixing lemma]\label{split_amenable2} Suppose that $\vec C=\langle C_\delta\mid\delta\in\Gamma\rangle$ is an amenable $C$-sequence over a stationary subset $\Gamma\s \acc(\kappa)$.
Suppose that $\langle T_\theta\mid\theta\in\Theta\rangle$ is a sequence of stationary subsets of $\Gamma$, with $\Theta\s \kappa$.
Then there exists a faithful postprocessing function $\Phi:\mathcal K(\kappa)\rightarrow\mathcal K(\kappa)$, satisfying the two:
\begin{itemize}
\item For all $x\in\mathcal K(\kappa)$,  we have $\Phi(x)=^* x$;
\item For cofinally many $\theta\in\Theta$, $\{\delta\in T_\theta\mid \min(\Phi(C_\delta))=\theta\}$ is stationary.
\end{itemize}
\end{lemma}
\begin{proof} Set $\Lambda:=\otp(\Theta)$ and let $\pi:\Theta\leftrightarrow\Lambda$ denote the order-preserving bijection. For all $\iota<\Lambda$, denote $\Omega^\iota:=T_{\pi^{-1}(\iota)}$.
By Lemma~\ref{split_amenable}, let us fix a postprocessing function $\Phi:\mathcal K(\kappa)\rightarrow\mathcal K(\kappa)$,
a cofinal subset $B\s\Lambda$, and an injection $h:B\rightarrow\kappa$ such that $\{ \delta\in\Omega^\iota\mid \min(\Phi(C_\delta))=h(\iota)\}$ is stationary for all $\iota\in B$.
As made clear by the proof of that lemma, we also have $\Phi(x)=^* x$ for all $x\in\mathcal K(\kappa)$.

Define $\Phi':\mathcal K(\kappa)\rightarrow\mathcal K(\kappa)$ by stipulating:
$$\Phi'(x):=\begin{cases}
\Phi(x),&\text{if }\min(\Phi(x))\notin\rng(h);\\
\Phi(x),&\text{if }\min(\Phi(x))=h(\pi(\theta))\text{ but }\theta\ge\sup(x);\\
\{\theta\}\cup(\Phi(x)\setminus \theta),&\text{if }\min(\Phi(x))=h(\pi(\theta))\text{ and }\theta < \sup(x).
\end{cases}$$

To see that $\Phi'$ is a postprocessing function, fix $x \in \mathcal K(\kappa)$.
Evidently, $\Phi'(x)$ is a club in $\sup(\Phi(x)) = \sup(x)$, and $\acc(\Phi'(x)) \subseteq \acc(\Phi(x)) \subseteq \acc(x)$.
Next, suppose that $\bar\alpha\in\acc(\Phi'(x))$.
\begin{itemize}
\item[$\br$] If $\min(\Phi(x))$ is not in $\rng(h)$, then $\min(\Phi(x\cap\bar\alpha))=\min(\Phi(x)\cap\bar\alpha)=\min(\Phi(x))$ is not in $\rng(h)$,
and hence $\Phi'(x\cap\bar\alpha)=\Phi(x\cap\bar\alpha)=\Phi(x)\cap\bar\alpha=\Phi'(x)\cap\bar\alpha$.
\item[$\br$] If $\min(\Phi(x))=h(\pi(\theta))$ for some $\theta\in\Theta$, but $\theta\ge\sup(x)$,
then $\min(\Phi(x\cap\bar\alpha))=\min(\Phi(x)\cap\bar\alpha)=\min(\Phi(x))=h(\pi(\theta))$ and $\theta\ge\sup(x)>\sup(x\cap\bar\alpha)$.
Consequently, $\Phi'(x\cap\bar\alpha)=\Phi(x\cap\bar\alpha)=\Phi(x)\cap\bar\alpha=\Phi'(x)\cap\bar\alpha$.
\item[$\br$] If $\min(\Phi(x))=h(\pi(\theta))$ for some $\theta\in\Theta$ and $\theta<\sup(x)$, then $\min(\Phi'(x))=\theta$, and hence $\sup(x\cap\bar\alpha)=\bar\alpha>\theta$.
As $\min(\Phi(x\cap\bar\alpha))=h(\pi(\theta))$, we altogether have $\Phi'(x\cap\bar\alpha)=\Phi'(x)\cap\bar\alpha$.
\end{itemize}

Let $\faithful$ be given by Example~\ref{faithful_correction}.
Clearly, $(\faithful\circ\Phi')(x)=^* x$ for all $x\in\mathcal K(\kappa)$.

Finally, for every $\iota \in B$ and $\theta$ such that $\pi(\theta)=\iota$, we have that
$$\{\delta\in T_\theta\mid \min((\faithful\circ\Phi')(C_\delta))=\theta\}=\{\delta\in T_\theta\mid \min(\Phi'(C_\delta))=\theta\}$$
covers the stationary set
$\{\delta\in \Omega^\iota\mid \delta>\theta\allowbreak\ \&\ \min(\Phi(C_\delta))=h(\iota)\}$.
As $\sup(B)=\Lambda$, we then conclude that $(\faithful\circ\Phi')$ is as sought.
\end{proof}

\begin{lemma}\label{Phi_D} Suppose that $D$ is a club in $\kappa$. Then the function $\Phi_{D} : \mathcal K(\kappa) \to \mathcal K(\kappa)$ defined by
\[\Phi_{D}(x):=\begin{cases}
\{ \sup(D\cap\eta)\mid \eta\in x\ \&\ \eta>\min(D)\}, &\text{if } \sup(D \cap \sup(x)) = \sup(x); \\
x \setminus \sup(D \cap \sup(x)), &\text{otherwise}
\end{cases}\]
is a postprocessing function.
\end{lemma}
\begin{proof} Let $x \in \mathcal K(\kappa)$ be arbitrary. Denote $\alpha := \sup(x)$.
We consider two cases in turn:

$\br$ Suppose $\sup(D \cap \alpha) = \alpha$.
Notice that $\Phi_{D}(x) \subseteq D$ in this case, since $D$ is club in $\kappa$.
To see that $\Phi_{D}(x)$ is cofinal in $\alpha$, let $\beta < \alpha$ be arbitrary.
We can find $\gamma \in D \cap \alpha$ above $\beta$, and then $\eta \in x$ above $\gamma$.
Clearly, $\eta > \min(D)$, so that $\delta := \sup(D \cap \eta)$ is in $\Phi_{D}(x)$.
But $\delta \geq \gamma > \beta$, as required.
To see that $\Phi_{D}(x)$ is closed in $\alpha$, let $\bar\alpha < \alpha$ be such that $\sup(\Phi_{D}(x) \cap \bar\alpha) = \bar\alpha$.
Fix a strictly increasing sequence $\langle \beta_i \mid i < \Lambda \rangle$ of elements of $\Phi_{D}(x)$, converging to $\bar\alpha$.
For each $i<\Lambda$, $\beta_i$ must be of the form $\sup(D \cap \eta_i)$ for some $\eta_i \in x$, and furthermore $\beta_i \leq \eta_i \leq \beta_{i+1}$ for every $i<\Lambda$.
Thus $\langle \eta_i \mid i<\Lambda \rangle$ is a sequence of elements of $x$ converging to $\bar\alpha$, so that $\bar\alpha \in x$ since $x$ is club in $\alpha$.
Furthermore, $\bar\alpha \in \acc^+(\Phi_{D}(x)) \subseteq \acc^+(D) \subseteq D$, so that $\sup(D \cap \bar\alpha) = \bar\alpha$, and clearly $\bar\alpha > \min(D)$.
Thus $\bar\alpha \in \Phi_{D}(x)$, as witnessed by $\bar\alpha \in x$.
Altogether, we have shown that $\Phi_{D}(x)$ is club in $\alpha$, and also that $\acc(\Phi_{D}(x)) \subseteq \acc(x)$.

Consider arbitrary $\bar\alpha \in \acc(\Phi_{D}(x))$, in order to compare $\Phi_{D}(x \cap \bar\alpha)$ with $\Phi_{D}(x) \cap \bar\alpha$.
We have already seen that $\bar\alpha \in \acc(x) \cap \acc(D)$ in this case.
In particular, $\sup(x \cap \bar\alpha) = \bar\alpha = \sup(D \cap \bar\alpha)$, and $\sup(D \cap \eta) < \bar\alpha \iff \eta < \bar\alpha$, so that
\begin{align*}
\Phi_{D}(x \cap \bar\alpha) &= \{ \sup(D\cap\eta)\mid \eta\in x \cap \bar\alpha \ \&\ \eta>\min(D)\} \\
&= \{ \sup(D\cap\eta)\mid \eta\in x \ \&\ \eta>\min(D)\} \cap \bar\alpha \\
&= \Phi_{D}(x) \cap \bar\alpha,
\end{align*}
as required.

$\br$ Suppose $\sup(D \cap \alpha) < \alpha$, so that $\Phi_{D}(x) = x \setminus \sup(D \cap \alpha)$ is a nonempty final segment of $x$,
which is certainly a club in $\alpha$ and satisfies $\acc(\Phi_{D}(x)) \subseteq \acc(x)$.
Consider arbitrary $\bar\alpha \in \acc(\Phi_{D}(x))$.
In particular, $\bar\alpha \in \acc(x)$ and $\sup(D \cap \alpha) < \bar\alpha < \alpha$.
But then $D \cap \bar\alpha = D \cap \alpha$, so that $\sup(D \cap \bar\alpha) = \sup(D \cap \alpha) < \bar\alpha$.
Since $\sup(x \cap \bar\alpha) = \bar\alpha$, it follows that
\begin{align*}
\Phi_{D}(x \cap \bar\alpha) &= (x \cap \bar\alpha) \setminus \sup(D \cap \bar\alpha)\\
&= (x \cap \bar\alpha) \setminus \sup(D \cap \alpha) = (x \setminus \sup(D \cap \alpha)) \cap \bar\alpha = \Phi_{D}(x) \cap \bar\alpha,
\end{align*}
as required.
\end{proof}
\begin{remark} Whenever $\sup(D\cap \sup(x)) = \sup(x)$, we have that $\Phi_D(x)$ coincides with  $\Drop(x,D)$ of \cite{kojman-abc-of-pcf}.
See Fact~3 of that paper for some of the basic properties of $\Drop$, and hence of $\Phi_D$.
\end{remark}

\begin{fact}[Shelah, {\cite[$\S2$]{Sh:365}}]\label{clubguessing}
For any $C$-sequence  $\vec C=\langle C_\alpha\mid\alpha\in S\rangle$ over some stationary $S\s \acc(\kappa)$:
\begin{enumerate}
\item If $\sup_{\alpha \in S} |C_\alpha|^+ < \kappa$, then there exists some club $D \subseteq \kappa$ such that $\{ \alpha \in S \mid \Phi_{D}(C_\alpha)\subseteq E \}$ is stationary for every club $E \subseteq \kappa$;
\item If $\vec C$ is amenable, then there exists a club $D\s\kappa$ such that $\{\alpha\in S\mid \sup(\nacc(\Phi_{D}(C_\alpha))\cap E)=\alpha\}$ is stationary for every club $E\s\kappa$.
\end{enumerate}
\end{fact}

\begin{lemma}[wide club guessing]\label{wide-club-guessing}
Suppose that  $\langle \mathcal C_\alpha\mid\alpha<\kappa\rangle$ is a $\square(\kappa,{<}\mu,{\sq_\chi})$-sequence with support $\Gamma$,
and $\mathcal S$ is a collection of less than $\kappa$ many stationary subsets of $\Gamma$.

If $\kappa \geq \aleph_2$ and $\mu<\kappa$, then there exists a $\min$-preserving faithful postprocessing function $\Phi:\mathcal K(\kappa)\rightarrow\mathcal K(\kappa)$ such that for every club $E\s\kappa$ and every $S\in\mathcal S$, there exists $\alpha\in S$ with $\sup(\nacc(\Phi(C))\cap E)=\alpha$ for all $C\in\mathcal C_\alpha$.
\end{lemma}
\begin{proof} For every club $D\s\kappa$, let $\Phi_D$ denote the postprocessing function given by Lemma~\ref{Phi_D}.
The proof of the next claim is essentially the same as that of Fact~\ref{clubguessing}(2).

\begin{claim} For every $S \in \mathcal S$, there exists a club $D\s\kappa$  such that for every club $E\s\kappa$,
there is $\delta\in S$ with $\sup(\nacc(\Phi_{D}(C))\cap E)=\delta$ for all $C\in\mathcal C_\delta$.
\end{claim}
\begin{proof} Let $S\in\mathcal S$ be arbitrary, and suppose the conclusion fails.
Thus, for every club $D \subseteq \kappa$ we can fix some club $E^D \subseteq \kappa$ such that for every $\delta \in S$ there is some $C^D_\delta \in \mathcal C_\delta$ with $\sup(\nacc(\Phi_D(C^D_\delta)) \cap E^D)<\delta$.

Fix $\bar\mu<\mu$ for which $S':=\{\alpha\in S\mid |\mathcal C_\alpha|=\bar\mu\}$ is stationary.
Let $\mu' := \max\{\bar\mu^+,\aleph_1\}$, so that $\mu'<\kappa$.
Define a sequence $\langle E_i\mid i \leq \mu' \rangle$ of clubs in $\kappa$ as follows:
\begin{itemize}
\item Set $E_0:=\kappa$;
\item For all $i< \mu'$, set $E_{i+1} := E^{E_i} \cap E_i$;
\item For all $i\in\acc(\mu' +1)$, set $E_i=\bigcap_{j<i}E_j$, which is a club in $\kappa$ since $\mu' < \kappa$.
\end{itemize}

Write $E:=E_{\mu'}$.

Consider arbitrary $\delta\in S'$.
Since $\{ C^{E_i}_\delta \mid i<\mu' \} \subseteq \mathcal C_\delta$, and $\mu'$ is regular and greater than $\bar\mu = | \mathcal C_\delta |$,
we can pick $C_\delta\in\mathcal C_\delta$ such that $I_\delta:=\{ i< \mu' \mid C^{E_i}_\delta=C_\delta\}$ is cofinal in $\mu'$.
For all $\delta\in\Gamma\setminus S'$, pick $C_\delta\in\mathcal C_\delta$ arbitrarily.

Then, $\vec C:=\langle C_\delta \mid \delta \in \Gamma\rangle$ is a transversal for $\square(\kappa,{<}\mu,{\sq_\chi})$.
By $\mu<\kappa$ and Lemma \ref{square_is_amenable}, $\vec C$ is amenable,
so that we can pick $\delta$ in the stationary set $\acc(E)\cap S'$ for which $\sup(E\cap\delta\setminus C_\delta)=\delta$.

$\br$ Suppose that $\cf(\delta)>\omega$.
Let $\{ i_n\mid n<\omega\}$ be the increasing enumeration of some subset of $I_\delta$.
Since $\langle E_i \mid i<\mu'\rangle$ is a $\s$-decreasing sequence, for all $n<\omega$, we have in particular that $E_{i_{n+1}} \subseteq E_{i_n+1} \subseteq E^{E_{i_n}}$,
so that $\alpha_n:=\sup(\nacc(\Phi_{E_{i_n}}(C_\delta))\cap E_{i_{n+1}})$ is $<\delta$.
Put $\alpha:=\sup_{n<\omega}\alpha_n$.
As $\cf(\delta)>\omega$, we have $\alpha<\delta$.
Fix $\beta\in (E\cap\delta)\setminus C_\delta$ above $\alpha$.
Put $\gamma:=\min(C_\delta\setminus\beta)$.
Then $\delta>\gamma>\beta>\alpha$, and for all $i < \mu'$, since $\beta\in E_i$, we infer that $\sup( E_i\cap\gamma)\ge\beta$.
So it follows from the definition of $\Phi_{E_i}(C_\delta)$ that $\min(\Phi_{E_i}(C_\delta)\setminus\beta)=\sup(E_i\cap \gamma)$ for all $i< \mu'$.
Since $\langle E_{i_n}\mid n<\omega\rangle$ is an infinite $\s$-decreasing sequence, let us fix some $n<\omega$ such that $\sup(E_{i_n}\cap\gamma)=\sup(E_{i_{n+1}}\cap\gamma)$.
Then $\min(\Phi_{E_{i_n}}(C_\delta)\setminus\beta)=\min(\Phi_{E_{i_{n+1}}}(C_\delta)\setminus\beta)$,
and in particular, $\beta^*:=\min(\Phi_{E_{i_n}}(C_\delta)\setminus\beta)$ is in $E_{i_{n+1}}\setminus(\alpha+1)$.
Now, there are two options, each leading to a contradiction:
\begin{itemize}
\item If $\beta^*\in\nacc(\Phi_{E_{i_n}}(C_\delta))$, then we get a contradiction to the fact that  $\beta^*>\alpha\ge\alpha_n$.
\item If $\beta^*\in\acc(\Phi_{E_{i_n}}(C_\delta))$, then $\beta^*=\beta$ and $\beta^*\in\acc(C_\delta)$, contradicting the fact that $\beta\notin C_\delta$.
\end{itemize}

$\br$ Suppose that $\cf(\delta)=\omega$.
For all $i\in I_\delta$, we have that $\alpha_i:=\sup(\nacc(\Phi_{E_{i}}(C_\delta))\cap E_{i+1})$ is $<\delta$.
As $\cf(\delta)\neq\omega_1$, let $\{ i_\nu\mid \nu<\omega_1\}$ be the increasing enumeration of some subset of $I_\delta$, for which  $\alpha:=\sup_{\nu<\omega_1}\alpha_{i_\nu}$ is $<\delta$.
Fix $\beta\in (E\cap\delta)\setminus C_\delta$ above $\alpha$.
Put $\gamma:=\min(C_\delta\setminus\beta)$.
Then $\delta>\gamma>\beta>\alpha$, and $\min(\Phi_{E_i}(C_\delta)\setminus\beta)=\sup(E_i\cap \gamma)$ for all $i < \mu'$.
Fix some $\nu<\omega_1$ such that $\sup(E_{i_\nu}\cap\gamma)=\sup(E_{i_{\nu+1}}\cap\gamma)$.
Then $\beta^*:=\min(\Phi_{E_{i_\nu}}(C_\delta)\setminus\beta)$ is in $E_{i_{\nu+1}}\setminus(\alpha+1)$, and as in the previous case, each of the two possible options leads to a contradiction.
\end{proof}

For each $S\in\mathcal S$, let $D_S$ be given by the preceding claim. Put $D:=\bigcap_{S\in\mathcal S}D_S$.
Since $|\mathcal S|<\kappa$, $D$ is a club in $\kappa$.
Define $\Phi:\mathcal K(\kappa)\rightarrow\mathcal K(\kappa)$ by stipulating:
$$\Phi(x):=\begin{cases}
\{1\}\cup(\Phi_D(x)\setminus 3),&\text{if }\min(x)=1\ \&\ 2\in \Phi_D(x);\\
\{\min(x)\}\cup(\Phi_D(x)\setminus\min(x)),&\text{otherwise}.\\
\end{cases}$$

Since $\Phi_D$ is a postprocessing function, it follows that $\Phi$ is as well.
Of course, $\Phi$ is faithful and $\min$-preserving.
Consider arbitrary club $E \subseteq \kappa$ and $S \in \mathcal S$. We shall find some $\delta\in S$ such that $\sup(\nacc(\Phi(C))\cap E)=\delta$ for all $C\in\mathcal C_\delta$.
As $\Phi(x)=^*\Phi_D(x)$ for all $x$, it suffices to verify this against $\Phi_D$.

Let $E':=E\cap\acc(D)$.
By our choice of $D_S$, we may fix $\delta \in S$ with $\sup(\nacc(\Phi_{D_S}(C)) \cap E') = \delta$ for all $C \in \mathcal C_\delta$.
Consider arbitrary $C \in \mathcal C_\delta$.
To see that $\sup(\nacc(\Phi_D(C)) \cap E) = \delta$, first notice that $\delta \in \acc(E') \subseteq \acc(D) \subseteq \acc(D_S)$,
so that $\Phi_{D_S}(C) = \{\sup(D_S \cap \eta) \mid \eta \in C\ \&\ \eta>\min(D_S)\}$ and $\Phi_D(C) = \{\sup(D \cap \eta) \mid \eta \in C\ \&\ \eta>\min(D)\}$.
In particular, $\acc(\Phi_{D_S}(C))=\acc(D_S)\cap\acc(C)$ and $\acc(\Phi_D(C))=\acc(D)\cap\acc(C)$.

Fix an arbitrary $\beta\in\nacc(\Phi_{D_S}(C)) \cap E'$, and we shall show that $\beta\in\nacc(\Phi_D(C))\cap E$.
As $\beta\in\ E'$, we have $\beta\in\acc(D)\s\acc(D_S)$.
So, since $\acc(\Phi_{D_S}(C))=\acc(D_S)\cap\acc(C)$, we have that $\eta:=\min(C\setminus\beta)$ is in $\nacc(C)$, and $\beta=\sup(D_S\cap\eta)$.
As $D\s D_S$, we have $\sup(D\cap\eta)\le\beta$. As $\beta\in\acc(D\cap\eta)$, we also have $\beta\le\sup(D\cap\eta)$.
Recalling that $\eta\in\nacc(C)$ and $\beta\in E'\s E$, we conclude that $\beta=\sup(D\cap\eta)\in\nacc(\Phi_D(C))\cap E$.
\end{proof}

\begin{remark} Our ``wide club guessing'' lemma is as wide as provably possible, in the sense that the hypothesis $\mu<\kappa$ cannot be waived.
Specifically, in Kunen's model from \cite[\S3]{MR495118}, there exists an inaccessible cardinal $\kappa$ such that $\square(\kappa,{<}\kappa)$ holds,
but for any $\square(\kappa,{<}\kappa)$-sequence $\langle\mathcal C_\alpha\mid\alpha<\kappa\rangle$, there exists a club $E\s\kappa$
such that for all $\alpha<\kappa$, there is $C\in\mathcal C_\alpha$ with $\sup(\nacc(C)\cap E)<\alpha$.
\end{remark}

\begin{defn} A \emph{$\kappa$-assignment} is a matrix $\mathfrak Z=\langle Z_{x,\beta}\mid x\in\mathcal K(\kappa), \beta\in\nacc(x)\rangle$
satisfying for all $x\in\mathcal K(\kappa)$:
\begin{itemize}
\item $Z_{x,\beta}\s \beta$ for all $\beta\in\nacc(x)$;
\item $Z_{x,\beta}=Z_{x\cap\bar\alpha,\beta}$ for all $\bar\alpha\in\acc(x)$ and $\beta\in\nacc(x\cap\bar\alpha)$.
\end{itemize}
\end{defn}

\begin{lemma}\label{phiZ} Suppose that $\mathfrak Z=\langle Z_{x,\beta}\mid x\in\mathcal K(\kappa), \beta\in\nacc(x)\rangle$ is a $\kappa$-assignment.
Define $\Phi_{\mathfrak Z}:\mathcal K(\kappa)\rightarrow\mathcal K(\kappa)$ by stipulating $\Phi_{\mathfrak Z}(x):=\rng(g_{x,\mathfrak Z})$, where $g_{x,\mathfrak Z}:x\rightarrow\sup(x)$ is the function satisfying
\[g_{x,\mathfrak Z}(\beta) := \begin{cases}
\beta,                                                                      &\text{if } \beta \in \acc(x);\\
\min(Z_{x,\beta}\cup\{\beta\}),                                             &\text{if } \beta = \min(x);\\
\min\left((Z_{x,\beta}\cup\{\beta\})\setminus(\sup(x\cap\beta)+1)\right),   &\text{otherwise}.
\end{cases}\]

Then $\Phi_{\mathfrak Z}$ is an $\acc$-preserving postprocessing function, as follows from the following:
\begin{enumerate}
\item $g_{x,\mathfrak Z}$ is strictly increasing, continuous, and cofinal in $\sup(x)$, so that $\rng(g_{x,\mathfrak Z})$ is a club in $\sup(x)$ of order-type $\otp(x)$;
\item $\acc(\rng(g_{x,\mathfrak Z}))=\acc(x)$ and $\nacc(\rng(g_{x,\mathfrak Z}))=g_{x,\mathfrak Z}[\nacc(x)]$;
\item if $x,y\in\mathcal K(\kappa)$ and $y\sq x$, then $g_{y,\mathfrak Z} = g_{x,\mathfrak Z} \restriction \sup(y)$, so that $\rng(g_{y,\mathfrak Z})\sq \rng(g_{x,\mathfrak Z})$.
\end{enumerate}
\end{lemma}
\begin{proof} Recall that every $x \in \mathcal K(\kappa)$ is a club subset of $\sup(x)$.
For every $x \in \mathcal K(\kappa)$ and every $\beta \in \nacc(x) \setminus \{0\}$, we have $\sup(x\cap\beta) < g_{x, \mathfrak Z}(\beta) \leq \beta$.
Parts (1) and~(2) follow immediately.

For (3), consider $x, y \in \mathcal K(\kappa)$ with $y \sq x$, and arbitrary $\beta \in y$.
If $\beta \in \acc(y)$ then clearly $g_{y, \mathfrak Z}(\beta) = \beta = g_{x, \mathfrak Z}(\beta)$.
If $\beta \in \nacc(y)$, then since $y \sq x$ we have $\sup(y \cap \beta) = \sup(x \cap \beta)$,
and by the fact that $\mathfrak Z$ is a $\kappa$-assignment we have $Z_{y,\beta} = Z_{x,\beta}$,
so that $g_{y, \mathfrak Z}(\beta) = g_{x, \mathfrak Z}(\beta)$ follows.
\end{proof}

\begin{example}\label{phiZ-simpler} Any sequence $\langle Z_{\beta}\mid \beta<\kappa\rangle$ gives rise to a $\kappa$-assignment
$\mathfrak Z := \langle Z_{x,\beta}\mid x\in\mathcal K(\kappa),\allowbreak\beta\in\nacc(x)\rangle$ via the rule $Z_{x,\beta} := Z_\beta \cap \beta$.
Now, let  $\Phi_{\mathfrak Z}$ be given by Lemma~\ref{phiZ}.
Note that if $x \in \mathcal K(\kappa)$ and $A \subseteq \kappa$ are such that $\nacc(x)\s\{\beta\in \acc^+(A)\mid Z_\beta = A \cap \beta\}$, then
$\nacc(\Phi_{\mathfrak Z}(x)) \subseteq A$ and $\min(\Phi_{\mathfrak Z}(x)) = \min(A)$.
Likewise, for $f\in\{\sup,\otp\}$, if $f(\nacc(x)\cap\{\beta\in \acc^+(A)\mid Z_\beta = A \cap \beta\})=f(x)$, then $f(\nacc(\Phi_{\mathfrak Z}(x))\cap A)=f(x)$.
\end{example}

\begin{fact}[\cite{paper28}]\label{paper28b} Suppose that $\diamondsuit(\kappa)$ holds.
For any infinite regular cardinal $\theta<\kappa$, there exists an $\acc$-preserving postprocessing function $\Phi:\mathcal K(\kappa)\rightarrow\mathcal K(\kappa)$ satisfying the following.
For every sequence $\langle A_i\mid i<\theta\rangle$ of cofinal subsets of $\kappa$, there exists some stationary subset $G\s \kappa$ such that for all $x\in\mathcal K(\kappa)$,
if $\sup(\nacc(x)\cap G)=\sup(x)$ and $\cf(\sup(x))=\theta$, then $\sup(\nacc(\Phi(x))\cap A_i)=\sup(x)$ for all $i<\theta$.
\end{fact}

\begin{fact}[folklore]\label{diamond_matrix} $\diamondsuit(\kappa)$ entails the existence of a matrix $\langle A^i_\gamma\mid i,\gamma<\kappa\rangle$
with $A^i_\gamma\s\gamma$ for all $i,\gamma<\kappa$, such that for every sequence $\vec A=\langle A^i\mid i<{\kappa} \rangle $ of cofinal subsets of ${\kappa}$,
the following set is stationary: $$G(\vec{A}):=\{\gamma<{\kappa}\mid \forall i<\gamma(\sup(A^i\cap\gamma)=\gamma\ \&\ A^i\cap\gamma=A^i_\gamma)\}.$$
\end{fact}

\begin{defn}[\cite{paper24}] Define an ideal $J[\kappa]\s\mathcal P(\kappa)$, as follows.
A subset $S\s\kappa$ is in $J[\kappa]$ iff there exists a club $C\s\kappa$ and a sequence of functions $\langle f_i:\kappa\rightarrow\kappa\mid i<\kappa\rangle$ satisfying the following.
For every $\alpha\in S\cap C$, every regressive function $f:\alpha\rightarrow\alpha$, and every cofinal subset $B\s\alpha$, there exists some $i<\alpha$ such that $\sup\{ \beta\in B\mid f_i(\beta)=f(\beta)\}=\alpha$.
\end{defn}

The following Lemma is a postprocessing-function version of \cite[Theorem~4.3]{paper24}.

\begin{lemma}\label{thm16} Suppose that $\diamondsuit(\kappa)$ holds and that $\theta \in \reg(\kappa)$.
Suppose also that $\vec C=\langle C_\alpha\mid \alpha\in S\rangle$ is a $C$-sequence, with $S\in J[\kappa]\cap\mathcal P(E^\kappa_\theta)$,
satisfying that for every club $E\s\kappa$, there exists  $\alpha\in S$ with $\sup(\nacc(C_\alpha)\cap E)=\alpha$.

Then there exists a faithful postprocessing function $\Phi:\mathcal K(\kappa)\rightarrow\mathcal K(\kappa)$ satisfying:
\begin{enumerate}
\item for every $x$, $\sup(\nacc(x)\setminus \nacc(\Phi(x)))<\sup(x)$;
\item for every $\zeta<\kappa$ and every sequence $\langle A_i\mid i<\theta\rangle$ of cofinal subsets of $\kappa$, there exists $\alpha\in S$ with $\min(C_\alpha)=\zeta$ such that $\sup(\nacc(\Phi(C_\alpha))\cap A_i)=\alpha$ for all $i<\theta$.
\end{enumerate}
\end{lemma}
\begin{proof} Define an ideal $\mathcal I\s\mathcal P(S)$, as follows.
For every $T\s S$: $T\in\mathcal I$ iff there exists some club $E\s\kappa$ such that $\sup(\nacc(C_\alpha)\cap E)<\alpha$ for all $\alpha\in T$.
By the club-guessing feature of $\vec C$, $S\notin\mathcal I$.
\begin{claim} $\mathcal I$ is normal.
\end{claim}
\begin{proof} Suppose that $\langle T_i\mid i<\kappa\rangle$ is a sequence of elements of $\mathcal I$. For each $i<\kappa$, pick a club $E_i$ witnessing that $T_i\in\mathcal I$.
Consider the diagonal union $T:=\{\alpha<\kappa\mid \exists i<\alpha(\alpha\in T_i)\}$ and the diagonal intersection $E:=\{\alpha<\kappa\mid \forall i<\alpha(\alpha\in E_i)\}$.
We claim that the club $E$ witnesses that $T\in\mathcal I$. Let $\alpha\in T$ be arbitrary.
Pick $i<\alpha$ such that $\alpha\in T_i$.
As $\sup(E\setminus E_i)\le i+1<\alpha$ and $\sup(\nacc(C_\alpha)\cap E_i)<\alpha$, we have $\sup(\nacc(C_\alpha)\cap E)<\alpha$, so we are done.
\end{proof}

It is clear that just like Lemma~\ref{split_amenable} admits a generalization to arbitrary normal filters (recall Lemma~\ref{lemma127}), so does Lemma~\ref{split_amenable2}.
Thus, by appealing to this generalization with the dual filter of $\mathcal I$, $\vec C$,\footnote{Note that $\vec C$ is amenable due to its club-guessing feature.}
and the constant $\kappa$-sequence whose unique element is $S$, we obtain a postprocessing function $\Phi_0:\mathcal K(\kappa)\rightarrow\mathcal K(\kappa)$ such that:
\begin{itemize}
\item For all $x\in\mathcal K(\kappa)$, we have $\Phi_0(x)=^* x$;
\item $\Upsilon:=\{\tau<\kappa\mid \{ \alpha\in S \mid \min(\Phi_0(C_\alpha))=\tau\}\notin\mathcal I\}$ is cofinal in $\kappa$.
\end{itemize}
Denote $C_\alpha^\circ:=\Phi_0(C_\alpha)$.
For each $\tau\in\Upsilon$, put $S_\tau:=\{\alpha\in S\mid \min(C_\alpha^\circ)=\tau\}$.
Then for every club $E\s\kappa$ and every $\tau\in\Upsilon$, there exists $\alpha\in S_\tau$ such that $\sup(\nacc(C^\circ_\alpha)\cap E)=\alpha$.
In particular, $S_\tau$ is stationary.

Next, we follow the proof of \cite[Theorem 4.3]{paper24}.
As $S\in J[\kappa]$, let us fix a club $C\s\kappa$ and a sequence of functions $\langle f_i:\kappa\rightarrow\kappa\mid i<\kappa \rangle$
such that for every $\alpha\in S\cap C$, every regressive function $f:\alpha\rightarrow\alpha$,
and every cofinal subset $B\s\alpha$, there exists some $i<\alpha$ such that $\sup\{ \beta\in B\mid f_i(\beta)=f(\beta)\}=\alpha$.
By $\diamondsuit(\kappa)$, fix a matrix $\langle A^i_\gamma\mid i,\gamma<\kappa \rangle$ as in Fact~\ref{diamond_matrix}.
For each $i<\kappa$, derive a $\kappa$-assignment $\mathfrak Z^i=\langle Z^i_{x,\beta}\mid x\in\mathcal K(\kappa), \beta\in\nacc(x)\rangle$
via the rule $Z^i_{x,\beta}:=A^i_{f_i(\beta)} \cap\beta$, and consider the corresponding postprocessing function $\Phi_{\mathfrak Z^i}$ given by Lemma~\ref{phiZ}.

\begin{claim} For each $\tau\in\Upsilon$,  there is $i<\kappa$ such that for every cofinal $A\s\kappa$,
there exists $\alpha\in S_\tau$ with $$\sup(\nacc(\Phi_{\mathfrak Z^i}(C^\circ_\alpha))\cap A)=\alpha.$$
\end{claim}
\begin{proof} Towards a contradiction, suppose that $\tau\in\Upsilon$ is a counterexample.
Then there exists a sequence of cofinal subsets of $\kappa$, $\vec A=\langle A^i\mid i<\kappa \rangle $,
such that for all $i<\kappa$ and $\alpha\in S_\tau$,
we have $\sup(\nacc(\Phi_{\mathfrak Z^i}(C^\circ_\alpha))\cap A^i)<\alpha$. Let $G$ be $G(\vec{A})$ as in Fact~\ref{diamond_matrix}.
Then $G$ is a stationary subset of $\kappa$, and $E:= C \cap \acc^+(G)$ is a club in $\kappa$.
Pick $\alpha\in S_\tau$ such that $B:=\nacc(C^\circ_\alpha)\cap E$ is cofinal in $\alpha$.
In particular, $\alpha \in \acc^+(E) \subseteq E \subseteq C$.
For all $\beta\in B$, since $\beta\in E$, we know that the relative interval $G\cap(\sup(C_\alpha^\circ\cap\beta),\beta)$ is nonempty.
Consequently, we may find some regressive function $f:\alpha\rightarrow\alpha$ such that $f(\beta)\in G\cap(\sup(C^\circ_\alpha\cap\beta),\beta)$ for all $\beta\in B$.
Since $\alpha \in S \cap C$, we may pick $i<\alpha$ and a cofinal subset $B'\s B$ such that $f_i\restriction B'=f\restriction B'$.
Fix a large enough $\eta\in C^\circ_\alpha$ such that $\sup(C^\circ_\alpha\cap\eta)\ge i$.
By omitting an initial segment, we may assume that $\min(B')> \eta$.

Let $\beta\in B'$ be arbitrary. As $\beta\in\nacc(C^\circ_\alpha)$ and $\beta>\eta\geq\min(C^\circ_\alpha)$,
we have $$g_{C^\circ_\alpha,\mathfrak Z^i}(\beta)=\min\left(((A^i_{f_i(\beta)}\cap\beta)\cup\{\beta\})\setminus(\sup(C^\circ_\alpha\cap\beta)+1)\right),$$
where the function $g_{C^\circ_\alpha,\mathfrak Z^i} : C^\circ_\alpha\rightarrow\alpha$ is the one defined in Lemma~\ref{phiZ}.

Write $\gamma:=f_i(\beta)$. Then $\gamma\in G\cap(\sup(C^\circ_\alpha\cap\beta),\beta)\s(i,\beta)$.
In particular, $\sup(A^i\cap\gamma)=\gamma$ and $A^i\cap\gamma=A^i_\gamma$, so that
$$g_{C^\circ_\alpha,\mathfrak Z^i}(\beta)=\min\left(A^i\setminus(\sup(C^\circ_\alpha\cap\beta)+1)\right).$$

Consequently, $g_{C^\circ_\alpha,\mathfrak Z^i}[B']\s A^i$, and hence $\sup(\nacc(\Phi_{\mathfrak Z^i}(C^\circ_\alpha))\cap A^i)=\alpha$, contradicting the choice of $A^i$.
\end{proof}

For each $\tau\in\Upsilon$, let $i_\tau$ be given by the preceding claim. Let $\Phi_1$ be given by Fact~\ref{paper28b}.
Define $\Phi_2 : \mathcal K(\kappa) \to \mathcal K(\kappa)$ by stipulating:
$$\Phi_2(x):=\begin{cases}
\nacc(\Phi_0(x))\cup(\Phi_1(\Phi_{\mathfrak Z^{i_\tau}}(\Phi_0(x)))\setminus\tau)\cup\{\otp(\Upsilon\cap\tau)\},&\text{if }\tau:=\min(\Phi_0(x))\text{ is in }\Upsilon;\\
\Phi_0(x),&\text{otherwise}.
\end{cases}$$
\begin{claim} $\Phi_2$ is a postprocessing function, and $\sup(\nacc(x)\setminus \nacc(\Phi_2(x)))<\sup(x)$ for all $x$.
\end{claim}
\begin{proof} Since all elements of $\{\Phi_1,\Phi_{\mathfrak Z^i}\mid i<\kappa\}$ are $\acc$-preserving postprocessing functions,
we get that for all $x\in\mathcal K(\kappa)$ and $i<\kappa$, $\acc(\Phi_2(x))=\acc(\Phi_1(\Phi_{\mathfrak Z^{i}}(\Phi_0(x))))=\acc(\Phi_0(x))$, so that $\Phi_2(x)$ is a club in $\sup(x)$.
Since $\nacc(\Phi_0(x))\s \nacc(\Phi_2(x))$ and $\Phi_0(x)=^*x$, we also have $\sup(\nacc(x)\setminus \nacc(\Phi_2(x)))<\sup(x)$.

Finally, suppose that $x\in\mathcal K(\kappa)$ and $\bar\alpha\in\acc(\Phi_2(x))$. Put $\tau:=\min(\Phi_0(x))$.
As $\Phi_0(x\cap\bar\alpha)=\Phi_0(x)\cap\bar\alpha$, we have $\min(\Phi_0(x\cap\bar\alpha))=\tau$, and hence $\Phi_2(x\cap\bar\alpha)=\Phi_2(x)\cap\bar\alpha$.
\end{proof}

\begin{claim} For every $\zeta<\kappa$ and every sequence $\langle A_i\mid i<\theta\rangle$ of cofinal subsets of $\kappa$,
there exists $\alpha\in S$ such that $\min(\Phi_2(C_\alpha))=\zeta$ and $\sup(\nacc(\Phi_2(C_\alpha))\cap A_i)=\alpha$ for all $i<\theta$.
\end{claim}
\begin{proof} Let $\langle A_i \mid i<\theta \rangle$ be an arbitrary sequence of cofinal subsets of $\kappa$.
By our choice of $\Phi_1$, let $G \subseteq \kappa$ be a stationary subset satisfying the property given in Fact~\ref{paper28b}.
Let $\zeta<\kappa$ be arbitrary. Put $\tau:=\Upsilon(\zeta)$.
Then, by our choice of $i_\tau$, fix $\alpha \in S_\tau$ such that $\sup(\nacc(\Phi_{\mathfrak Z^{i_\tau}}(C_\alpha^\circ)) \cap G) = \alpha$.
Since $S_\tau\s S \subseteq E^\kappa_\theta$, we have $\cf(\alpha) =  \theta$.
Then by our choice of $G$ we have $\sup(\nacc(\Phi_1(\Phi_{\mathfrak Z^{i_\tau}}(C^\circ_\alpha))) \cap A_i) = \alpha$ for all $i<\theta$.
But $C_\alpha^\circ=\Phi_0(C_\alpha)$ and $\alpha\in S_\tau$, so that $\min(\Phi_0(C_\alpha))=\tau$.
Consequently, $\min(\Phi_2(C_\alpha))=\otp(\Upsilon\cap\tau)=\zeta$, and $\sup(\nacc(\Phi_2(C_\alpha)) \cap A_i) = \alpha$ for all $i<\theta$.
\end{proof}

Let $\faithful$ be given by Example~\ref{faithful_correction}.
Then $\Phi:=\faithful\circ\Phi_2$ is as sought.
\end{proof}

\begin{defn} A $C$-sequence $\langle e_\alpha\mid\alpha<\kappa\rangle$ is said to be \emph{standard} if all of the following hold:
\begin{itemize}
\item For all $\alpha<\kappa$, $e_{\alpha+1}:=\{\alpha\}$;
\item For all $\alpha\in E^\kappa_\omega$, $\otp(e_\alpha)=\omega$;
\item For all $\alpha\in E^\kappa_{>\omega}$, $\otp(e_\alpha)=\cf(\alpha)$, $\nacc(e_\alpha)\s\nacc(\alpha)$, and $\omega\setminus e_\alpha=\{0\}$.
\end{itemize}
\end{defn}

\begin{notation}\label{notationaction} Given a $\mathcal C$-sequence $\cvec{C}=\langle \mathcal C_\alpha\mid\alpha<\kappa\rangle$,
and a postprocessing function $\Phi:\mathcal K(\kappa)\rightarrow\mathcal K(\kappa)$,
we let $\acts{\Phi}{C}$ denote any $\mathcal C$-sequence $\langle\mathcal D_\alpha\mid\alpha<\kappa\rangle$ satisfying the following two properties:
\begin{itemize}
\item For all $\alpha\in\Gamma(\cvec{C})$, $\mathcal D_\alpha=\{ \Phi(C) \mid C \in \mathcal C_\alpha \}$;
\item For all $\alpha\in\kappa\setminus\Gamma(\cvec{C})$, $\mathcal D_\alpha=\{e_\alpha\}$, where $\langle e_\alpha\mid\alpha<\kappa\rangle$ is some standard $C$-sequence,
and if $\sup(\bigcup\mathcal C_\alpha\cap \nacc(\alpha))=\alpha$, then $e_\alpha\s\bigcup\mathcal C_\alpha\cup(\omega+1)$.
\end{itemize}
It is clear that such a sequence can always be found.
\end{notation}

\begin{lemma}\label{pp-preserves-square} Suppose that $\square_\xi(\kappa,{<}\mu, {\sq_\chi}, \mathcal R_1)$ holds, where either:
\begin{itemize}
\item $\mathcal R_1 = V$ or
\item $\mathcal R_1 = {\notin}$ and $\min\{\xi,\mu\}<\kappa$.
\end{itemize}

For every postprocessing function $\Phi : \mathcal K(\kappa) \to \mathcal K(\kappa)$, we have:
\begin{enumerate}
\item If $\cvec{C}=\langle\mathcal C_\alpha \mid \alpha < \kappa\rangle$ is a witness to $\square_\xi(\kappa,{<}\mu, {\sq_\chi}, \mathcal R_1)$,
then $\acts{\Phi}{C}$ is yet another witness, with some support $\Gamma^\bullet\supseteq\Gamma(\cvec{C})$;
if $\Phi$ is faithful or $\chi=\aleph_0$, then $\Gamma^\bullet=\Gamma(\cvec{C})$;
\item Suppose that $\langle C_\alpha\mid\alpha\in\Gamma\rangle$ is a transversal for $\square_\xi(\kappa,{<}\mu, {\sq_\chi}, \mathcal R_1)$.
If $\Phi$ is faithful or $\chi=\aleph_0$, then $\langle \Phi(C_\alpha)\mid\alpha\in\Gamma\rangle$ is yet another transversal for $\square_\xi(\kappa,{<}\mu, {\sq_\chi}, \mathcal R_1)$.
\end{enumerate}
\end{lemma}
\begin{proof} (1) Write $\Gamma:=\Gamma(\cvec{C})$. By Lemma~\ref{Gamma-closure}(3), $E^\kappa_\omega\cup E^\kappa_{\ge\chi} \subseteq \Gamma$.
Write $\cvec{C^\bullet}=\langle\mathcal C_\alpha^\bullet\mid\alpha<\kappa\rangle$ for $\acts{\Phi}{C}$, and $\Gamma^\bullet:=\Gamma(\cvec{C^\bullet})$.
Consider arbitrary $\alpha \in \Gamma$.
Clearly $|\mathcal C_\alpha^\bullet| \leq |\mathcal C_\alpha| < \mu$.
Let $c \in \mathcal C_\alpha^\bullet$ be arbitrary.
Then we can fix some $C \in \mathcal C_\alpha$ such that $c = \Phi(C)$.
Since $\Phi$ is a postprocessing function, $c$ is club in $\sup(C) = \alpha$ and $\otp(c) \leq \otp(C)  \leq \xi$.
In particular, if $c$ has order-type $\xi$, then so does $C$, so that $|\{c \in \mathcal C_\alpha^\bullet \mid \otp(c) = \xi \}| \leq |\{C \in \mathcal C_\alpha \mid \otp(C) = \xi \}| \leq 1$.

Let $\bar\alpha \in \acc(c)$ be arbitrary.
Since $\Phi$ is a postprocessing function, we have $\bar\alpha \in \acc(C)$ and $c \cap \bar\alpha = \Phi(C) \cap\bar\alpha = \Phi(C \cap\bar\alpha)$,
where by $\alpha \in \Gamma$, we have $C\cap\bar\alpha \in \mathcal C_{\bar\alpha}$.
By Lemma~\ref{Gamma-closure}(2), we also have $\bar\alpha \in \Gamma$.
Thus $c \cap \bar\alpha \in \mathcal C^\bullet_{\bar\alpha}$.
In particular, it follows that $\alpha \in \Gamma^\bullet$.
Thus we have shown that $\Gamma \subseteq \Gamma^\bullet$.
Consequently, $\chi=\aleph_0$ would entail that  $\acc(\kappa)=\Gamma=\Gamma^\bullet$.

Next, suppose that $\chi>\aleph_0$ and consider an arbitrary $\alpha \in \acc(\kappa) \setminus\Gamma$.
Then $\alpha\notin E^\kappa_\omega\cup E^\kappa_{\ge\chi}$ and $\mathcal C^\bullet_\alpha$ is a singleton whose unique element $e_\alpha$ comes from a standard $C$-sequence.
In particular, $\otp(e_\alpha) = \cf(\alpha) < \chi$ and $\nacc(e_\alpha)$ consists only of successor ordinals,
guaranteeing $\sq_\chi$-coherence as well. Furthermore, if $\Phi$ is faithful, then $e_\alpha\cap\omega=(\omega\setminus\{0\})\notin\rng(\Phi)$,
while $\mathcal C_\omega^\bullet \subseteq \rng(\Phi)$ since $\omega \in \Gamma$, and hence $e_\alpha\cap\omega\notin\mathcal C_\omega^\bullet$, so that $\alpha\notin\Gamma^\bullet$.

Consider arbitrary cofinal $A \subseteq \kappa$, and we must find some $\alpha \in \acc^+(A)$ such that $(A \cap \alpha) \mathrel{\mathcal R_1} \mathcal C_\alpha^\bullet$.
If $\mathcal R_1 = V$, then any $\alpha \in \acc^+(A)$ satisfies the required relation.
Thus, suppose that  $\mathcal R_1 = {\notin}$, so that $\min\{\xi,\mu\}<\kappa$.
Towards a contradiction, suppose that for all $\alpha\in\acc^+(A)$, we have $A\cap\alpha\in\mathcal C_{\alpha}^\bullet$.
Consider the club $D:=\acc^+(A)$, and let $\vec{C}=\langle C_\alpha\mid \alpha\in\Gamma\rangle$ be a transversal for $\cvec{C}$ satisfying $A\cap\alpha=\Phi(C_\alpha)$ for all $\alpha\in\Gamma\cap D$.
For all $\alpha\in \Gamma\cap D$, we have $D\cap\alpha=\acc^+(A)\cap\alpha=\acc(\Phi(C_\alpha))\s\acc(C_\alpha)$.
In particular, $\{\alpha\in\Gamma\mid D\cap\alpha\s C_\alpha\}$ is stationary, so that $\vec C$ is not amenable, contradicting Lemma~\ref{square_is_amenable}.

(2) Let $\cvec{C} = \langle \mathcal C_\alpha \mid \alpha<\kappa \rangle$ be a $\square_\xi(\kappa,{<}\mu, {\sq_\chi}, \mathcal R_1)$-sequence for which
$\Gamma=\Gamma(\cvec{C})$ and $\langle C_\alpha\mid\alpha\in\Gamma\rangle\in\prod_{\alpha\in\Gamma}\mathcal C_\alpha$.
Then by Clause~(1), $\acts{\Phi}{C}$ is a $\square_\xi(\kappa,{<}\mu, {\sq_\chi}, \mathcal R_1)$-sequence.
Writing $\cvec{C^\bullet}=\langle\mathcal C_\alpha^\bullet\mid\alpha<\kappa\rangle$ for $\acts{\Phi}{C}$,
it is clear that $\langle \Phi(C_\alpha) \mid\alpha\in\Gamma\rangle\in\prod_{\alpha\in\Gamma}\mathcal C_\alpha^\bullet$.
If $\Phi$ is faithful or $\chi=\aleph_0$, then Clause~(1) gives $\Gamma(\cvec{C^\bullet}) = \Gamma$,
and hence $\cvec{C^\bullet}$ witnesses that $\langle \Phi(C_\alpha)\mid\alpha\in\Gamma\rangle$ is a transversal for $\square_\xi(\kappa,{<}\mu, {\sq_\chi}, \mathcal R_1)$.
\end{proof}

\begin{remark}\label{min<k-necessary} Again, the hypothesis $\min\{\xi,\mu\}<\kappa$ cannot be waived.
Indeed, in Kunen's model from \cite[\S3]{MR495118}, there exists an inaccessible cardinal $\kappa$ such that $\square(\kappa,{<}\kappa)$ holds,
but for any witness $\cvec{C}$ to $\square(\kappa,{<}\kappa)$, there exists some postprocessing function $\Phi$ for which $\acts{\Phi}{C}$ fails to witness $\square(\kappa,{<}\kappa)$.
\end{remark}

\begin{remark} Lemma~\ref{pp-preserves-square}(1) along with the special case $(\kappa,\mu,\chi):=(\lambda^+,(\cf(\lambda))^+,\aleph_0)$ and $\mathcal S:=\{ E^{\lambda^+}_\theta\mid \theta\in\reg(\lambda)\}$
of Lemma~\ref{wide-club-guessing}, together provide an affirmative answer to Question~16 of \cite{rinot_s01}.
\end{remark}

Recall that $\mathcal D(\lambda,\theta)$ stands for the density of $[\lambda]^\theta$, that is, $\mathcal D(\lambda,\theta)=\cf([\lambda]^\theta,{\supseteq})$.
Note that $\mathcal D(\theta,\theta)\le\mathcal D(\lambda,\theta)$ whenever $\theta\le\lambda$,
and that $\mathcal D(\lambda,\theta) = \lambda$ whenever $\lambda$ is a strong-limit cardinal and $\theta \in \reg(\lambda) \setminus \{\cf(\lambda)\}$.\footnote{%
Indeed, for $\lambda$ regular, we would have $\mathcal D(\lambda,\theta)=\lambda^\theta=\lambda$;
for $\lambda$ singular, see, e.g., the proof of \cite[Claim~4.5.1]{paper24}.}

\begin{thm}\label{mixing_paper24} Suppose that $\lambda$ is an uncountable cardinal, and $\square_\xi(\lambda^+,{<}\mu,{\sq_\chi})$ holds with $\xi\le\lambda^+$, $\chi\in\reg(\lambda^+)$ and $\mu\le\lambda$.
Suppose also that $\ch_\lambda$ holds, and that $\{\theta\in\reg(\lambda)\mid \mathcal D(\lambda,\theta)=\lambda\}$ is cofinal in $\reg(\lambda)$.
Then there exists a $\square_\xi(\lambda^+,{<}\mu,{\sq_\chi})$-sequence $\langle \mathcal C_\alpha\mid\alpha<\lambda^+\rangle$ with a transversal $\langle C_\alpha\mid \alpha\in\Gamma\rangle$, and a cofinal subset $\Theta\s\reg(\lambda)$ such that for every $\theta\in\Theta$:
\begin{enumerate}
\item For every club $E\s\lambda^+$, there exists $\alpha\in E^{\lambda^+}_\theta\cap\Gamma$ such that $\sup(\nacc(C)\cap E)=\alpha$ for all $C\in\mathcal C_\alpha$;
\item For every $\zeta<\kappa$ and every sequence $\langle A_i\mid i<\theta\rangle$ of cofinal subsets of $\lambda^+$,
there exists $\alpha\in E^{\lambda^+}_\theta\cap\Gamma$ with $\min(C_\alpha)=\zeta$ such that $\sup(\nacc(C_\alpha)\cap A_i)=\alpha$ for all $i<\theta$.
\end{enumerate}
\end{thm}
\begin{proof}  We commence with the following:
\begin{claim} There exists a transversal $\langle C_\alpha\mid\alpha\in\Gamma\rangle$ for $\square_\xi(\lambda^+,{<}\mu,{\sq_\chi})$
and a sequence $\langle S_\theta\mid \theta\in\Theta\rangle$ such that $\Theta$ is a cofinal subset of $\reg(\lambda)$, and for all $\theta\in\Theta$:
\begin{itemize}
\item $S_\theta$ is a stationary subset of $E^{\lambda^+}_\theta\cap\Gamma$;
\item $\min(C_\alpha)=\theta$ for all $\alpha\in S_\theta$;
\item $S_\theta\in J[\lambda^+]$.
\end{itemize}
\end{claim}
\begin{proof} Fix a transversal $\vec C = \langle C_\alpha\mid\alpha\in\Gamma\rangle$ for $\square_\xi(\lambda^+,{<}\mu,{\sq_\chi})$.
\begin{itemize}
\item[$\br$] If $\lambda$ is regular, then for all $\theta\in\reg(\lambda)$ with $\mathcal D(\lambda,\theta)=\lambda$, consider the set $T_\theta:=E^{\lambda^+}_\theta\cap\Gamma$,
which is stationary by Lemma~\ref{Gamma-closure}(3), and note that by \cite[Proposition~2.2]{paper24}, $T_\theta$ is in the ideal $J[\lambda^+]$.
\item[$\br$] If $\lambda$ is singular, then $\chi<\lambda$, and so by \cite[Corollary~2.5]{paper24}, for each $\theta\in\reg(\lambda)\setminus\chi$ with $\mathcal D(\lambda,\theta)=\lambda$,
we may fix a stationary subset $T_\theta\s E^{\lambda^+}_\theta\s \Gamma$ in $J[\lambda^+]$.
\end{itemize}

By $\mu\leq\lambda$ and Lemma \ref{square_is_amenable}, $\vec C$ is amenable.
Thus, by Lemma~\ref{split_amenable2}, there exist a faithful postprocessing function $\Phi:\mathcal K(\lambda^+)\rightarrow\mathcal K(\lambda^+)$ and
a cofinal subset $\Theta\s\reg(\lambda)$ (of course, if $\lambda$ is singular, then $\Theta\s\reg(\lambda)\setminus\chi$) such that
$S_\theta:=\{ \alpha\in T_\theta\mid \min(\Phi(C_\alpha))=\theta\}$ is stationary for all $\theta\in\Theta$.
But $J[\lambda^+]$ is an ideal, and hence $S_\theta\in J[\lambda^+]$ for all $\theta\in\Theta$.
By Lemma~\ref{pp-preserves-square}(2), then, $\langle \Phi(C_\alpha)\mid\alpha\in\Gamma\rangle$ and $\langle S_\theta\mid\theta\in\Theta\rangle$ are as sought.
\end{proof}

Let $\vec C=\langle C_\alpha\mid\alpha\in\Gamma\rangle$ and $\langle S_\theta\mid\theta\in\Theta\rangle$ be given by the preceding claim.
Let $\cvec{C}=\langle\mathcal C_\alpha\mid\alpha<\lambda^+\rangle$ be a $\square_\xi(\lambda^+,{<}\mu,{\sq_\chi})$-sequence for which $\vec C$ is a transversal.
Let $\Phi$ be given by Lemma~\ref{wide-club-guessing} when fed with $\cvec{C}$ and $\mathcal S:=\{ S_\theta\mid \theta\in\Theta\}$.
In particular, for every $\theta\in\Theta$ and every club $E\s\lambda^+$, there exists $\alpha\in S_\theta$ with $\sup(\nacc(\Phi(C_\alpha))\cap E)=\alpha$.
By $\ch_\lambda$ and \cite{Sh:922}, $\diamondsuit(\lambda^+)$ holds.
Thus, for each $\theta\in\Theta$, appeal to Lemma~\ref{thm16} with $\langle \Phi(C_\alpha)\mid\alpha\in S_\theta\rangle$ to obtain a corresponding postprocessing function $\Phi_\theta$.
Define $\Phi':\mathcal K(\lambda^+)\rightarrow\mathcal K(\lambda^+)$ by stipulating:
$$\Phi'(x):=\begin{cases}
\Phi_{\theta}(\Phi(x)),&\text{if }\theta:=\min(x)\text{ is in }\Theta;\\
\Phi(x),&\text{otherwise}.
\end{cases}$$

Since all elements of $\{\Phi,\Phi_\theta\mid \theta\in\Theta\}$ are faithful postprocessing functions, $\Phi'$ is a faithful postprocessing function.

\begin{claim} $\acts{\Phi'}{C}$, $\langle\Phi'(C_\alpha)\mid \alpha\in \Gamma\rangle$  and $\Theta$ are as sought.
\end{claim}
\begin{proof} By Lemma~\ref{pp-preserves-square}(1), $\acts{\Phi'}{C}$ is a $\square_\xi(\lambda^+,{<}\mu,{\sq_\chi})$-sequence with support $\Gamma$, for which $\langle\Phi'(C_\alpha)\mid \alpha\in \Gamma\rangle$ is a transversal.
Write $\langle \mathcal C_\alpha^\bullet\mid\alpha<\lambda^+\rangle$ for $\acts{\Phi'}{C}$.
Let $\theta\in\Theta$ be arbitrary. We now verify the two clauses of the statement of the theorem:
\begin{enumerate}
\item Let $E\s\lambda^+$ be an arbitrary club.
By $S_\theta\in\mathcal S$ and the choice of $\Phi$, we may find $\alpha\in S_\theta\s E^{\lambda^+}_\theta\cap\Gamma$ such that $\sup(\nacc(\Phi(C))\cap E)=\alpha$ for all $C\in\mathcal C_\alpha$.
Let $C^\bullet\in\mathcal C_\alpha^\bullet$ be arbitrary.
As $\alpha\in \Gamma$, pick $C\in\mathcal C_\alpha$ such that $C^\bullet=\Phi'(C)$.
By definition of $\Phi'$, either $C^\bullet = \Phi(C)$ or $C^\bullet = \Phi_\theta(\Phi(C))$ for some $\theta \in \Theta$.
Thus, $\sup(\nacc(\Phi(C))\setminus\nacc(C^\bullet))<\alpha$, where in the second case we appeal to Clause~(1) of Lemma~\ref{thm16}.
Consequently, $\sup(\nacc(C^\bullet)\cap E)=\alpha$.
\item Let $\zeta<\kappa$ and let $\langle A_i\mid i<\theta\rangle$ be an arbitrary sequence of cofinal subsets of $\lambda^+$.
By the choice of $\Phi_\theta$, let us pick $\alpha\in S_\theta\s E^{\lambda^+}_\theta\cap\Gamma$ such that
$\min(\Phi_\theta(\Phi(C_\alpha))=\zeta$ and $\sup(\nacc(\Phi_\theta(\Phi(C_\alpha)))\cap A_i)=\alpha$ for all $i<\theta$.
As $\alpha\in S_\theta$, we have $\min(C_\alpha) = \theta$, so that $\Phi'(C_\alpha)=\Phi_\theta(\Phi(C_\alpha))$, and hence $\alpha$ is as sought.\qedhere
\end{enumerate}
\end{proof}
This completes the proof.
\end{proof}

Theorem~\ref{thm22} now follows as a corollary:
\begin{cor} Suppose that $\lambda$ is an uncountable strong-limit cardinal and $\ch_\lambda$ holds.

Then $\square(\lambda^+,{<}\lambda)$ holds iff there exists a $C$-sequence $\langle C_\alpha\mid \alpha<\lambda^+\rangle$ such that:
\begin{itemize}
\item[(a)] $|\{ C_\alpha\cap\delta\mid \alpha<\lambda^+\ \&\ \sup(C_\alpha\cap\delta)=\delta\}|<\lambda$ for all $\delta<\lambda^+$;
\item[(b)] For every club $D\s\lambda^+$, there exists $\delta\in\acc(D)$ such that for all $\alpha<\lambda^+$ either $\sup(C_\alpha\cap\delta)<\delta$ or $\sup(\nacc(C_\alpha\cap\delta)\cap D)=\delta$;
\item[(c)] For every sequence $\langle A_i\mid i<\theta\rangle$ of cofinal subsets of $\lambda^+$, with $\theta<\lambda$,
the following set is stationary: $$\{\alpha<\lambda^+\mid \sup(\nacc(C_\alpha)\cap A_i)=\alpha\text{ for all }i<\theta\}.$$
\end{itemize}
\end{cor}
\begin{proof} $(\impliedby):$ Fix a $C$-sequence $\langle C_\alpha\mid \alpha<\lambda^+\rangle$ satisfying (a) and (b).
For every $\delta \in \acc(\lambda^+)$, let
\[
\mathcal C_\delta := \{ C_\alpha\cap\delta\mid \alpha<\lambda^+\ \&\ \sup(C_\alpha\cap\delta)=\delta\}.
\]
Let $\mathcal C_0 := \{\emptyset\}$, and let $\mathcal C_{\delta+1} := \{\{\delta\}\}$ for every $\delta < \lambda^+$.
We will show that $\langle \mathcal C_\delta \mid \delta<\lambda^+ \rangle$ is a $\square(\lambda^+,{<}\lambda)$-sequence.
Of course, $|\mathcal C_\delta| <\lambda$ follows from Clause~(a), and $\sq$-coherence is clear from the definition.

Finally, suppose that $A$ is some cofinal subset of $\lambda^+$, and we shall find some $\delta\in\acc^+(A)$ such that $A\cap\delta\notin\mathcal C_\delta$.
Consider the club $E:=\acc^+(A)$.
By Clause~(b), fix $\delta\in\acc(E)$ such that for all $\alpha<\lambda^+$, either $\sup(C_\alpha\cap\delta)<\delta$ or $\sup(\nacc(C_\alpha\cap\delta)\cap E)=\delta$.
Towards a contradiction, suppose that $A\cap\delta \in \mathcal C_\delta$.
Then we can fix some $\alpha<\lambda^+$ such that $A \cap \delta = C_\alpha \cap \delta$ and $\sup(C_\alpha\cap\delta)=\delta$.
But then $\delta = \sup(\nacc(C_\alpha\cap\delta)\cap E) = \sup(\nacc(A \cap\delta)\cap \acc^+(A))$, contradicting the fact that $\nacc(A) \cap \acc^+(A) = \emptyset$.

$(\implies):$ Suppose that $\square(\lambda^+,{<}\lambda)$ holds.
Since $\lambda$ is a strong-limit cardinal, $\{\theta\in\reg(\lambda)\mid \mathcal D(\lambda,\theta)=\lambda\}$ is cofinal in $\reg(\lambda)$.
So, by Theorem~\ref{mixing_paper24} with $(\xi,\mu,\chi):=(\lambda^+,\lambda,\aleph_0)$,
there exists a $\square(\lambda^+,{<}\lambda)$-sequence $\cvec{C}=\langle \mathcal C_\alpha\mid\alpha<\lambda^+\rangle$ with a transversal $\langle C_\alpha\mid \alpha\in\acc(\lambda^+)\rangle$, such that:
\begin{enumerate}
\item For every club $D\s\lambda^+$, there exists $\delta\in \acc(\lambda^+)$ such that $\sup(\nacc(C)\cap D)=\delta$ for all $C\in\mathcal C_\delta$;
\item For every sequence $\langle A_i\mid i<\lambda\rangle$ of cofinal subsets of $\lambda^+$, and every $\theta<\lambda$,
there exists $\alpha\in \acc(\lambda^+)$ such that $\sup(\nacc(C_\alpha)\cap A_i)=\alpha$ for all $i<\theta$.
\end{enumerate}

Let $C_0 := \{\emptyset\}$ and $C_{\alpha+1} := \{\alpha\}$ for every $\alpha<\lambda^+$.
We now verify that $\langle C_\alpha\mid \alpha < \lambda^+ \rangle$ is as sought:

(a) This Clause is witnessed by $\cvec{C}$ (cf.\ Proposition~\ref{transversal-width}(1)).

(b) Given a club $D\s\lambda^+$, appeal to Clause~(1) to find $\delta\in \acc(\lambda^+)$ such that $\sup(\nacc(C)\cap D)=\delta$ for all $C\in\mathcal C_\delta$.
Clearly, $\delta\in\acc(D)$. Now, for all $\alpha<\lambda^+$, if $\sup(C_\alpha\cap\delta)=\delta$, then $C:=C_\alpha\cap\delta$ is in $\mathcal C_\delta$, and hence $\sup(\nacc(C)\cap D)=\delta$.

(c) Let $\langle A_i\mid i<\theta\rangle$ be an arbitrary sequence of cofinal subsets of $\lambda^+$, with $\theta<\lambda$.
To see that $S:=\{\alpha<\lambda^+\mid \sup(\nacc(C_\alpha)\cap A_i)=\alpha\text{ for all }i<\theta\}$ is stationary, let $D$ be an arbitrary club in $\lambda^+$.
For all $i<\theta$, put $A_i':=A_i$, and for all $i\in[\theta,\lambda)$, put $A_i':=D$. By appealing to Clause~(2) with $\langle A_i'\mid i<\lambda\rangle$ and $\theta':=\theta^+$,
let us pick $\alpha\in\acc(\lambda^+)$ such that $\sup(\nacc(C_\alpha)\cap A_i')=\alpha$ for all $i<\theta'$.
Clearly, $\alpha\in D\cap S$.
\end{proof}

The proof of Theorem~\ref{mixing_paper24}
makes it clear that the following holds, as well.
\begin{thm}\label{transversal_of_mixing_paper24} Suppose that $\lambda$ is an uncountable cardinal,  $\ch_\lambda$ holds,  $\Theta$ is a subset of $\{\theta\in\reg(\lambda)\mid \mathcal D(\lambda,\theta)=\lambda\}$,
and $\langle C_\alpha\mid \alpha<\lambda^+, \cf(\alpha)\in\Theta\rangle$ is an amenable $C$-sequence.

Then there exists a faithful postprocessing function $\Phi:\mathcal K(\lambda^+)\rightarrow\mathcal K(\lambda^+)$ such that for cofinally many $\theta\in\Theta$,
for every $\zeta<\kappa$ and every sequence $\langle A_i\mid i<\theta\rangle$ of cofinal subsets of $\lambda^+$,
there exists $\alpha\in E^{\lambda^+}_\theta$ with $\min(\Phi(C_\alpha))=\zeta$ such that $\sup(\nacc(\Phi(C_\alpha))\cap A_i)=\alpha$ for all $i<\theta$.\qed
\end{thm}

\subsection{Trees derived from walks on ordinals}
Let us recall some of the basic characteristic functions surrounding walks on ordinals:

\begin{defn}[Todorcevic, \cite{MR908147},\cite{MR2355670}]\label{walks} Given a $C$-sequence $\vec D=\langle D_\alpha\mid\alpha\in\acc(\kappa)\rangle$, define
$\Tr^{\vec D}:[\kappa]^2\rightarrow{}^\omega\kappa$, $\rho_2^{\vec D}:[\kappa]^2\rightarrow\omega$,
$\rho_1^{\vec D}:[\kappa]^2\rightarrow\kappa$ and $\rho_0^{\vec D}:[\kappa]^2\rightarrow{}^{<\omega}\kappa$
as follows.
For all  $\beta<\alpha<\kappa$, let
\begin{itemize}
\item $\Tr^{\vec D}(\beta,\alpha)(n):=\begin{cases}
\alpha,&n=0;\\
\min(D_{\Tr^{\vec D}(\beta,\alpha)(n-1)}\setminus\beta),&n>0\ \&\ \Tr^{\vec D}(\beta,\alpha)(n-1)\in\acc(\kappa\setminus(\beta+1));\\
\epsilon,&n>0\ \&\ \Tr^{\vec D}(\beta,\alpha)(n-1)=\epsilon+1$, where $\epsilon\ge\beta;\\
\beta,&\text{otherwise};
\end{cases}$
\item $\rho_2^{\vec D}(\beta,\alpha):=\min\{ n<\omega\mid \Tr^{\vec D}(\beta,\alpha)(n)=\beta\}$;
\item $\rho_1^{\vec D}(\beta,\alpha):=\max(\rng(\rho_0^{\vec D}(\beta,\alpha)))$, where
\item $\rho_0^{\vec D}(\beta,\alpha):=\langle \otp(D_{\Tr^{\vec D}(\beta,\alpha)(i)}\cap\beta)\mid i<\rho_2^{\vec D}(\beta,\alpha)\rangle$.
\end{itemize}
\end{defn}

For any function $c:[\kappa]^2\rightarrow \Xi$ and any $\alpha<\kappa$, we denote by $c({\cdot},\alpha)$ the unique function from $\alpha$ to $\Xi$
satisfying $c({\cdot},\alpha)(\beta)=c(\beta,\alpha)$ for all $\beta<\alpha$.
Then, the \emph{tree induced by $c$} is $$\mathcal T(c):=\{c({\cdot},\gamma)\restriction\beta\mid \beta\le\gamma<\kappa\}.$$

\begin{lemma}\label{lemma410} Suppose that $\diamondsuit(\kappa)$ holds.
Suppose that $\cvec{C}=\langle \mathcal C_\alpha\mid\alpha<\kappa\rangle$ is a $\square(\kappa,{<}\kappa)$-sequence satisfying that
for every club $E\s\kappa$, there exists  $\alpha\in\acc(E)$ with $\sup(\nacc(C)\cap E)=\alpha$ for all $C\in\mathcal C_\alpha$.

To any transversal $\langle C_\alpha\mid \alpha\in\acc(\kappa)\rangle$  for $\cvec{C}$,  there exists a corresponding $C$-sequence $\vec D$ over $\acc(\kappa)$ satisfying the following:
\begin{itemize}
\item $(\mathcal T(\rho_0^{\vec D}),{\stree})$ is a $\kappa$-Aronszajn tree;
\item For any infinite cardinal $\theta$, if for every $\zeta<\kappa$ and every sequence $\langle A_i\mid i<\theta\rangle$ of cofinal subsets of $\kappa$,
there exists $\alpha\in\acc(\kappa)$ such that $\min(C_\alpha)=\zeta$ and $\sup(\nacc(C_\alpha)\cap A_i)=\alpha$ for all $i<\theta$,
then $(\mathcal T(\rho_0^{\vec D}),{\stree})$ is normal and $\theta$-distributive.
\end{itemize}
\end{lemma}
\begin{proof} Fix $\cvec{C}$ as in the hypothesis, along with some transversal $\vec C=\langle C_\alpha\mid\alpha\in\acc(\kappa)\rangle$.

For each $\alpha \in \acc(\kappa)$, let $\{ C_\alpha^i \mid i<\kappa\}$ be some enumeration (with repetition) of $\mathcal C_\alpha$ such that $C_\alpha^0=C_\alpha$.
Let $\langle X_\beta\mid\beta<\kappa\rangle$ be a $\diamondsuit(\kappa)$-sequence.
We shall now define a matrix $\langle D^i_\alpha\mid \alpha \in \acc(\kappa), i<\kappa\rangle$ by recursion over $\alpha \in \acc(\kappa)$.

Fix $\alpha\in\acc(\kappa)$ and suppose that $\langle D^i_\gamma\mid \gamma \in \acc(\alpha),i<\kappa\rangle$ has already been defined.
Fix $i<\kappa$. To determine $D^i_\alpha$, we recursively define an $\sq$-increasing sequence $\langle D^i_{\alpha,\beta}\mid \beta\in C^i_\alpha\rangle$ in such a way that for all $\beta\in C^i_\alpha$,
$D^i_{\alpha,\beta}$ is a nonempty closed set of ordinals satisfying $\max(D^i_{\alpha,\beta})=\beta$. Here goes:
\begin{itemize}
\item[$\br$] For $\beta=\min(C^i_\alpha)$, we consider three alternatives:
\begin{itemize}
\item[$\br\br$] If $\beta=0$, then let $D^i_{\alpha,\beta}=\{\beta\}$;
\item[$\br\br$] If $\beta$ is a successor ordinal, say, $\beta=\epsilon+1$, then let $D^i_{\alpha,\beta}:=\{\epsilon,\beta\}$;
\item[$\br\br$] Otherwise, let $D^i_{\alpha,\beta}:=D^0_\beta\cup\{\beta\}$.
\end{itemize}

\item[$\br$] For $\beta\in\nacc(C^i_\alpha) \setminus \{\min(C^i_\alpha)\}$ such that $D^i_{\alpha,\beta^-}$ has already been defined, where $\beta^-:=\sup(C^i_\alpha\cap\beta)$, we consider two alternatives:
\begin{itemize}
\item[$\br\br$] If there exists some $\gamma\in X_\beta \cap \acc(\beta)$ such that $D^i_{\alpha,\beta^-}\sq D_\gamma^0$, then let $D^i_{\alpha,\beta}:=D^0_\gamma\cup\{\gamma,\beta\}$ for the least such $\gamma$;
\item[$\br\br$] Otherwise, let $D^i_{\alpha,\beta}:=D^i_{\alpha,\beta^-}\cup\{\beta\}$.
\end{itemize}
\item[$\br$] For $\beta\in\acc(C^i_\alpha)$, let $D^i_{\alpha,\beta}:=(\bigcup_{\gamma\in C^i_\alpha\cap\beta}D_{\alpha,\gamma}^i)\cup\{\beta\}$.
\end{itemize}
Having constructed $\langle D^i_{\alpha,\beta}\mid \beta\in C^i_\alpha\rangle$, we put $D^i_\alpha:=\bigcup_{\beta\in C_\alpha^i}D^i_{\alpha,\beta}$.

\begin{claim} $\vec D:=\langle D_\alpha^0\mid \alpha\in\acc(\kappa)\rangle$ is a transversal for $\square(\kappa,{<}\kappa)$.
\end{claim}
\begin{proof} Let $\mathcal D_0 := \{\emptyset\}$, $\mathcal D_{\alpha+1} := \{\{\alpha\}\}$ for all $\alpha<\kappa$, and $\mathcal D_\alpha := \{ D^i_\alpha \mid i<\kappa \}$ for all $\alpha \in \acc(\kappa)$.
We shall show that $\langle\mathcal D_\alpha\mid\alpha<\kappa\rangle$ is a $\square(\kappa,{<}\kappa)$-sequence.

First, notice that for all $\alpha\in\acc(\kappa)$ and $i<i'<\kappa$, we have:
\begin{itemize}
\item $D_\alpha^i$ is a club in $\alpha$;
\item $D_\alpha^i=D_\alpha^{i'}$ whenever $C_\alpha^i=C_\alpha^{i'}$;
\item Each $\beta\in \nacc(C_\alpha^i)\setminus\{\min(C_\alpha^i)\}$ is in $\nacc(D_\alpha^i)$.
\end{itemize}

In particular, $0<|\mathcal D_\alpha|\le|\mathcal C_\alpha|<\kappa$ for all $\alpha<\kappa$.
Next, let us show that for all $\alpha\in \acc(\kappa)$, $i<\kappa$, and $\bar\alpha\in\acc(D^i_\alpha)$, we have $D^i_\alpha\cap\bar\alpha\in\mathcal D_{\bar\alpha}$.
Suppose not, and let $\alpha \in \acc(\kappa)$ be the least counterexample. Fix $i<\kappa$ and $\bar\alpha\in\acc(D^i_\alpha)$ with $D^i_\alpha\cap\bar\alpha\notin\mathcal D_{\bar\alpha}$.
If $\bar\alpha\in\acc(C^i_\alpha)$, then fix $j<\kappa$ such that $C^i_\alpha\cap\bar\alpha=C^j_{\bar\alpha}$.
But, then, due to the uniform nature of the construction, we have $D^i_{\alpha,\beta} = D^j_{\bar\alpha,\beta}$ for all $\beta \in C^j_{\bar\alpha}$.
Consequently, $D^i_\alpha\cap\bar\alpha=D^j_{\bar\alpha}$, contradicting the fact that $D^j_{\bar\alpha}\in\mathcal D_{\bar\alpha}$.
Thus, it must be the case that $\bar\alpha\in\acc(D^i_\alpha)\setminus\acc(C^i_\alpha)$.
Put $\beta:=\min(C^i_\alpha\setminus\bar\alpha)$, so that $\beta$ is an element of $\nacc(C^i_\alpha)$.
Now, there are two options to consider, each yielding a contradiction:
\begin{itemize}
\item[$\br$] If $\beta=\min(C^i_\alpha)$, then   $D^i_\alpha\cap\beta=D^i_{\alpha,\beta}\cap\beta=D^0_\beta$.
But then either $\bar\alpha=\beta$, so that $D^i_\alpha \cap\bar\alpha = D^0_{\bar\alpha}$, contradicting the fact that $D^0_{\bar\alpha} \in \mathcal D_{\bar\alpha}$, or $\bar\alpha \in \acc(D^0_\beta)$,
so that $\beta$ contradicts the minimality of $\alpha$.
\item[$\br$] Otherwise, there must exist some $\gamma\in X_\beta \cap \acc(\beta)$ such that $D_\alpha^i\cap\bar\alpha = D^i_{\alpha,\beta} \cap\bar\alpha =D_\gamma^0\cap\bar\alpha$.
But then either $\bar\alpha=\gamma$, so that $D^i_\alpha \cap\bar\alpha = D^0_{\bar\alpha}$, contradicting the fact that $D^0_{\bar\alpha} \in \mathcal D_{\bar\alpha}$,
or $\bar\alpha \in \acc(D^0_\gamma)$, so that $\gamma$ contradicts the minimality of $\alpha$.
\end{itemize}

Finally, suppose that $A$ is some cofinal subset of $\kappa$, and we shall find some $\alpha\in\acc^+(A)$ such that $A\cap\alpha\notin\mathcal D_\alpha$.
Consider the club $E:=\acc^+(A)$. By the hypothesis of the Lemma, fix $\alpha\in\acc(E)$ such that $\sup(\nacc(C_\alpha^i)\cap E)=\alpha$ for all $i<\kappa$.
Then $\nacc(D^i_\alpha)\cap\acc^+(A)\neq\emptyset$ for all $i<\kappa$, so that $A\cap\alpha\notin\mathcal D_\alpha$.
\end{proof}

By \cite[\S6]{MR2355670}, $(\mathcal T(\rho_0^{\vec D}),{\stree})$ is a $\kappa$-Aronszajn  tree iff $\vec D$ is a transversal for $\square(\kappa,{<}\kappa)$.
Thus, we are left with proving the following.

\begin{claim}Suppose that $\theta$ is an infinite cardinal such that for every $\zeta<\kappa$ and every sequence $\langle A_i\mid i<\theta\rangle$ of cofinal subsets of $\kappa$,
there exists $\alpha\in\acc(\kappa)$ with $\min(C_\alpha)=\zeta$ and $\sup(\nacc(C_\alpha)\cap A_i)=\alpha$ for all $i<\theta$.
Then $(\mathcal T(\rho_0^{\vec D}),{\stree})$ is normal and $\theta$-distributive.
\end{claim}
\begin{proof} For notational simplicity, denote $T_\alpha:=\{s\in \mathcal T(\rho_0^{\vec D})\mid \dom(s)=\alpha\}$, and $T\restriction\beta:=\bigcup_{\alpha<\beta}T_\alpha$.

Let $\langle \Omega_i\mid i<\theta\rangle$ be an arbitrary sequence of dense open subsets of $(\mathcal T(\rho_0^{\vec D}),{\stree})$,
let $r$ be an arbitrary element of $\mathcal T(\rho_0^{\vec D})$, and let $\varepsilon$ be an arbitrary ordinal $<\kappa$.
We shall prove that $(\bigcap_{i<\theta} \Omega_i)\setminus(T\restriction\varepsilon)$ contains an element that extends $r$.

Let $i<\theta$ be arbitrary.
Define a function $f_i:\mathcal T(\rho_0^{\vec D})\rightarrow\kappa$ by stipulating:
$$f_i(s):=\min\{\gamma<\kappa\mid s\stree \rho_0^{\vec D}({\cdot},\gamma)\in\Omega_i\}.$$
To see that $f_i$ is well-defined, let $s\in \mathcal T(\rho_0^{\vec D})$ be arbitrary.
Since $\Omega_i$ is dense, there exists $t\in\Omega_i$ with $s\stree t$.
By definition of $\mathcal T(\rho_0^{\vec D})$, any such $t$ is of the form $\rho_0^{\vec D}({\cdot},\gamma)\restriction\beta$ for some $\beta\leq\gamma<\kappa$.
Since $\Omega_i$ is open and $s\stree t\stree \rho_0^{\vec D}({\cdot},\gamma)$, then we have found $\gamma<\kappa$ such that $\rho_0^{\vec D}({\cdot},\gamma)\in \Omega_i$ and $s\stree \rho_0^{\vec D}({\cdot},\gamma)$.

Since $(\mathcal T(\rho_0^{\vec D}),{\stree})$ is a $\kappa$-tree, $|T\restriction\beta|<\kappa$ for all $\beta<\kappa$, and hence $E_i:=\{\beta \in \acc(\kappa) \mid f_i[T\restriction\beta]\s\beta\}$ is a club in $\kappa$.
Consequently, the following set is stationary in $\kappa$:
$$A_i:=\{ \beta\in E_i\setminus\varepsilon\mid X_\beta=\rng(f_i)\cap\beta\}.$$

Put $\zeta:=f_0(r)$.
Now, by the hypothesis of the Claim, let us pick some $\alpha \in \acc(\kappa)$ such that $\min(C^0_\alpha)=\zeta$ and $\sup(\nacc(C^0_\alpha)\cap A_i)=\alpha$ for all $i<\theta$.
In particular, $\alpha>\varepsilon$.

We now show that $\rho_0^{\vec D}({\cdot},\alpha)$ extends $r$. Recalling the definition of $D_\alpha^0$ and $\rho_0^{\vec D}({\cdot},\alpha)$, we see that:
\begin{itemize}
\item[$\br$] If $\zeta=0$, then $r=\emptyset$, and trivially $r\stree \rho_0^{\vec D}({\cdot},\alpha)$;
\item[$\br$] If $\zeta=\epsilon+1$, then $D_\alpha^0\cap\zeta=D^0_{\alpha,\zeta}\cap\zeta=\{\epsilon\}$, and hence $r\stree \rho_0^{\vec D}({\cdot},\zeta)\stree \rho_0^{\vec D}({\cdot},\alpha)$;
\item[$\br$] If $\zeta\in\acc(\kappa)$, then $D^0_\alpha\cap\zeta=D^0_{\alpha,\zeta}\cap\zeta=D^0_\zeta$, and hence $r\stree \rho_0^{\vec D}({\cdot},\zeta)\stree \rho_0^{\vec D}({\cdot},\alpha)$.
\end{itemize}

Finally, we show that $\rho_0^{\vec D}({\cdot},\alpha)$ is an element of $\bigcap_{i<\theta}\Omega_i$.
Let  $i<\theta$ be arbitrary.
Pick a large enough $\beta\in\nacc(C^0_\alpha)\cap A_i$ for which $\beta^-:=\sup(C^0_\alpha\cap\beta)$ is greater than $\zeta$.
Put $\gamma:=f_i(s)$, for $s:=\rho_0^{\vec D}({\cdot},\alpha)\restriction(\beta^-+1)$.
Note that since $\beta\in A_i$, we have $\gamma\in X_\beta\setminus(\beta^-+1)$.

Let $\beta^0:=\min(D^0_\alpha)$, so that $\beta^0\le\zeta$, $\Tr^{\vec D}(\beta^0,\alpha)(1)=\beta^0$ and $s(\beta^0)=\rho_0^{\vec D}(\beta^0,\alpha)=\langle 0\rangle$.
As $\gamma=f_i(s)$, in particular, $\rho_0^{\vec D}(\beta^0,\gamma)=\langle0\rangle$, and hence $\Tr^{\vec D}(\beta^0,\gamma)(1)=\beta^0$, as well.
Since $\beta^0+1<\gamma$ (indeed, $\beta^0\le\zeta<\beta^-<\gamma$) and $\Tr^{\vec D}(\beta^0,\gamma)(1)=\beta^0$, $\gamma$ is not a successor ordinal.
It then follows from $\alpha,\gamma\in\acc(\kappa)$ and $\rho_0^{\vec D}({\cdot},\alpha)\restriction(\beta^-+1)\stree \rho_0^{\vec D}({\cdot},\gamma)$,
that for every $\beta'\le\beta^-$: $$\beta'\in D_\alpha^0 \iff |\rho_0^{\vec D}(\beta',\alpha)|=1 \iff |\rho_0^{\vec D}(\beta',\gamma)|=1 \iff \beta'\in D^0_\gamma.$$
Altogether, $\gamma\in X_\beta\cap\acc(\beta)$ and $D^0_{\alpha,\beta^-}=D^0_\alpha\cap(\beta^-+1)\sq D_\gamma^0$.
Consequently, $D_{\gamma^\bullet}^0\cup\{\gamma^\bullet,\beta\} = D_{\alpha,\beta}^0\sq D^0_\alpha$ for some $\gamma^\bullet \in \im(f_i)\cap\acc(\beta)$.
Fix such a $\gamma^\bullet$.
Then  $\rho_0^{\vec D}({\cdot},\gamma^\bullet)\stree \rho_0^{\vec D}({\cdot},\alpha)$ and $\rho_0^{\vec D}({\cdot},\gamma^\bullet)\in \Omega_i$.
But $\Omega_i$ is open, and hence $\rho_0^{\vec D}({\cdot},\alpha)\in \Omega_i$.
\end{proof}
This completes the proof.
\end{proof}

We are now ready to prove the first case of Theorem~A.

\begin{cor}\label{first-case-A} Suppose that $\lambda$ is a strong-limit singular cardinal, and $\square(\lambda^+,{<}\lambda,{\sq_{(\cf(\lambda))^+}})$ and $\ch_\lambda$ both hold.
Then there is a $C$-sequence $\vec D$ for which $(\mathcal T(\rho_0^{\vec D}),{\stree})$ is a normal $\lambda$-distributive $\lambda^+$-Aronszajn tree.
\end{cor}
\begin{proof} Since $\lambda$ is a strong-limit cardinal, $D(\lambda,\theta)=\lambda$ for every $\theta \in \reg(\lambda)\setminus\{\cf(\lambda)\}$.
Thus, by appealing to Theorem~\ref{mixing_paper24} with $(\xi,\mu,\chi):=(\lambda^+,\lambda,(\cf(\lambda))^+)$, we can fix a $\square(\lambda^+,{<}\lambda,{\sq_{\cf(\lambda)^+}})$-sequence $\langle \mathcal C_\alpha\mid\alpha<\lambda^+\rangle$ with
a transversal $\langle C_\alpha\mid \alpha\in\Gamma\rangle$, and a cofinal subset $\Theta\s\reg(\lambda)\setminus\cf(\lambda)$ such that for every $\theta\in\Theta$:
\begin{enumerate}
\item For every club $E\s\lambda^+$, there exists $\alpha\in E^{\lambda^+}_\theta\cap\Gamma$ such that $\sup(\nacc(C)\cap E)=\alpha$ for all $C\in\mathcal C_\alpha$;
\item For every $\zeta<\kappa$ and every sequence $\langle A_i\mid i<\theta\rangle$ of cofinal subsets of $\lambda^+$, there exists $\alpha\in E^{\lambda^+}_\theta\cap\Gamma$ such that $\min(C_\alpha)=\zeta$ and $\sup(\nacc(C_\alpha)\cap A_i)=\alpha$ for all $i<\theta$.
\end{enumerate}

For all $\alpha<\lambda^+$, let
$$\mathcal C_\alpha':=\mathcal C_\alpha\cup\{x\in[\alpha]^{<\cf(\lambda)}\mid x\text{ is a club in }\alpha\}.$$

By Lemma~\ref{Gamma-closure}(3) and Definition~\ref{def115}, for all $\alpha\in\lambda^+\setminus\Gamma$, we have $\mathcal C_\alpha=\{e_\alpha\}$ for some club $e_\alpha\s\alpha$ with $\otp(e_\alpha)\le\cf(\lambda)$.
Since $\lambda$ is a strong-limit cardinal, $|[\alpha]^{<\cf(\lambda)}|\leq \lambda$ for every $\alpha<\lambda^+$.
Consequently, $\langle\mathcal C_\alpha'\mid\alpha<\lambda^+\rangle$ is a $\square(\lambda^+,\lambda)$-sequence.
As $\mathcal C_\alpha'=\mathcal C_\alpha$ for all $\alpha\in E^{\lambda^+}_{\ge\cf(\lambda)}$, we get that Clause~(1) remains true after replacing $\mathcal C_\alpha$ by $\mathcal C_\alpha'$.
By $\ch_\lambda$ and \cite{Sh:922}, $\diamondsuit(\lambda^+)$ holds.
It now follows from Lemma~\ref{lemma410} that there exists a $C$-sequence $\vec D$ for which $(\mathcal T(\rho_0^{\vec D}),{\stree})$ is a normal $\lambda^+$-Aronszajn tree that is $\theta$-distributive for all $\theta\in\Theta$.
As $\sup(\Theta)=\lambda$ and $\lambda$ is a singular cardinal, it follows that $(\mathcal T(\rho_0^{\vec D}),{\stree})$ is $\lambda$-distributive.
\end{proof}

\begin{remark} The hypothesis $\sq_{\chi}$ with $\chi=(\cf(\lambda))^+$ in the preceding is optimal, since Hayut and Magidor \cite[\S3]{arXiv:1603.05526} proved that
$\square_{\aleph_\omega}(\aleph_{\omega+1},{<}\aleph_\omega,{\sq_{\aleph_2}})+\gch$ is consistent with the non-existence of an $\aleph_{\omega+1}$-Aronszajn tree.
\end{remark}

\section{$C$-sequences of small order-types}\label{section3}

The postprocessing function $\Phi^{\{j\}}$ of Example~\ref{chop-bottom} is a special case of the following:

\begin{fact}[\cite{paper23}]\label{newPhiSigma} For a set $\Sigma$  satisfying $\acc^+(\Sigma)\s\Sigma\s (\kappa+1)$, define $\Phi^\Sigma : \mathcal K(\kappa) \to \mathcal K(\kappa)$ by:
\[\Phi^\Sigma(x) := \begin{cases}
\{ x(i) \mid i \in \Sigma \cap \otp(x) \},&\text{if } \otp(x) = \sup(\Sigma \cap \otp(x)); \\
x \setminus x(\sup(\Sigma \cap \otp(x))), &\text{otherwise}.
\end{cases}\]

Then $\Phi^\Sigma : \mathcal K(\kappa) \to \mathcal K(\kappa)$ is a conservative postprocessing function satisfying for every $x \in \mathcal K(\kappa)$:
\begin{enumerate}
\item\label{PhiSigma3} If $\otp(x) = \sup(\Sigma)$ and $\otp(\Sigma)$ is a limit ordinal, then $\otp(\Phi^\Sigma(x)) = \otp(\Sigma)$;
\item\label{PhiSigma9} If $\otp(x) > \min(\Sigma)$, then $\Phi^\Sigma(x) \subseteq x \setminus x(\min(\Sigma))$.
\end{enumerate}
\end{fact}

The following is a minor variation of Lemma~\ref{phiZ}:
\begin{fact}\label{Psi} Suppose that $\Phi:\mathcal K(\kappa)\rightarrow\mathcal K(\kappa)$ is some postprocessing function,
and $\mathfrak Z=\langle Z_x \mid x \in \mathcal K(\kappa) \rangle$ is some indexed collection satisfying for all $x\in\mathcal K(\kappa)$:
\begin{itemize}
\item $Z_x \subseteq \sup(x)$;
\item $Z_x \cap\bar\alpha=Z_{x\cap\bar\alpha}$ for all $\bar\alpha\in\acc(\Phi(x))$.
\end{itemize}
Define $\Phi^{\mathfrak Z} : \mathcal K(\kappa) \to \mathcal K(\kappa)$ by stipulating
$\Phi^{\mathfrak Z}(x) := \rng(g^\Phi_{x,Z_x})$, where $g^\Phi_{x,Z_x} : \Phi(x) \to \sup(x)$ is defined by:
\[
g^\Phi_{x,Z_x}(\beta) := \min((Z_x \cup\{\beta\}) \setminus \sup(\Phi(x) \cap\beta)).
\]
Then $\Phi^{\mathfrak Z}$ is also a postprocessing function, and for every $x \in \mathcal K(\kappa)$, $\nacc(\Phi^{\mathfrak Z}(x)) \subseteq g^\Phi_{x,Z_x}[\nacc(\Phi(x))]$. \qed
\end{fact}

The language of postprocessing functions allows us to formalize the opening remark of \cite[\S6]{rinot11} as follows:

\begin{lemma}\label{phi1} Suppose that $\lambda$ is an uncountable cardinal, $\ch_\lambda$ holds, and $\langle C_\alpha\mid \alpha \in S \rangle$ is a $C$-sequence
over some subset $S$ of $E^{\lambda^+}_{\neq\cf(\lambda)}$ for which $\{ \alpha\in S\mid \otp(C_\alpha)<\alpha\}$ is stationary.

Then there exists a faithful postprocessing function $\Phi:\mathcal K(\lambda^+)\rightarrow\mathcal K(\lambda^+)$ satisfying the following.
For every cofinal $A\s\lambda^+$, there exist stationarily many $\alpha\in S \cap \acc(\lambda^+)$ such that $\otp(\Phi(C_\alpha))=\cf(\alpha)$, $\nacc(\Phi(C_\alpha))\s A$ and $\min(\Phi(C_\alpha))=\min(A)$.
\end{lemma}
\begin{proof} By Fodor's lemma, let us fix some $\epsilon<\lambda^+$ for which $\{\alpha\in S\mid \otp(C_\alpha)=\epsilon\}$ is stationary.
By passing to a stationary subset of $S$, we may simply assume that $S \subseteq \{ \alpha \in \acc(\lambda^+) \mid \otp(C_\alpha)=\epsilon \}$.

Put $\chi:=\cf(\epsilon)$, so that $\chi\in \reg(\lambda) \setminus\{\cf(\lambda)\}$.
Let $\Sigma$ be some club in $\epsilon$ of order-type $\chi$.
Let $\Phi^\Sigma$ be given by Fact~\ref{newPhiSigma}, so that $\otp(\Phi^\Sigma(C_\alpha))=\chi$ for all $\alpha\in S$, as seen by Fact~\ref{newPhiSigma}(\ref{PhiSigma3}).

Suppose first that $\chi=\aleph_0$. By $\ch_\lambda$, $S\s E^{\lambda^+}_{\neq\cf(\lambda)}$, the fact that $S$ is stationary, and the main result of \cite{Sh:922}, let $\langle S_\alpha\mid\alpha\in S\rangle$ be a $\diamondsuit(S)$-sequence.
For all $\alpha\in S$ such that $\sup(S_\alpha)=\alpha$, let $\Omega_\alpha$ be a cofinal subset of $S_\alpha$ satisfying $\otp(\Omega_\alpha)=\omega$ and $\min(\Omega_\alpha)=\min(S_\alpha)$.
Define $\Phi':\mathcal K(\lambda^+)\rightarrow\mathcal K(\lambda^+)$ by stipulating:
$$\Phi'(x):=\begin{cases}
\Omega_{\sup(x)},&\text{if } \otp(x) = \omega\ \&\ \sup(x)\in S\ \&\ \sup(S_{\sup(x)})=\sup(x);\\
x\setminus x(\omega),&\text{if }\otp(x)>\omega;\\
x,&\text{otherwise}.
\end{cases}$$

Let $\faithful$ be given by Example~\ref{faithful_correction}.

\begin{claim} $\Phi^\bullet := \faithful\circ \Phi'\circ \Phi^\Sigma$ is as sought.
\end{claim}
\begin{proof} We know that $\Phi^\Sigma$ and $\faithful$ are postprocessing functions.
Thus, to show that $\Phi^\bullet$ is a postprocessing function, it suffices to prove that $\Phi'$ is a postprocessing function.
Fix arbitrary $x \in \mathcal K(\lambda^+)$.
It is clear from each case of the definition that $\Phi'(x)$ is a club in $\sup(x)$.
In the first and third cases, we have $\otp(\Phi'(x)) = \omega$, so that $\acc(\Phi'(x)) = \emptyset$.
Thus, it remains to consider the case $\otp(x) > \omega$.
In this case, $\Phi'(x) \subseteq x$, so that clearly $\acc(\Phi'(x)) \subseteq \acc(x)$.
Consider any $\bar\alpha \in \acc(\Phi'(x))$, in order to compare $\Phi'(x) \cap\bar\alpha$ with $\Phi'(x\cap\bar\alpha)$.
Then $\bar\alpha \in \acc(x \setminus x(\omega))$, so that $\otp(x \cap \bar\alpha) > \omega$, and it follows that
\[
\Phi'(x \cap \bar\alpha) = (x \cap\bar\alpha) \setminus ((x \cap\bar\alpha)(\omega))
= (x \cap\bar\alpha) \setminus x(\omega) = (x \setminus x(\omega)) \cap \bar\alpha = \Phi'(x) \cap \bar\alpha,
\]
as required.

Next, suppose that $A$ is a cofinal subset of $\lambda^+$.
As $\acc^+(A)$ is a club in $\lambda^+$, and $\langle S_\alpha\mid\alpha\in S\rangle$ is a $\diamondsuit(S)$-sequence, there are stationarily many $\alpha\in S\cap\acc^+(A)$ such that $A\cap\alpha=S_\alpha$.
Consider any such $\alpha$.
Then $\otp(\Phi^\Sigma(C_\alpha)) = \omega$.
As $\sup(S_\alpha) = \sup(A \cap \alpha) = \alpha$, we have $\Phi'(\Phi^\Sigma(C_\alpha)) = \Omega_\alpha$.
But $\faithful$ is conservative and $\min$-preserving, and hence $\Phi^\bullet(C_\alpha)$ is a cofinal subset of $S_\alpha \subseteq A$ of order-type $\omega=\cf(\alpha)$,
and $\min(\Phi^\bullet(C_\alpha)) = \min(S_\alpha) = \min(A)$, as required.
\end{proof}

From here on, suppose that $\chi>\aleph_0$. We shall follow the arguments of \cite[$\S2$]{rinot11}.

By applying Fact~\ref{clubguessing}(1) to the sequence $\langle \Phi^\Sigma(C_\alpha)\mid \alpha \in S \rangle$, let us fix a club $D\s\lambda^+$ such that for every club $E\s\lambda^+$,
the set $\{ \alpha\in S\mid \Phi_D(\Phi^\Sigma(C_\alpha))\s E \}$ is stationary, where $\Phi_D$ is the function defined in Lemma~\ref{Phi_D}.
For every $x\in\mathcal K(\lambda^+)$, denote $x^\circ:= \Phi_{D}(\Phi^\Sigma(x))$.

Let $\langle \lambda_j\mid j<\cf(\lambda)\rangle$ be an increasing sequence of ordinals, converging to $\lambda$.
Fix a sequence of injections $\langle \psi_\gamma:\gamma+1\rightarrow\lambda\mid \gamma<\lambda^+\rangle$.
For every $y\in\mathcal K(\lambda^+)$, define an injection $\varphi_y:\sup(y)\rightarrow\lambda^+\times\lambda$ by stipulating
$$\varphi_y(\delta):=(\otp(y\cap\delta),\psi_{\min(y\setminus\delta)}(\delta)),$$
and put $H^j_y:=(\varphi_y^{-1}[\lambda_j\times \lambda_j])^2$ for all $j<\cf(\lambda)$.
Clearly, for all $y \in \mathcal K(\lambda^+)$ with $\otp(y)\le\lambda$, letting $\alpha := \sup(y)$, we have that $\langle H_y^j\mid j<\cf(\lambda) \rangle$
is an $\s$-increasing sequence of elements of $[\alpha\times\alpha]^{<\lambda}$, converging to $\alpha\times\alpha$,
and if $\bar\alpha\in\acc(y)$, then $\varphi_{y \cap \bar\alpha} = \varphi_y \restriction \bar\alpha$,
so that $H^j_{y \cap\bar\alpha} = H^j_y \cap(\bar\alpha\times\bar\alpha)$ for all $j<\cf(\lambda)$.

By $\ch_\lambda$, let $\{X_\gamma\mid \gamma<\lambda^+\}$ be some enumeration of $[\lambda\times\lambda\times\lambda^+]^{\le\lambda}$.
For all $(j,\tau)\in\lambda\times\lambda$ and $X\s\lambda\times\lambda\times\lambda^+$, let $\pi_{j,\tau}(X):=\{\varsigma<\lambda^+\mid (j,\tau,\varsigma)\in X\}$.
For every $j < \cf(\lambda)$, $Y \subseteq \lambda^+ \times \lambda^+$, $y \in \mathcal K(\lambda^+)$, and $\eta < \sup(y)$, define
\[
F^{j,Y}_{y,\eta} := \{ \gamma < \min (y \setminus (\eta+1)) \mid (\eta,\gamma) \in H^j_y \setminus Y \}
\]
and
\[
W^{j,Y}_y := \{ \eta < \sup(y) \mid F^{j,Y}_{y,\eta} \neq \emptyset \}.
\]

\begin{claim} There exist $(j,\tau)\in\cf(\lambda)\times\lambda$ and $Y\s\lambda^+\times\lambda^+$
such that for every club $E\s\lambda^+$ and every subset $Z\s\lambda^+$, there exists some  $\alpha\in S$ such that:
\begin{enumerate}
\item $C_\alpha^\circ\s E$;
\item $H^j_{C_\alpha^\circ}\bks Y\s \{(\eta,\gamma)\mid Z\cap\eta=\pi_{j,\tau}(X_\gamma)\}$;
\item $\sup(\acc^+(W^{j,Y}_{C_\alpha^\circ}) \cap\acc(C_\alpha^\circ))=\alpha$.
\end{enumerate}
\end{claim}
\begin{proof} We know that for every $\alpha\in S$, $C_\alpha^\circ$ is a club in $\alpha$ of order-type $\cf(\alpha)\in\reg(\lambda)\setminus\{\cf(\lambda),\aleph_0\}$,
and that for every club $E\s\lambda^+$, there exists some $\alpha\in S$ with $C_\alpha^\circ\s E$.
Thus, the proof of Claim~2.5.2 of \cite{rinot11} establishes our claim.
\end{proof}

Let $(j,\tau)$ and $Y$ be given by the previous claim.
Let $y\in\mathcal K(\lambda^+)$ be arbitrary.
Denote $W_y := W^{j,Y}_y$,
and define the function $f_y  : W_y \to \sup(y)$ by setting, for all $\eta \in W_y$:
\[
f_y(\eta) := \min(F^{j,Y}_{y,\eta}).
\]

Let $y^*$ be the set of all $\delta\in y$ such that the following properties hold:
\begin{enumerate}
\item[(i)] $\sup( W_y \cap\delta)\ge\sup(y\cap\delta)$;
\item[(ii)] $\pi_{j,\tau}(X_{f_y(\eta)})\s\eta$ for every $\eta\in W_y \cap\delta$;
\item[(iii)] For all $\eta', \eta \in  W_y $, if $\eta'<\eta<\delta$ then $\pi_{j,\tau}(X_{f_y(\eta)})\bks \pi_{j,\tau}(X_{f_y(\eta')})\s[\eta',\eta)$.
\end{enumerate}
Notice that $y^*$ is a closed subset of $y$, so that $y^* \in \mathcal K(\lambda^+)$ whenever $\sup(y^*) = \sup(y)$.

\begin{claim}\label{coherence-y} Suppose $y \in \mathcal K(\lambda^+)$. Then:
\begin{enumerate}
\item If $\bar\alpha \in \acc(y)$, then $F^{j,Y}_{y \cap \bar\alpha, \eta} = F^{j,Y}_{y, \eta}$ for every $\eta < \bar\alpha$,
$W_{y \cap \bar\alpha} = W_y \cap \bar\alpha$, $f_{y \cap \bar\alpha} = f_y \restriction \bar\alpha$,
and $(y \cap \bar\alpha)^* = y^* \cap \bar\alpha$.
\item If $\bar\alpha \in \acc(y^*) \cup \{\sup(y^*)\}$, then $W_{y \cap \bar\alpha}$ is cofinal in $\bar\alpha$.
\item If $\bar\alpha \in \acc^+(W_y)$, $\delta \in y^*$, and $\eta \in W_y \cap \delta \setminus \bar\alpha$, then
\[
\pi_{j,\tau}(X_{f_y(\eta)}) \setminus
\bigcup \{ \pi_{j,\tau}(X_{f_y(\eta')}) \mid {\eta' \in W_y \cap \bar\alpha}  \}
\subseteq [\bar\alpha,\eta).
\]
\end{enumerate}
\end{claim}
\begin{proof} (1)--(2) are proved in the same way as Claim~2.5.3 of~\cite{rinot11}.

(3) By $\eta \geq \bar\alpha$ and Clause~(iii) of the definition of $\delta \in y^*$, for every $\eta' \in W_y \cap \bar\alpha$ we have $\pi_{j,\tau}(X_{f_y(\eta)})\bks \pi_{j,\tau}(X_{f_y(\eta')})\s[\eta',\eta)$.
Then
\begin{align*}
\pi_{j,\tau}(X_{f_y(\eta)}) \setminus
\bigcup \{ \pi_{j,\tau}(X_{f_y(\eta')}) \mid {\eta' \in W_y \cap \bar\alpha} \}
&= \bigcap \{ \pi_{j,\tau}(X_{f_y(\eta)}) \setminus \pi_{j,\tau}(X_{f_y(\eta')}) \mid
                     {\eta' \in W_y \cap \bar\alpha} \} \\
&\subseteq \bigcap \{ [\eta', \eta) \mid {\eta' \in W_y \cap \bar\alpha} \}  \\
&= [ \sup (W_y \cap \bar\alpha), \eta) = [\bar\alpha,\eta).
\qedhere
\end{align*}
\end{proof}

Define a function $\Phi:\mathcal K(\lambda^+)\rightarrow\mathcal K(\lambda^+)$ by stipulating:
\[\Phi(y):=\begin{cases}
y^*,&\text{if } \sup(y^*)=\sup(y); \\
y \setminus \sup(y^*),&\text{otherwise}.
\end{cases}\]

Define $\mathfrak Z=\langle Z_y\mid y\in\mathcal K(\lambda^+)\rangle$ by stipulating:
$$Z_y:=\begin{cases}
\bigcup\{\pi_{j,\tau}(X_{f_{y}(\eta)})\mid \eta\in  W_{y} \},&\text{if } \sup(y^*)=\sup(y); \\
\emptyset,&\text{otherwise}.
\end{cases}$$

\begin{claim}\label{c443}
\begin{enumerate}
\item $\Phi$ is a conservative postprocessing function;
\item For all $y\in\mathcal K(\lambda^+)$:
\begin{enumerate}
\item $Z_y \subseteq \sup(y)$;
\item $Z_y\cap{\bar\alpha}=Z_{y\cap{\bar\alpha}}$ for all ${\bar\alpha}\in\acc(\Phi(y))$.
\end{enumerate}
\item For every club $E\s\lambda^+$ and every subset $Z\s\lambda^+$, there exists some $\alpha\in S$ such that:
\begin{enumerate}
\item $\Phi(C_\alpha^\circ)\s E$;
\item $Z_{C_\alpha^\circ}=Z\cap\alpha$.
\end{enumerate}
\end{enumerate}
\end{claim}
\begin{proof} (1) It is clear from the definition that $\Phi(y)$ is a subclub of $y$.
Suppose $y\in\mathcal K(\lambda^+)$ and $\bar\alpha\in\acc(\Phi(y))$.
In particular, $\bar\alpha\in \acc(y)$.
To compare $\Phi(y\cap\bar\alpha)$ with $\Phi(y)\cap\bar\alpha$, we shall consider two cases:

$\br$ If $\sup(y^*) = \sup(y)$, then $\Phi(y) = y^*$ and hence $\bar\alpha\in\acc(y^*)$.
By Claim~\ref{coherence-y}(1), $y^* \cap \bar\alpha = (y \cap \bar\alpha)^*$.
Then also $\sup(y \cap \bar\alpha) = \bar\alpha = \sup(y^* \cap \bar\alpha) = \sup((y \cap \bar\alpha)^*)$, so that
\[
\Phi(y\cap\bar\alpha) = (y\cap\bar\alpha)^* = y^* \cap \bar\alpha = \Phi(y) \cap \bar\alpha.
\]

$\br$ If $\sup(y^*) < \sup(y)$, then $\Phi(y) = y \setminus \sup(y^*)$, so that $\bar\alpha > \sup(y^*)$.
Then using Claim~\ref{coherence-y}(1), $y^* = y^* \cap \bar\alpha = (y \cap \bar\alpha)^*$,
so that $\sup((y \cap \bar\alpha)^*) = \sup(y^*) < \bar\alpha = \sup(y \cap \bar\alpha)$, and it follows that
\[
\Phi(y \cap \bar\alpha) = (y \cap \bar\alpha) \setminus \sup((y \cap \bar\alpha)^*)
= y \cap \bar\alpha \setminus \sup(y^*) = (y \setminus \sup(y^*)) \cap \bar\alpha = \Phi(y) \cap \bar\alpha.
\]

(2)(a) We may assume $\sup(y^*) =\sup(y)$.
In particular, using Clause~(ii) of the definition of $y^*$, there are cofinally many $\delta \in y$ such that
$\pi_{j,\tau}(X_{f_{y}(\eta)}) \subseteq \eta$ for every $\eta \in W_{y} \cap \delta$.
Thus $\pi_{j,\tau}(X_{f_{y}(\eta)}) \subseteq \eta \subseteq \sup(y)$ for every $\eta \in W_{y}$, and the result follows.

(2)(b) Consider arbitrary $\bar\alpha \in \acc(\Phi(y))$.
To compare $Z_{y \cap \bar\alpha}$ with $Z_y \cap \bar\alpha$, we shall consider two cases:

$\br$ Suppose $\sup(y^*) = \sup(y)$.
In this case, Claim~\ref{coherence-y} gives $W_{y \cap \bar\alpha} = W_{y} \cap \bar\alpha$ and $f_{y \cap \bar\alpha} = f_{y} \restriction \bar\alpha$.
Thus, for every $\eta \in W_{y} \cap \bar\alpha$, $f_{y \cap \bar\alpha}(\eta) = f_{y}(\eta)$, and $\pi_{j,\tau}(X_{f_{y}(\eta)}) \subseteq \eta$ by Clause~(ii) of $\bar\alpha \in y^*$. Then
\begin{align*}
Z_{y \cap \bar\alpha}
&= \bigcup\{\pi_{j,\tau}(X_{f_{y \cap \bar\alpha}(\eta)})\mid \eta\in  W_{y \cap \bar\alpha} \} \\
&= \bigcup\{\pi_{j,\tau}(X_{f_{y}(\eta)})\mid \eta\in  W_{y} \cap \bar\alpha \},
\end{align*}
while
\begin{align*}
Z_y
&= \bigcup\{\pi_{j,\tau}(X_{f_{y}(\eta)}) \mid \eta\in  W_{y} \} \\
&= \bigcup\{\pi_{j,\tau}(X_{f_{y}(\eta)}) \mid \eta\in  W_{y} \cap \bar\alpha \} \cup
       \bigcup\{\pi_{j,\tau}(X_{f_{y}(\eta)}) \mid \eta\in  W_{y} \setminus \bar\alpha\} \\
&= Z_{y \cap \bar\alpha} \cup
       \bigcup\{\pi_{j,\tau}(X_{f_{y}(\eta)}) \mid \eta\in  W_{y} \setminus \bar\alpha\} \\
&= Z_{y \cap \bar\alpha} \cup
 \bigcup\{\pi_{j,\tau}(X_{f_{y}(\eta)}) \setminus Z_{y \cap \bar\alpha} \mid \eta\in  W_{y} \setminus \bar\alpha\}.
\end{align*}
By Claim~\ref{coherence-y}(3), the big union in the last line above does not contain any ordinal below $\bar\alpha$.
Thus $Z_y \cap \bar\alpha = Z_{y \cap \bar\alpha}$, as required.

$\br$ Suppose $\sup(y^*) < \sup(y)$.
Then as shown in the proof of clause (1), $\sup((y \cap \bar\alpha)^*) < \sup(y \cap \bar\alpha)$, so that $Z_{y}\cap\bar\alpha=\emptyset\cap\bar\alpha=Z_{y\cap\bar\alpha}$.

(3) For the reader's convenience, we repeat the proof of  \cite[Claim~2.5.4]{rinot11}.
Consider arbitrary $Z \subseteq \lambda^+$ and club $E \subseteq \lambda^+$.
By the choice of $j,\tau$ and $Y$, let us fix some $\alpha\in S$ such that
\begin{enumerate}
\item[(a)] $C_\alpha^\circ\s E$;
\item[(b)] $H^j_{C_\alpha^\circ}\bks Y\s \{(\eta,\gamma)\mid Z\cap\eta=\pi_{j,\tau}(X_\gamma)\}$;
\item[(c)] $\sup(\acc^+(W_{C_\alpha^\circ}) \cap\acc(C_\alpha^\circ))=\alpha$.
\end{enumerate}

By Clause~(a) and since $\Phi$ is conservative, $\Phi(C_\alpha^\circ)\s E$.

Let $\eta \in W_{C_\alpha^\circ}$ be arbitrary.
Then $f_{C_\alpha^\circ}(\eta) \in F^{j,Y}_{C_\alpha^\circ, \eta}$, so that in particular, $(\eta, f_{C_\alpha^\circ}(\eta)) \in H^j_{C_\alpha^\circ} \setminus Y$,
and it follows from Clause~(b) above that $Z \cap \eta = \pi_{j,\tau}(X_{f_{C_\alpha^\circ}(\eta)})$.
In particular, $\pi_{j,\tau}(X_{f_{C_\alpha^\circ}(\eta)}) \subseteq \eta$, and for any $\eta'<\eta$ also in $W_{C_\alpha^\circ}$, we have
$\pi_{j,\tau}(X_{f_{C_\alpha^\circ}(\eta)}) \setminus \pi_{j,\tau}(X_{f_{C_\alpha^\circ}(\eta')}) = Z \cap [\eta',\eta)$.
Consequently, every $\delta\in\acc^+(W_{C_\alpha^\circ})\cap\acc(C_\alpha^\circ)$ is in $(C_\alpha^\circ)^*$.
It then follows from Clause~(c) that $\sup((C_\alpha^\circ)^*)=\alpha = \sup(W_{C_\alpha^\circ})$, so that
$$Z_{C_\alpha^\circ}=\bigcup\{\pi_{j,\tau}(X_{f_{C_\alpha^\circ}(\eta)})\mid \eta\in  W_{C_\alpha^\circ}\}
=\bigcup\{ Z \cap \eta \mid \eta\in  W_{C_\alpha^\circ}\} = Z \cap \alpha,$$
as sought.
\end{proof}

Let $\Phi^\mathfrak Z$ be the function given by Fact~\ref{Psi}.

\begin{claim} $\Phi^\bullet := \faithful\circ\Phi^{\mathfrak Z}\circ \Phi_{D}\circ\Phi^\Sigma$ is as sought.
\end{claim}
\begin{proof} $\Phi^\bullet$ is a composition of postprocessing functions, and hence is a postprocessing function.
Let $A \subseteq \lambda^+$ be an arbitrary cofinal subset, and let $F \subseteq \lambda^+$ be an arbitrary club.
Put $Z := A$ and $E := F \cap \acc^+(A)$.
Then by Claim~\ref{c443}(3), we can pick some $\alpha \in S$ such that $\Phi(C_\alpha^\circ) \subseteq E$ and $Z_{C_\alpha^\circ} = Z \cap \alpha$.
In particular, $\alpha \in F$.
Furthermore, $\cf(\alpha)\le\otp(\Phi^\bullet(C_\alpha)) \le\otp(\Phi^\Sigma(C_\alpha)) = \chi = \cf(\alpha)$.
Since in particular $\nacc(\Phi(C_\alpha^\circ))\subseteq E \cap \alpha \subseteq \acc^+(A\cap\alpha) = \acc^+(Z_{C_\alpha^\circ})$,
it follows from the definition of $\Phi^{\mathfrak Z}$ in Fact~\ref{Psi} that $\nacc(\Phi^{\mathfrak Z}(C_\alpha^\circ)) \subseteq Z_{C_\alpha^\circ}$ and $\min(\Phi^{\mathfrak Z}(C_\alpha^\circ)) = \min(Z_{C_\alpha^\circ})$.
But then clearly
$\nacc(\Phi^\bullet(C_\alpha)) \subseteq \nacc(\Phi^{\mathfrak Z}(C_\alpha^\circ)) \s Z_{C_\alpha^\circ} \subseteq A$
and $\min(\Phi^\bullet(C_\alpha)) = \min(\Phi^{\mathfrak Z}(C_\alpha^\circ)) = \min(Z_{C_\alpha^\circ}) = \min(A)$, as required.
\end{proof}
This completes the proof.
\end{proof}

We remark that the opening of the proof of the preceding makes it clear that the following holds as well.
\begin{prop} Suppose that $\langle C_\alpha\mid \alpha \in S \rangle$ is a $C$-sequence over some stationary $S\s E^\kappa_\omega$ for which $\diamondsuit(S)$ holds and $\{ \alpha\in S\mid \otp(C_\alpha)<\alpha\}$ is stationary.

Then there exists a faithful postprocessing function $\Phi:\mathcal K(\kappa)\rightarrow\mathcal K(\kappa)$ satisfying the following.
For every cofinal $A\s\kappa$, there exist stationarily many $\alpha\in S$ such that $\otp(\Phi(C_\alpha))=\cf(\alpha)$, $\nacc(\Phi(C_\alpha))\s A$ and $\min(\Phi(C_\alpha))=\min(A)$. \qed
\end{prop}

\begin{lemma}\label{lemma3.5} Suppose that $\kappa=\lambda^+$ for a given singular cardinal $\lambda$, and $\ch_\lambda$ holds.
Suppose also that $\cvec{C}=\langle\mathcal C_\alpha\mid\alpha<\kappa\rangle$ is a $\square_\xi(\kappa,{<}\mu,{\sq_\chi},V)$-sequence with support $\Gamma$,
for which $\{ \alpha\in E^\kappa_{>\cf(\lambda)}\cap\Gamma\mid \exists C\in\mathcal C_\alpha[\otp(C)<\alpha]\}$ is stationary.

Then there exists a transversal $\langle C_\alpha\mid\alpha \in \Gamma\rangle$ for $\cvec{C}$ and a faithful postprocessing function $\Phi:\mathcal K(\kappa)\rightarrow\mathcal K(\kappa)$ such that:
\begin{enumerate}
\item If $\xi=\lambda$, then $\otp(\Phi(C_\alpha))<\lambda$ for all $\alpha\in \Gamma$;
\item For every cofinal $A\s\kappa$, there exist stationarily many $\alpha\in \Gamma$ with $\otp(\Phi(C_\alpha))=\cf(\lambda)$ and $\nacc(\Phi(C_\alpha))\s A$.
\end{enumerate}
\end{lemma}
\begin{proof} For each $\alpha<\kappa$, fix an enumeration (possibly with repetition) $\{C_\alpha^i\mid i<\lambda\}$ of $\mathcal C_\alpha$.
By the hypothesis, we may also ensure that $S:=\{ \alpha\in E^\kappa_{>\cf(\lambda)}\cap\Gamma\mid \otp(C_\alpha^0)<\alpha\}$ is stationary.
Let $\Sigma$ be some club in $\lambda$ of order-type $\cf(\lambda)$, and let $\Phi^\Sigma$ be given by Fact~\ref{newPhiSigma}.
Let $\Phi_0$ be given by Lemma~\ref{phi1} when fed with the $C$-sequence $\langle \Phi^\Sigma(C^0_\alpha)\mid\alpha\in S\rangle$.
For all $\alpha<\kappa$ and $i<\lambda$, let $D^i_\alpha:=\Phi_0(\Phi^\Sigma(C^i_\alpha))$.

By $\ch_\lambda$ and \cite{Sh:922}, $\diamondsuit(\kappa)$ holds, so let $\langle A^i_\gamma\mid i,\gamma<\kappa\rangle $ be as in Fact~\ref{diamond_matrix}.
For each $i<\kappa$, define a $\kappa$-assignment $\mathfrak Z^i=\langle Z^i_{x,\beta}\mid x\in\mathcal K(\kappa), \beta\in\nacc(x)\rangle$ via the rule $Z^i_{x,\beta}:=A^i_\beta$,
and consider the corresponding postprocessing function $\Phi_{\mathfrak Z^i}$ given by Lemma~\ref{phiZ}.

\begin{claim}\label{claim2.21.3} There exists some $i<\lambda$ such that for every cofinal $A\s\kappa$,
there are stationarily many $\alpha \in \Gamma$ such that $\otp(D_\alpha^{i})=\cf(\lambda)$ and $\nacc(\Phi_{\mathfrak Z^{i}}(D_\alpha^{i}))\s A$.
\end{claim}
\begin{proof} Suppose not, and for each $i<\lambda$, pick a cofinal $A^i\s\kappa$ and a club $E^i\s\kappa$ such that for all $\alpha\in E^i \cap \Gamma$, we have
$\otp(D_\alpha^i)\neq\cf(\lambda)$ or $\nacc(\Phi_{\mathfrak Z^i}(D_\alpha^i))\nsubseteq A^i$.
Consider the stationary set $G:=\allowbreak\{\gamma\in\bigcap_{i<\lambda}E^i\mid \forall i<\lambda(\sup(A^i\cap\gamma)=\gamma\ \&\ A^i\cap\gamma=A^i_\gamma)\}$.
By the choice of $\Phi_0$, pick some $\alpha\in S$ for which $\nacc(\Phi_0(\Phi^\Sigma(C^0_\alpha)))\s G$. That is, $\nacc(D^0_\alpha)\s G$.
As $\otp(D_\alpha^0) \geq\cf(\alpha)>\cf(\lambda)$, we may let $\bar\alpha:=D^0_\alpha(\cf(\lambda))$, so that $\bar\alpha\in\acc(D^0_\alpha)$ and $\otp(D^0_\alpha\cap\bar\alpha)=\cf(\lambda)$.
As $\alpha \in \Gamma$ and $\bar\alpha\in\acc(C^0_\alpha)$, we have $\bar\alpha\in\Gamma$ and $C^0_\alpha \cap\bar\alpha \in \mathcal C_{\bar\alpha}$,
so that we may fix some $i<\lambda$ such that $C^0_\alpha\cap\bar\alpha=C_{\bar\alpha}^i$. In particular, $D^0_\alpha\cap\bar\alpha=D_{\bar\alpha}^i$.
Since $\nacc(D^0_\alpha)\s G$, it follows from the choice of $\bar\alpha$ that $\bar\alpha \in \acc^+(G) \subseteq \acc^+(E^i) \subseteq E^i$.
Then $\nacc(D^i_{\bar\alpha}) \subseteq \nacc(D^0_\alpha) \subseteq G \subseteq \{ \beta \in \acc^+(A^i) \mid A^i \cap \beta = A^i_\beta \}$,
so that as explained in Example~\ref{phiZ-simpler}, $\nacc(\Phi_{\mathfrak Z^i}(D^i_{\bar\alpha})) \subseteq A^i$.
Altogether, we have found $\bar\alpha \in E^i \cap \Gamma$ such that $\otp(D_{\bar\alpha}^{i})=\cf(\lambda)$ and $\nacc(\Phi_{\mathfrak Z^i}(D_{\bar\alpha}^i))\s A^i$, contradicting the choice of $A^i$ and $E^i$.
\end{proof}

Let $i$ be given by the preceding.
Let $\faithful$ be given by Example~\ref{faithful_correction},
so that  $\Phi:=\faithful \circ \Phi_{\mathfrak Z^i}\circ \Phi_0\circ \Phi^\Sigma$ is faithful.
Clearly, $\vec C:=\langle C_\alpha^i\mid \alpha\in\Gamma\rangle$ is a transversal for $\cvec{C}$.
We verify that $\vec C$ and $\Phi$ are as sought:
\begin{enumerate}
\item Consider arbitrary $\alpha\in\Gamma$, and assume that $\xi=\lambda$.
To see that $\otp(\Phi(C_\alpha^i))<\lambda$, we consider two possibilities, noting that $\otp(C_\alpha^i) \leq\xi=\lambda$:
\begin{itemize}
\item[$\br$] If $\otp(C_\alpha^i)<\lambda$, then since $\Phi$ is a postprocessing function, $\otp(\Phi(C_\alpha^i)) \leq \otp(C_\alpha^i) < \lambda$.
\item[$\br$] If $\otp(C_\alpha^i)=\lambda$, then $\otp(\Phi^\Sigma(C_\alpha^i)) = \cf(\lambda)$ by Fact~\ref{newPhiSigma}(\ref{PhiSigma3}),
so that $\otp(\Phi(C_\alpha^i)) = \otp((\faithful \circ \Phi_{\mathfrak Z^i}\circ \Phi_0)(\Phi^\Sigma(C_\alpha^i))) \leq \otp(\Phi^\Sigma(C_\alpha^i)) = \cf(\lambda) < \lambda$, since $\lambda$ is singular.
\end{itemize}
\item Note that $\Phi(C_\alpha^i) = (\faithful\circ\Phi_{\mathfrak Z^i})(D_\alpha^i)$ for every $\alpha \in \Gamma$.
Given any cofinal $A \subseteq\kappa$, by our choice of $i$ there are stationarily many $\alpha\in\Gamma$
such that $\otp(D_\alpha^i) = \cf(\lambda)$ and $\nacc(\Phi_{\mathfrak Z^i}(D_\alpha^i)) \subseteq A$.
For any such $\alpha$, we have $\cf(\lambda) \leq \otp(\Phi(C_\alpha^i)) \leq \otp(D_\alpha^i) = \cf(\lambda)$
and $\nacc(\Phi(C_\alpha^i)) \subseteq \nacc(\Phi_{\mathfrak Z^i}(D_\alpha^i)) \subseteq A$, as sought. \qedhere
\end{enumerate}
\end{proof}

In the particular case where $\mu=\kappa=\lambda^+$, the fact that $\left| \reg(\lambda^+) \right|<\mu$
allows us to obtain a transversal by using a different postprocessing function for each cofinality, combining the results of Lemmas \ref{phi1} and~\ref{lemma3.5},
and essentially combining up to $\lambda$ many $\square^*_\lambda$-sequences into one, as follows:
\begin{cor}\label{c38} Suppose that $\lambda$ is an uncountable cardinal, and $\square_\lambda^*+\ch_\lambda$ holds.

Then there exists a transversal $\langle C_\alpha\mid\alpha \in \acc(\lambda^+)\rangle$ for $\square^*_\lambda$ such that:
\begin{itemize}
\item $\otp(C_\alpha)<\lambda$ for all $\alpha\in \acc(\lambda^+) \cap E^{\lambda^+}_{<\lambda}$;
\item for every cofinal $A\s\lambda^+$ and every $\theta\in\reg(\lambda)$, there exist stationarily many $\alpha \in\acc(\lambda^+)$ for which $\otp(C_\alpha)=\theta$ and $\nacc(C_\alpha)\s A$.
\end{itemize}
\end{cor}
\begin{proof} Fix an arbitrary $\square^*_\lambda$-sequence, $\cvec{C}:=\langle\mathcal C_\alpha\mid\alpha<\lambda^+\rangle$.
Fix a transversal $\vec C=\langle C_\alpha\mid\alpha\in\acc(\lambda^+)\rangle$ for $\cvec{C}$, so that $\otp(C_\alpha)<\lambda$ for all $\alpha\in \acc(\lambda^+) \cap E^{\lambda^+}_{\neq\cf(\lambda)}$.
If $\lambda$ is singular, then by Lemma~\ref{lemma3.5} (using $(\xi,\mu,\chi):=(\lambda,\lambda^+,\aleph_0)$),
we may moreover assume the existence of a faithful postprocessing function $\Phi_{\cf(\lambda)}:\mathcal K(\kappa)\rightarrow\mathcal K(\kappa)$ such that:
\begin{enumerate}
\item $\otp(\Phi_{\cf(\lambda)}(C_\alpha))<\lambda$ for all $\alpha\in \acc(\lambda^+)$;
\item For every cofinal $A\s\lambda^+$, there exist stationarily many $\alpha\in\acc(\lambda^+)$ with $\otp(\Phi_{\cf(\lambda)}(C_\alpha))=\cf(\lambda)$ and $\nacc(\Phi_{\cf(\lambda)}(C_\alpha))\s A$.
\end{enumerate}

Let $\Phi_\lambda$ be the identity postprocessing function. Next, for all $\theta\in\reg(\lambda)\setminus\{\cf(\lambda)\}$, appeal to Lemma~\ref{phi1} with $\langle C_\alpha\mid\alpha\in E^{\lambda^+}_\theta\rangle$ to get a postprocessing function $\Phi_\theta$.

For all $\alpha\in\acc(\lambda^+)$, let $D_\alpha:=\Phi_{\cf(\alpha)}(C_\alpha)$, so that for every cofinal $A\s\lambda^+$ and every $\theta\in\reg(\lambda)$,
there exist stationarily many $\alpha \in\acc(\lambda^+)$ for which $\otp(D_\alpha)=\theta$ and $\nacc(D_\alpha)\s A$.
In addition, $\otp(D_\alpha) <\lambda$  for all $\alpha\in \acc(\lambda^+) \cap E^{\lambda^+}_{<\lambda}$.

Thus, $\langle D_\alpha\mid\alpha\in\acc(\lambda^+)\rangle$ is as sought, provided that it is a transversal for $\square_\lambda(\lambda^+,\lambda,{\sq},V)$.
By Proposition~\ref{transversal-width}(1), this now amounts to showing that for every $\gamma\in\acc(\lambda^+)$, the set
$$\mathcal D_\gamma:=\{ D_\alpha \cap\gamma \mid \alpha \in \acc(\lambda^+) \text{ and } \sup(D_\alpha \cap\gamma) = \gamma \}$$
has size $\le\lambda$. But, of course, $\mathcal D_\gamma\s \bigcup\{\Phi_\theta(C)\mid C\in\mathcal C_\gamma, \theta\in\reg(\lambda^+)\}$, so we are done.
\end{proof}

\begin{lemma}\label{phi0} If $\diamondsuit(\kappa)$ holds, then there exists a faithful, $\acc$-preserving
postprocessing function $\Phi:\mathcal K(\kappa)\rightarrow\mathcal K(\kappa)$ satisfying the following.
For every (possibly constant) sequence $\langle A_i\mid i<\kappa\rangle$ of cofinal subsets of $\kappa$,
there exists some stationary set $G$ in $\kappa$ such that for all $x\in\mathcal K(\kappa)$, if $\nacc(x)\s G$, then $\min(\Phi(x))=\min(A_0)$ and $\Phi(x)(i+1)\in A_i$ for all $i<\otp(x)$.
\end{lemma}
\begin{proof} Assume $\diamondsuit(\kappa)$, and let $\langle A^i_\gamma\mid i,\gamma<\kappa \rangle $ be given by Fact~\ref{diamond_matrix}.

Define $\mathfrak Z:=\langle Z_{x,\beta}\mid x\in\mathcal K(\kappa),\allowbreak\beta\in\nacc(x)\rangle$ by stipulating:
\[Z_{x,\beta}:=\begin{cases}
A_\beta^{0},&\text{if } \beta = \min(x);\\
A_\beta^{\otp(x\cap\beta)-1}\setminus\omega, &\text{otherwise}.
\end{cases}\]

As $\mathfrak Z$ is a $\kappa$-assignment, let $\Phi_{\mathfrak Z}$ be the $\acc$-preserving postprocessing function given by Lemma~\ref{phiZ}.
Let $\faithful$ be as in Example~\ref{faithful_correction}, so that $\Phi:=\faithful\circ \Phi_{\mathfrak Z}$ is faithful and $\acc$-preserving.

To see that $\Phi$ works, let $\langle A^i\mid i<\kappa\rangle$ be an arbitrary sequence of cofinal subsets of $\kappa$.
We claim that the stationary set $G:=\{\beta\in\acc(\kappa\setminus\omega)\mid \forall i<\beta(\sup(A^i\cap\beta)=\beta\ \&\ A^i\cap\beta=A^i_\beta)\}$ is as sought.
To see this, suppose that we are given $x\in\mathcal K(\kappa)$ for which $\nacc(x)\s G$.
\begin{itemize}
\item[$\br$] Let $\beta:=\min(x)$. Then $\beta\in\nacc(x)\s G$, and hence $Z_{x,\beta}=A^0_\beta=A^0\cap\beta$ is a cofinal subset of $\beta$, so that $\min(\Phi_{\mathfrak Z}(x))=\min(Z_{x,\beta}\cup\{\beta\})=\min(A^0)$.
As $\faithful$ is $\min$-preserving, we have $\min(\Phi(x))=\min(\Phi_{\mathfrak Z}(x))=\min(A^0)$.
\item[$\br$] Let $i < \otp(x)$ be arbitrary. Put $\beta:=x(i+1)$ and $\beta^-:=x(i)$.
Then $$\Phi_{\mathfrak Z}(x)(i+1) = \min(((A^i_{\beta}\setminus\omega)\cup\{\beta\}) \setminus (\beta^-+1)).$$
As $i\le\beta^-<\beta$ and $\beta\in \nacc(x)\s G$, we have that $A^i_\beta=A^i\cap\beta$ is a cofinal subset of $\beta>\omega$, and hence $$\Phi_\mathfrak{Z}(x)(i+1)\in A^i\setminus\omega.$$
In particular, $\Phi_\mathfrak{Z}(x)(1)>2$, so that by the definition of $\faithful$, $\Phi(x) = \Phi_\mathfrak{Z}(x)$.
Thus, $\Phi(x)(i+1)=\Phi_\mathfrak{Z}(x)(i+1)\in A^i$ for all $i < \otp(x)$.\qedhere
\end{itemize}
\end{proof}

We are now ready to prove Theorem~\ref{main3.2}:

\begin{cor}\label{cor310} Suppose that $\lambda$ is an uncountable cardinal, and $\square_\lambda^*+\ch_\lambda$ holds.

Then there exists a $C$-sequence $\langle C_\alpha\mid\alpha <\lambda^+\rangle$ such that:
\begin{itemize}
\item $\otp(C_\alpha)<\lambda$ for all $\alpha\in \acc(\lambda^+) \cap E^{\lambda^+}_{<\lambda}$;
\item $|\{ C_\alpha\cap\gamma\mid \alpha<\lambda^+\}|\le\lambda$ for all $\gamma<\lambda^+$;
\item For every $\theta\in\reg(\lambda)$ and every sequence $\langle A_i\mid i<\theta\rangle$ of cofinal subsets of $\lambda^+$,
there exist stationarily many $\alpha \in\acc(\lambda^+)$ such that $\otp(C_\alpha)=\theta$, $\min(C_\alpha)=\min(A_0)$ and $C_\alpha(i+1)\in A_i$ for all $i<\theta$.
\end{itemize}
\end{cor}
\begin{proof} Let $\langle C_\alpha\mid\alpha \in \acc(\lambda^+)\rangle$ be given by Corollary~\ref{c38}.
By $\ch_\lambda$ and \cite{Sh:922}, $\diamondsuit(\lambda^+)$ holds, so we may let $\Phi$ be given by Lemma~\ref{phi0}.
For all $\alpha<\lambda^+$, define
$$C_\alpha^\bullet:=\begin{cases}
\Phi(C_\alpha),&\text{if }\alpha\in\acc(\lambda^+);\\
\{\beta\},&\text{if }\alpha=\beta+1;\\
\emptyset,&\text{if }\alpha=0.
\end{cases}$$
Then $\langle C_\alpha^\bullet\mid\alpha<\lambda^+\rangle$ is as sought.
\end{proof}

The second case of Theorem~A now follows, with the additional benefit that we do not require $\lambda$ to be a strong-limit:

\begin{cor}\label{cor39} Suppose that $\lambda$ is a singular cardinal, and $\square_\lambda^*+\ch_\lambda$ holds.

Then there exists a normal $\lambda$-distributive $\lambda^+$-Aronszajn tree.
\end{cor}
\begin{proof} Let $\langle C_\alpha\mid\alpha <\lambda^+\rangle$ be given by Corollary~\ref{cor310}.
Let $\cvec{C}=\langle\mathcal C_\alpha\mid\alpha<\lambda^+\rangle$ be the induced $\square^*_\lambda$-sequence.
That is, $\mathcal C_{\alpha+1}:=\{\{\alpha\}\}$ for each $\alpha<\lambda^+$, and $\mathcal C_\alpha:=\{ C_\delta\cap\alpha\mid \delta<\lambda^+, \sup(C_\delta\cap\alpha)=\alpha\}$ for each limit $\alpha < \lambda^+$.

By running the very same construction of a normal $\lambda$-splitting $\lambda^+$-tree of \cite[Proposition~2.2]{paper26},
modulo a single change in the definition of the limit levels $T_\alpha$, letting
$$T_\alpha:=\left\{ \mathbf b^C_x \mid C \in \mathcal C_\alpha, x \in T \restriction C \right\}\setminus\left\{\bigcup\Omega_\alpha\right\},$$
we get an outcome normal tree that is $\lambda$-distributive and $\lambda^+$-Aronszajn.

Readers unfamiliar with \cite{paper26} may feel uncomfortable with the above sketch.
Thus, let us give a proof which is based on the exact same construction as in \cite{MR0861900}.

Let $\Phi^\Sigma$ be given by Fact~\ref{newPhiSigma} for $\Sigma := \acc(\lambda^+)$.
For each $\alpha\in\acc(\lambda^+)$, let $A_\alpha:=\{ \Phi^\Sigma(C)\mid C\in\mathcal C_\alpha\}$.
By $\ch_\lambda$ and \cite{Sh:922}, let $\langle b_\beta\mid\beta<\lambda^+\rangle$ be a $\diamondsuit(\lambda^+)$-sequence. We have:
\begin{itemize}
\item For all $\alpha\in\acc(\lambda^+)$, each $a\in A_\alpha$ is a club in $\alpha$ of order-type $<\lambda$;
\item For all $\beta<\lambda^+$, $b_\beta\s\beta$;
\item For every $X\s\lambda^+$, every club $C\s\lambda^+$ and every $\delta<\lambda$, there exists some $\alpha\in\acc(\lambda^+)$ and $a\in  A_\alpha$ such that $\otp(a)>\delta$ and for all $\beta\in a\cup\{\alpha\}$, $a\cap\beta\in A_\beta$ and $X\cap\beta=b_\beta$.
\end{itemize}
So we have established the existence of two sequences $\langle A_\alpha\mid\alpha\in\acc(\lambda^+)\rangle$ and $\langle \{b_\alpha\}\mid \alpha\in\acc(\lambda^+)\rangle$ as in \cite[p.~94]{MR0861900}, from which the construction of the desired tree can be carried out.
\end{proof}

\section{$C$-sequences of intermediate order-types}\label{section4}

The key Lemma of this section (Lemma~\ref{blowup-nacc} below) transforms a $\mathcal C$-sequence whose clubs typically have short order-types
into another one with typically longer order-types (of the sort needed to construct uniformly coherent Souslin trees).
The expanded generality obtained by not requiring the target order-type to be a cardinal, and by introducing the relation $\mathcal R^\Omega$,
will allow us to apply the Lemma to scenarios beyond the scope of this paper (see \cite{paper32}).

We begin with some preliminaries.
The first-time reader may assume throughout this section that $\Omega=\emptyset$, and that either $\Lambda=\kappa$ or $\Lambda=\lambda$ for a singular cardinal $\lambda$ satisfying $\lambda^+=\kappa$,
without sacrificing the flow and the results of this paper.

\begin{defn} For any binary relation $\mathcal R$ and any set $\Omega$, we let $\mathcal R^\Omega:=\{ (C,D)\in\mathcal R\mid \sup(C)\notin\Omega\}$.
\end{defn}

In particular, a $\mathcal C$-sequence $\langle \mathcal C_\alpha \mid \alpha<\kappa \rangle$ is $\mathcal R^\Omega$-coherent iff it is $\mathcal R$-coherent and
$\acc(C) \cap \Omega = \emptyset$ for all $C \in \bigcup_{\alpha < \kappa} \mathcal C_\alpha$.

\begin{example}\label{avoids} If $C\sq^\Omega D$ for $C, D \in \mathcal K(\kappa)$, and $\Phi$ is a postprocessing function, then $\Phi(C)\sq^\Omega\Phi(D)$.
\end{example}

\begin{defn}\label{def46} The binary relation $\sin$ is defined as follows. $A\sin\mathcal C$ iff $A\s C$ for some $C\in\mathcal C$. Its negation is denoted by $\nsin$.
\end{defn}

The principle $\square_\xi(\kappa,{<}\mu,{\sq^\Omega_\chi},{\nsin})$ is syntactically stronger than $\square_\xi(\kappa,{<}\mu,{\sq^\Omega_\chi},{\notin})$,
but it follows from the upcoming Lemma that the two notions coincide whenever $\min\{\xi,\mu\}<\kappa$.
As we shall see, the advantage of $\nsin$ over $\notin$ is that the former characterizes amenability of transversals
(Lemma~\ref{nsintransversal}) and is preserved by postprocessing functions (Lemma~\ref{pp-preserves-square2}) and by the procedure of Lemma~\ref{blowup-nacc},
even when $\xi=\mu=\kappa$, unlike the latter (cf.\ the remark before Lemma~\ref{square_is_amenable}, and Remark~\ref{min<k-necessary}).
In particular, in Kunen's model from \cite[\S3]{MR495118}, there exists an inaccessible cardinal $\kappa$ for which $\square(\kappa,{<}\kappa,{\sq},{\notin})$ holds but $\square(\kappa,{<}\kappa,{\sq},{\nsin})$ fails.

\begin{lemma}\label{hitting-implies-nontrivial} Suppose that $\cvec{C} = \langle \mathcal C_\alpha \mid \alpha<\kappa \rangle$ is a $\square_\xi(\kappa,{<}\mu,{\sq^\Omega_\chi},\mathcal R_1)$-sequence.

Each of the following implies that $\cvec{C}$ witnesses $\square_\xi(\kappa,{<}\mu,{\sq^\Omega_\chi},{\nsin})$:
\begin{enumerate}
\item $\xi<\kappa$;
\item $\Omega$ is a stationary subset of $\kappa$;
\item For every club $D \subseteq\kappa$, there is $\alpha\in\acc(\kappa)$ such that $\sup(\nacc(C)\cap D)=\alpha$ for all $C \in \mathcal C_\alpha$;
\item $\mathcal R_1={\notin}$ and there exists $\mu'<\kappa$ for which $\{\alpha\in \Gamma(\cvec{C})\mid |\mathcal C_\alpha|<\mu'\}$ is stationary;
\item $\mathcal R_1={\notin}$ and there exists a sequence of injections $\langle i_\alpha:\mathcal C_\alpha\rightarrow\kappa\mid \alpha<\kappa\rangle$
such that:
\begin{itemize}
\item For all $\alpha\in\Gamma(\cvec{C})$,  $C\in\mathcal C_\alpha$, and $\bar\alpha\in\acc(C)$,  $i_{\bar\alpha}(C\cap\bar\alpha)=i_\alpha(C)$;
\item $\{\alpha\in\Gamma(\cvec{C})\mid \rng(i_\alpha)\s\alpha\}$ is stationary.\footnote{That is, $\cvec{C}$ is the union of a (locally) small number of partial squares.}
\end{itemize}
\end{enumerate}
\end{lemma}
\begin{proof} Let $A$ denote an arbitrary cofinal subset of $\kappa$.
Put $D:=\acc^+(A)$. We need to find $\alpha\in D$ such that $A\cap\alpha\nsin\mathcal C_\alpha$.

(1) Pick a large enough $\alpha\in D$ such that $\otp(A\cap\alpha)>\xi$. Trivially, $A\cap\alpha\nsin\mathcal C_\alpha$.

(2) Since $D$ is a club in $\kappa$, $D \cap \Omega$ is stationary.
Fix $\alpha \in D\cap\Omega$ such that $\otp(D\cap\Omega\cap\alpha) = \omega$.
Let $C\in\mathcal C_\alpha$ be arbitrary.
If $A\cap\alpha\s C$, then $D\cap\Omega\cap\alpha \subseteq \acc^+(A\cap\alpha) \s\acc(C)$, meaning that $\Omega\cap\acc(C)\neq\emptyset$, contradicting $\sq_\chi^\Omega$-coherence.

(3) Pick $\alpha\in\acc(\kappa)$ such that $\sup(\nacc(C)\cap D)=\alpha$ for every $C\in\mathcal C_\alpha$.
As $\alpha\in\acc^+(D)\s D$, we are left with verifying that $A\cap\alpha\nsin\mathcal C_\alpha$. Fix an arbitrary $C\in\mathcal C_\alpha$.
If $A\cap\alpha\s C$, then $D\cap\alpha\s\acc(C)$, contradicting the fact that $D\cap \nacc(C)\neq\emptyset$.

(4) Towards a contradiction, suppose that there exists a transversal $\langle C_\alpha\mid\alpha\in\Gamma\rangle$ for $\cvec{C}$ such that $A\cap\alpha\s C_\alpha$ for all $\alpha\in\Gamma\cap D$.
We then follow the proof of Lemma~\ref{square_is_amenable} until the end of Claim~\ref{claim1221},
noting that the cardinal inequality in the statement of Claim~\ref{claim1221} is, in fact,
$\left| T_\alpha \right| \leq \left| \mathcal C_{\beta_\alpha} \right|$ for all $\alpha<\kappa$.
Thus, instead of arguing about $\bigcup\{T_\alpha\mid \alpha<\kappa\}$, we infer from Kurepa's lemma that the $\kappa$-tree $\bigcup\{ T_\alpha\mid \alpha<\kappa\ \&\ |\mathcal C_{\beta_\alpha}|<\mu'\}$
admits a cofinal branch, and such a branch contradicts the fact that $\cvec{C}$ witnesses $\square_\xi(\kappa,{<}\mu,\sq_\chi,{\notin})$.

(5) Towards a contradiction, suppose that there exists a transversal $\langle C_\alpha\mid\alpha\in\Gamma\rangle$ for $\cvec{C}$ such that $A\cap\alpha\s C_\alpha$ for all $\alpha\in\Gamma\cap D$.
Fix a stationary $S\s\Gamma\cap D$ on which $\alpha\mapsto i_\alpha(C_\alpha)$ is constant, with value, say, $i^*$.
We shall show that the sequence $\langle C_\alpha\mid \alpha\in S\rangle$ is $\sq$-increasing, so that $\bigcup\{C_\alpha\mid\alpha\in S\}$ contradicts the fact that $\mathcal R_1={\notin}$.

Let $\bar\alpha<\alpha$ be a pair of ordinals from $S$.
As $A\cap\alpha\s C_\alpha$ and $\bar\alpha\in \acc^+(A\cap\alpha)$, we have $\bar\alpha\in\acc(C_\alpha)$, so that $i_{\bar\alpha}(C_\alpha\cap\bar\alpha) = i_\alpha(C_\alpha) =i^*=i_{\bar\alpha}(C_{\bar\alpha})$.
As $i_{\bar\alpha}$ is injective, we infer that indeed $C_{\bar\alpha}\sq C_\alpha$.
\end{proof}

\begin{lemma}\label{nsintransversal} Suppose that $\cvec{C}$ is a $\square_\xi(\kappa,{<}\mu,{\sq_\chi}, V)$-sequence.
Then the following are equivalent:
\begin{enumerate}
\item $\cvec{C}$ witnesses $\square_\xi(\kappa,{<}\mu,{\sq_\chi},{\nsin})$;
\item Every transversal for $\cvec{C}$ is amenable.
\end{enumerate}
\end{lemma}
\begin{proof} Write $\cvec{C}=\langle \mathcal C_\alpha\mid\alpha<\kappa\rangle$.
By Lemma~\ref{Gamma-closure}(3), $\Gamma := \Gamma(\cvec{C})$ is stationary in $\kappa$.

$\neg(2) \implies \neg(1)$:
Let $\vec C=\langle C_\alpha\mid\alpha\in\Gamma\rangle$ be a transversal for $\cvec{C}$ that is not amenable.
By Proposition~\ref{amenable_vs_trivial}, fix a cofinal subset $A\s\kappa$ for which $S:=\{\alpha\in\Gamma\mid A\cap\alpha\s C_\alpha\}$ is stationary.
Let $\alpha\in\acc^+(A)$ be arbitrary.
Put $\alpha':=\min(S\setminus(\alpha+1))$, so that $A\cap\alpha'\subseteq C_{\alpha'}$.
Then, $\alpha\in\acc^+(A)\cap\alpha'\s\acc(C_{\alpha'})$, so that $\alpha\in\Gamma$, $C_{\alpha'}\cap\alpha\in\mathcal C_\alpha$, and hence $A\cap\alpha\sin\mathcal C_\alpha$.

$\neg(1) \implies \neg(2)$:
Fix a cofinal set $A \subseteq \kappa$ such that $A\cap\alpha\sin\mathcal C_\alpha$ for all $\alpha\in\acc^+(A)$.
Choose a transversal $\vec{C} = \langle C_\alpha \mid \alpha\in\Gamma \rangle$ for $\cvec{C}$
such that $A\cap\alpha \subseteq C_\alpha$ for all $\alpha \in \acc^+(A) \cap \Gamma$.
Then $\acc^+(A) \cap\Gamma$ is stationary, and it follows by Proposition~\ref{amenable_vs_trivial} that $\vec{C}$ is not amenable.
\end{proof}

\begin{fact}[Facts about transfinite ordinal sums]\label{ordinal-sums} For any ordinal $\Lambda$:
\begin{enumerate}
\item If $\Lambda$ is indecomposable and $\langle \Lambda_j \mid j<\cf(\Lambda) \rangle$ is a sequence of ordinals each less than $\Lambda$, converging to $\Lambda$, then $\sum_{j<\cf(\Lambda)} \Lambda_j = \Lambda$.
\item If $\langle \Lambda_j \mid j<\cf(\Lambda) \rangle$ is a sequence of nonzero ordinals such that $\sum_{j<\cf(\Lambda)} \Lambda_j = \Lambda$, then $\sum_{j<j'} \Lambda_j < \Lambda$ for every $j' < \cf(\Lambda)$.
\item If $\langle \Lambda_j \mid j<\cf(\Lambda) \rangle$ is a nondecreasing sequence of ordinals such that $\sum_{j<\cf(\Lambda)} \Lambda_j = \Lambda$, and $Z \subseteq \cf(\Lambda)$ is a cofinal subset,
then $\sum_{j \in Z} \Lambda_j = \Lambda$, where the sum is understood to be taken according to the increasing enumeration of $Z$.
\end{enumerate}
\end{fact}

It follows from Fact~\ref{ordinal-sums}(1) that the following is well-defined:

\begin{notation}
For any indecomposable ordinal $\Lambda \leq \kappa$, we write
\[a(\Lambda,\kappa) := \begin{cases}
\min \left\{ \sup \{ \Lambda_j+1 \mid j<\cf(\Lambda)\} \mathrel{\Big|} \sum_{j<\cf(\Lambda)} \Lambda_j =\Lambda \right\},
&\text{if } \Lambda<\kappa; \\
\kappa, &\text{if } \Lambda=\kappa.
\end{cases}\]
\end{notation}
Clearly $2 \leq a(\Lambda,\kappa) \leq \Lambda$ for every $\Lambda$.

\begin{example}\label{examples-a}
\begin{enumerate}
\item If $\Lambda<\kappa$ is any regular infinite cardinal, then $a(\Lambda,\kappa) = 2$.
\item If $\Lambda<\kappa$ is any singular infinite cardinal, then $a(\Lambda,\kappa) = \Lambda$.
\item If $\Lambda = \lambda\cdot\eta$ (ordinal multiplication) for cardinals $\eta=\cf(\eta) \leq \lambda<\kappa$,
then $a(\Lambda,\kappa) = \lambda+1$.
\end{enumerate}
\end{example}

We now arrive at the key Lemma of this section:

\begin{lemma}\label{blowup-nacc} Suppose that $\chi\le\Lambda\le\xi\le\kappa$, with $\Lambda$ some indecomposable ordinal. Suppose also:
\begin{enumerate}[(a)]
\item $\diamondsuit(\kappa)$ holds;
\item $\cvec{C}=\langle \mathcal C_\alpha\mid \alpha<\kappa\rangle$ is a $\square_\xi(\kappa,{<}\mu,{\sq_\chi^\Omega}, V)$-sequence for some fixed subset $\Omega\s\kappa\setminus\{\omega\}$;
\item $\vec C=\langle C_\alpha \mid \alpha \in \Gamma \rangle$ is a transversal for $\cvec{C}$;
\item\label{hyp-c-nacc} for every cofinal $B\s\kappa$  and every $\Lambda' < a(\Lambda,\kappa)$, the following set is stationary in $\kappa$:
\[
\{\alpha\in\Gamma\setminus\Omega\mid \exists C \in \mathcal C_\alpha [\min(C)=\min(B), \Lambda' \le\otp(C)<\Lambda,\nacc(C)\s B]\}.
\]
\end{enumerate}

Then there exists a $\square_\xi(\kappa,{<}\mu,{\sq^\Omega_\chi},V)$-sequence, $\cvec{D}=\langle \mathcal D_\alpha\mid\alpha<\kappa\rangle$, with a transversal $\langle D_\alpha\mid\alpha\in\Gamma\rangle$,
satisfying the following properties:
\begin{enumerate}
\item $|\mathcal D_\alpha| \leq |\mathcal C_\alpha|$ for all $\alpha < \kappa$;
\item If $\Lambda<\kappa$ and $\cvec{C}$ witnesses $\square_\xi(\kappa,{<}\mu,{\sq^\Omega_\chi},{\nsin})$, then so does $\cvec{D}$;
\item For every cofinal $A\s\kappa$, there exists a stationary $S\s\kappa$, for which the set
$$\{\alpha\in\Gamma\mid \otp(D_\alpha)=\min\{\alpha,\Lambda\}, \nacc(D_\alpha)\s A\}$$
covers the set
$$\{\alpha\in \Gamma\mid \min(C_\alpha)=\min(S), \otp(C_{\alpha})=\cf(\min\{\alpha,\Lambda\}), \nacc(C_{\alpha})\s S \}.$$
\end{enumerate}
\end{lemma}
\begin{proof} Note that $\kappa\ge\xi\ge\Lambda \geq \chi \geq \aleph_0$.
The proof is an elaboration of the approach taken in \cite{rinot19}.
Fix a $\diamondsuit(\kappa)$-sequence $\langle X_\gamma\mid \gamma<\kappa\rangle$.
Fix a surjection $\varphi:(\kappa\setminus\{0\})\rightarrow [\kappa]^{<\omega} \times\kappa$ such that $\varphi(7)=(\emptyset,0)$ and such that if $\varphi(\beta)=(\varsigma,k)$, then $\sup(\varsigma)\leq\beta$.
For all $\alpha<\kappa$, let $\{ C_{\alpha,i}\mid i<\kappa\}$ be some enumeration (possibly with repetition) of $\mathcal C_\alpha$ such that
$C_{\alpha,0}=C_{\alpha,i}=C_\alpha$ whenever $\alpha \in \Gamma$ and $i\geq |\mathcal C_\alpha|$.
Without loss of generality, we may assume that for all $\alpha\in\acc(\kappa)$ and $i<\kappa$:
\begin{itemize}
\item If $\alpha\notin\Gamma$, then $C_{\alpha,i}=(C_{\alpha,i}\setminus (\omega+1))\cup((\omega+1)\setminus1)$;
\item If $\alpha=\omega$, then $C_{\alpha,i}=\omega\setminus 7$;
\item Otherwise, $|C_{\alpha,i}\cap\omega|\le1$.\footnote{That is, if $\min(C_{\alpha,i})<\omega$,
 then $C_{\alpha,i}\cap\omega=\{\min(C_{\alpha,i})\}$, and otherwise, $C_{\alpha,i}\cap\omega=\emptyset$.}
\end{itemize}

Fix a nondecreasing sequence $\langle \Lambda_j \mid j<\cf(\Lambda) \rangle$ of nonzero ordinals such that
$\sum_{j<\cf(\Lambda)} \Lambda_j = \Lambda$ and $\sup_{j<\cf(\Lambda)}(\Lambda_j+1) = a(\Lambda,\kappa)$.
Then, define a function $\rho : [\kappa]^{<\cf(\Lambda)} \to \Lambda$ by considering two cases:
\begin{itemize}
\item[$\br$] If $\Lambda<\kappa$, then we define $\rho:[\kappa]^{<\cf(\Lambda)} \rightarrow\Lambda$ by stipulating $\rho(x):= \sum_{j<\otp(x)} \Lambda_j$.
By Fact~\ref{ordinal-sums}(2) and the choice of the sequence $\langle \Lambda_j \mid j<\cf(\Lambda) \rangle$,
$\rho$ is well-defined.
\item[$\br$] Otherwise, define $\rho:[\kappa]^{<\kappa} \rightarrow\Lambda$ by stipulating $\rho(x):=\sup(x)$. Note that $\kappa=\xi=\Lambda = \cf(\Lambda) > \chi$ in this case.
\end{itemize}
In both cases, $\rho$ is an increasing, continuous and cofinal map from the poset $([\kappa]^{<\cf(\Lambda)},{\sq})$ to $(\Lambda,{<})$,
with $\rho(\emptyset) = 0$.

For each $C\in\mathcal K(\kappa)$, let us denote $C\setminus\{\min(C)\}$ by $C^\circ$.
We now define a matrix $\langle D_{\alpha,i}\mid \alpha,i<\kappa\rangle$ by recursion on $\alpha<\kappa$.
Let $D_{\alpha,i} := C_{\alpha,0}$ for all $\alpha \in \kappa\setminus \Gamma$ and all $i<\kappa$.
In particular, $\otp(D_{\alpha,i}) = \cf(\alpha) < \chi \leq \Lambda$.

Next, fix $\alpha \in \Gamma$, and suppose that $\langle D_{\zeta,k}\mid \zeta<\alpha,k<\kappa\rangle$ has already been defined.
Let $i<\kappa$ be arbitrary.
To define $D_{\alpha,i}$, we consider three cases:

\begin{itemize}
\item[Case 1.] If $\min(C_{\alpha,i})>0$, then write $(\varsigma,k):=\varphi(\min(C_{\alpha,i}))$, so that $\sup(\varsigma) \leq \min(C_{\alpha,i})$, and consider two subcases:
\begin{itemize}
\item[Case 1.1.] If $\varsigma\neq\emptyset$, $\min(\varsigma)\in(\Gamma\setminus\Omega)\cup\{0\}$ and $\otp(D_{\min(\varsigma),k})<\Lambda$, then let $$D_{\alpha,i}:= D_{\min(\varsigma),k} \cup \varsigma \cup (C_{\alpha,i})^\circ.$$
\item[Case 1.2.] Otherwise, let $D_{\alpha,i}:= (C_{\alpha,i})^\circ$.
\end{itemize}
\item[Case 2.] If $\min(C_{\alpha,i})=0$ and $\otp(C_{\alpha,i})\le\cf(\Lambda)$, then in order to define $D_{\alpha,i}$, we first define an $\sq$-increasing and continuous sequence of subsets of $\alpha$,
$\langle D_{\alpha,i}^j \mid j<\otp(C_{\alpha,i}) \rangle$, in such a way that for all $j<\otp(C_{\alpha,i})$: $\acc^+(D_{\alpha,i}^{j})\s D_{\alpha,i}^{j}$, $\otp(D_{\alpha,i}^{j})<\Lambda$,
and $C_{\alpha,i}(j)<\sup(D_{\alpha,i}^{j+1})\le C_{\alpha,i}(j+1)$.

The definition is by recursion on $j<\otp(C_{\alpha,i})$, as follows:
\begin{itemize}
\item[$\br$] Let $D_{\alpha,i}^0:=\emptyset$.
\item[$\br$] Suppose that $j<\otp(C_{\alpha,i})$ and $D_{\alpha,i}^j$ has already been defined. To define $D_{\alpha,i}^{j+1}$, we consider two cases:
\begin{itemize}
\item[$\br\br$] If there exists $(\beta,k)$ such that $\beta\in\Gamma\setminus\Omega$, $C_{\alpha,i}(j)<\beta<C_{\alpha,i}(j+1)$, $D_{\alpha,i}^j \sqsubseteq D_{\beta,k}$, $\nacc(D_{\beta,k})\s X_{C_{\alpha,i}(j+1)}$, and
$\rho(C_{\alpha,i}\cap (C_{\alpha,i}(j+1))) \le \otp(D_{\beta,k})<\Lambda$, then put $D_{\alpha,i}^{j+1}:=D_{\beta,k}$ for the lexicographically-least such pair $(\beta,k)$.
\item[$\br\br$] If the above fails, then let $D_{\alpha,i}^{j+1}:=D_{\alpha,i}^j\cup \{ \sup(D_{\alpha,i}^j),C_{\alpha,i}(j+1)\}$.
\end{itemize}
\item[$\br$] Suppose that $j'\in\acc(\otp(C_{\alpha,i}))$ for which $\langle D_{\alpha,i}^j\mid j<j'\rangle$ has already been defined.
Put $D_{\alpha,i}^{j'}:=\bigcup_{j<j'}D_{\alpha,i}^j$. As $j'<\otp(C_{\alpha,i})\le\cf(\Lambda)$, we indeed have $\otp(D_{\alpha,i}^{j'})<\Lambda$.
\end{itemize}
At the end of the above process, let $D_{\alpha,i}:=\bigcup\{D_{\alpha,i}^j\mid j<\otp(C_{\alpha,i})\}$.
\item[Case 3.] If $\min(C_{\alpha,i})=0$ and $\otp(C_{\alpha,i})>\cf(\Lambda)$, then let $D_{\alpha,i} := C_{\alpha,i}\setminus (C_{\alpha,i}(\cf(\Lambda)))$.
\end{itemize}

Now that the matrix $\langle D_{\alpha,i}\mid \alpha,i<\kappa\rangle$ is defined,
for every $\alpha < \kappa$ we let $$\mathcal D_\alpha := \{ D_{\alpha,i} \mid (i=0)\text{ or }(0<i<\kappa\ \&\ \otp(D_{\alpha,i})<\xi)\}.$$
Put $\cvec{D}:=\langle \mathcal D_\alpha \mid \alpha < \kappa\rangle$.

\begin{claim}\label{claim561} For every $\alpha\in\acc(\kappa)$ and $i<\kappa$:
\begin{enumerate}
\item $|\{ D\in\mathcal D_\alpha\mid \otp(D)=\xi\}|\le1$;
\item $0<\left| \mathcal D_\alpha \right| \leq \left| \mathcal C_\alpha \right| < \mu$;
\item $D_{\alpha,i}$ is a club in $\alpha$ of order-type $\le\xi$, with $\acc(D_{\alpha,i})\cap\Omega=\emptyset$,
and if $D_{\alpha,i}$ was defined according to Case~2, then furthermore $\otp(D_{\alpha,i})\le\Lambda$;
\item If $\alpha\in\Gamma$, then for every $\bar\alpha \in \acc(D_{\alpha,i})$, we have $D_{\alpha,i} \cap \bar\alpha \in\mathcal D_{\bar\alpha}$;
\item If $\alpha\notin\Gamma$, then $\omega\in\acc(D_{\alpha,i})$ but $D_{\alpha,i} \cap\omega \notin \mathcal D_\omega$;
\item If $D_{\alpha,i}$ was defined according to Case 2, then for every $j<\otp(C_{\alpha,i})$, there exist $(\zeta,k)$ and $\varsigma \in [\kappa\setminus\zeta]^{<\omega}$
such that $\zeta \in (\Gamma \setminus \Omega) \cup\{0\}$ and $D_{\alpha,i}^j = D_{\zeta,k} \cup \varsigma$;
if $j$ is a limit ordinal, then furthermore $\varsigma = \emptyset$ and $\zeta = C_{\alpha,i}(j)$;
\item $\cvec{D}$ is a $\square_\xi(\kappa,{<}\mu,{\sq_\chi^\Omega},V)$-sequence, with $\Gamma(\cvec{D})=\Gamma$;
\item If $\Lambda<\kappa$ and $\cvec{C}$ witnesses $\square_\xi(\kappa,{<}\mu,{\sq^\Omega_\chi},{\nsin})$, then so does $\cvec{D}$.
\end{enumerate}
\end{claim}
\begin{proof}
\begin{enumerate}
\item Clear from the definition of $\mathcal D_\alpha$.
\item As $C_{\alpha,i}=C_{\alpha,i'}$ entails $D_{\alpha,i}=D_{\alpha,i'}$ for all $i<i'<\kappa$, we have $|\mathcal D_\alpha|\le|\{ C_{\alpha,i}\mid i<\kappa\}| = |\mathcal C_\alpha|<\mu$.

\item Recall that we have assumed $\Lambda \leq \xi$.
The claim is trivial for $\alpha\notin\Gamma$.
We now prove the required statement by induction on $\alpha \in \Gamma$, examining each case in the construction of $D_{\alpha,i}$:
\begin{itemize}
\item[Case 1.1:] In this case, the result that $D_{\alpha,i}$ is a club in $\alpha$ of order-type $\leq\xi$ follows from the facts that $D_{\min(\varsigma),k}$ is a club in $\min(\varsigma)$ of order-type $<\Lambda$,
$\varsigma$ is finite, $\max(\varsigma) \leq \min(C_{\alpha,i})$, $\Lambda$ is an indecomposable ordinal, and $C_{\alpha,i}$ is a club in $\alpha$ of order-type $\leq\xi$.
Then, $\acc(D_{\alpha,i}) \cap \Omega = \emptyset$ follows from $\acc(D_{\min(\varsigma),k}) \cap\Omega = \emptyset$, $\min(\varsigma) \notin\Omega$, and $\acc(C_{\alpha,i}) \cap\Omega =\emptyset$.
\item[Case 1.2:] Since $C_{\alpha,i}$ is a club in $\alpha$ of order-type $\leq\xi$, so is $D_{\alpha,i}$ in this case. Likewise, $\acc(D_{\alpha,i}) \cap \Omega = \emptyset$.
\item[Case 2:] Since the sequence $\langle D_{\alpha,i}^j \mid j<\otp(C_{\alpha,i}) \rangle$ is continuous and satisfies $C_{\alpha,i}(j)<\sup(D_{\alpha,i}^{j+1})\le C_{\alpha,i}(j+1)$ for all $j<\otp(C_{\alpha,i})$,
it follows that $\sup(D_{\alpha,i}) = \sup_{j < \otp(C_{\alpha,i}) } \sup(D_{\alpha,i}^j) = \sup(C_{\alpha,i}) = \alpha$.
Since $\acc^+(D_{\alpha,i}^{j})\s D_{\alpha,i}^{j}$ for all $j<\otp(C_{\alpha,i})$ and since the sequence $\langle D_{\alpha,i}^j \mid j<\otp(C_{\alpha,i}) \rangle$ is $\sq$-increasing,
it follows that $D_{\alpha,i}$ is a club in $\alpha$.
Since $\otp(D_{\alpha,i}^{j}) < \Lambda$ for every $j < \otp(C_{\alpha,i})$, it follows that $\otp(D_{\alpha,i}) = \sup_{j<\otp(C_{\alpha,i})} \otp(D_{\alpha,i}^j) \leq \Lambda$.
Finally, every $\bar\alpha \in \acc(D_{\alpha,i})$ is either in $\acc(C_{\alpha,i})$ or in $\acc(D_{\beta,k}) \cup \{\beta\}$ for some $(\beta,k)$ with $\beta \in \Gamma \setminus \Omega$ below $\alpha$,
all of which are disjoint from $\Omega$.

\item[Case 3:] Notice that if $\Lambda=\kappa$, then this case never applies.
Since $C_{\alpha,i}$ is a club in $\alpha$ with $\acc(C_{\alpha,i}) \cap\Omega =\emptyset$ and $\cf(\Lambda) < \otp(C_{\alpha,i}) \leq\xi$,
it follows that in this case, $D_{\alpha,i}$ is a club in $\alpha$  of order-type $\leq\xi$, and $\acc(D_{\alpha,i})\cap\Omega=\emptyset$.
\end{itemize}

\item Suppose not, and let $\alpha\in\Gamma$ be the least counterexample.
Pick $i<\kappa$ and $\bar\alpha\in\acc(D_{\alpha,i})$ such that $D_{\alpha,i}\cap\bar\alpha\notin\mathcal D_{\bar\alpha}$.
By Clause~(3), we have $\otp(D_{\alpha,i})\le\xi$, and hence $\otp(D_{\alpha,i}\cap\bar\alpha)<\xi$.
Thus, to derive a contradiction, it suffices to find some $k<\kappa$ such that $D_{\alpha,i}\cap\bar\alpha=D_{\bar\alpha,k}$.

Write $\beta:=\min(C_{\alpha,i})$. There are a few cases to consider:

\begin{itemize}
\item[Case 1.1:] If $\beta>0$ and $D_{\alpha,i}= D_{\min(\varsigma),k}\cup \varsigma \cup (C_{\alpha,i})^\circ$ for $(\varsigma,k):=\varphi(\beta)$, we have:
\begin{itemize}
\item[$\br$] If $\bar\alpha \in \acc(D_{\min(\varsigma),k})$, then by $\min(\varsigma) \in \Gamma \cap \alpha$ and minimality of $\alpha$, there must be some $k'<\kappa$
such that $D_{\bar\alpha,k'} = D_{\min(\varsigma),k}\cap\bar\alpha = D_{\alpha,i} \cap \bar\alpha$.
\item[$\br$] If $\bar\alpha = \min(\varsigma)$, then $D_{\bar\alpha,k} = D_{\alpha,i} \cap \bar\alpha$.
\item[$\br$] If $\bar\alpha \in \acc(C_{\alpha,i})$, then by $\alpha\in\Gamma$, there exists some $\bar k<\kappa$ such that $C_{\alpha,i}\cap\bar\alpha=C_{\bar\alpha,\bar k}$,
so that $\varphi(\min(C_{\bar\alpha,\bar k}))=\varphi(\beta)=(\varsigma,k)$ and $D_{\bar\alpha,\bar k}=D_{\min(\varsigma),k}\cup\varsigma\cup(C_{\bar\alpha,\bar k})^\circ$.
That is, $D_{\bar\alpha,\bar k}=D_{\alpha,i}\cap\bar\alpha$.
\end{itemize}
\item[Case 1.2:] If $D_{\alpha,i}=(C_{\alpha,i})^\circ$, then by $\alpha \in \Gamma$ and $\bar\alpha \in \acc(C_{\alpha,i})$, let $\bar k<\kappa$ be such that $C_{\bar\alpha,\bar k}=C_{\alpha,i}\cap\bar\alpha$.
We have $D_{\bar\alpha,\bar k}=(C_{\bar\alpha,\bar k})^\circ =D_{\alpha,i}\cap\bar\alpha$.
\item[Case 2:] If $D_{\alpha,i}=\bigcup\{D_{\alpha,i}^j\mid j<\otp(C_{\alpha,i})\}$, then let $j'<\otp(C_{\alpha,i})$ be the least such that $\bar\alpha\in\acc(D_{\alpha,i}^{j'})\cup\{\sup(D_{\alpha,i}^{j'})\}$.
By minimality of $\alpha$, it cannot be the case that $D_{\alpha,i}^{j'}=D_{\zeta,k}$ for some $\zeta\in\Gamma\cap\alpha$ and $k<\kappa$.
By minimality of $j'$, it cannot be the case that $D_{\alpha,i}^{j'}=D_{\alpha,i}^j\cup \{ \sup(D_{\alpha,i}^j),C_{\alpha,i}(j+1)\}$ for some $j<j'$.
Thus, it must be the case that $j'\in\acc(\otp(C_{\alpha,i}))$ and $D_{\alpha,i}^{j'}=\bigcup_{j<j'}D_{\alpha,i}^j$.
So, by minimality of $j'$, it follows that $\bar\alpha=\sup(\bigcup_{j<j'}D_{\alpha,i}^j)$.
As $C_{\alpha,i}(j)<\sup(D_{\alpha,i}^{j+1})\le C_{\alpha,i}(j+1)$ for all $j<j'$, we have that $\bar\alpha=C_{\alpha,i}(j')$.
Let $\bar k<\kappa$ be such that $C_{\alpha,i}\cap\bar\alpha=C_{\bar\alpha,\bar k}$.
We have $\bar\alpha \in \Gamma$, $\min(C_{\bar\alpha,\bar k}) = \min(C_{\alpha,i}) = 0$, and $\otp(C_{\bar\alpha,\bar k}) = j' < \otp(C_{\alpha,i}) \leq \cf(\Lambda)$,
so that $D_{\bar\alpha,\bar k}$ was also constructed according to Case 2.
By the canonical nature of that construction and the fact that $C_{\bar\alpha,\bar k}(j) = C_{\alpha,i}(j)$ for all $j < j'$,
it is easy to see that $\langle D_{\bar\alpha,\bar k}^j\mid j<\otp(C_{\bar\alpha,\bar k})\rangle=\langle  D_{\alpha,i}^j\mid j<j'\rangle$,
so that $D_{\bar\alpha,\bar k}=\bigcup_{j<j'}D^j_{\bar\alpha,\bar k}=D_{\alpha,i}\cap\bar\alpha$.
\item[Case 3:] If $D_{\alpha,i}=C_{\alpha,i}\setminus (C_{\alpha,i}(\cf(\Lambda)))$, then let $\bar k<\kappa$ be such that $C_{\bar\alpha,\bar k}=C_{\alpha,i}\cap\bar\alpha$.
We have $D_{\bar\alpha,\bar k}=C_{\bar\alpha,\bar k}\setminus (C_{\bar\alpha,\bar k}(\cf(\Lambda)))=D_{\alpha,i}\cap\bar\alpha$.
\end{itemize}
\item As $C_{\omega,k}=\omega\setminus 7$ for all $k<\kappa$, Case~1.2 of the construction implies that $\mathcal D_\omega=\{\omega\setminus 8\}$.
Now, if $\alpha\in\acc(\kappa)\setminus\Gamma$, then $D_{\alpha,i} = C_{\alpha,0} \sqsupseteq ((\omega+1)\setminus1)$, so that $\omega \in \acc(D_{\alpha,i})$ but $D_{\alpha,i}\cap\omega \notin\mathcal D_\omega$.
\item By induction on $j'<\otp(C_{\alpha,i})$.
The base case is trivial, since $D_{\alpha,i}^0=\emptyset=D_{0,0}$.
The nonzero limit case is easy, as in this case we have $D_{\alpha,i}^{j'}=\bigcup_{j<j'}D_{\alpha,i}^{j}$ which, by Clause~(4), is equal to $D_{\bar\alpha,k}$ for $\bar\alpha:=C_{\alpha,i}(j') \in \Gamma\setminus\Omega$ and some $k<\kappa$.

Finally, if $j'=j+1$, then either there exists some $(\zeta,k)$ such that $\zeta\in\Gamma\setminus\Omega$ and $D_{\alpha,i}^{j'}=D_{\zeta,k}$, and we are done,
or there exists some $\varsigma\in[\kappa\setminus \sup(D_{\alpha,i}^j)]^{\le2}$ such that $D_{\alpha,i}^{j'}=D_{\alpha,i}^j\cup\varsigma$, and we are done by appealing to the induction hypothesis.

\item Follows from all of the previous clauses.
\item Suppose that $A$ is a cofinal subset of $\kappa$ such that $A\cap\alpha\sin\mathcal D_\alpha$ for all $\alpha\in\acc^+(A)$.
Fix a function $f:\kappa\rightarrow\kappa$ such that $A\cap\alpha\s D_{\alpha,f(\alpha)}$ for all $\alpha\in \acc^+(A)$.
Assuming $\Lambda<\kappa$, we consider the club $D:=\{\alpha\in\acc(\kappa\setminus\Lambda)\mid \otp(A\cap\alpha)=\alpha\}$.
We shall show that $\{\alpha\in\Gamma\mid \sup(D\cap\alpha\setminus C_{\alpha,f(\alpha)})<\alpha\}$ is stationary,
so that, by Lemma~\ref{nsintransversal}, $\cvec{C}$ does not witness $\square_\xi(\kappa,{<}\mu,{\sq^\Omega_\chi},{\nsin})$.

Let $\alpha\in\Gamma\cap D$ be arbitrary.
As $D\s\acc^+(A)$, we have $D_{\alpha,f(\alpha)}\supseteq A\cap\alpha$, and hence $\otp(D_{\alpha,f(\alpha)})=\alpha>\Lambda$.
Thus, by Clause~(3) of this Claim, $D_{\alpha,f(\alpha)}$ was not constructed according to Case~2.
By definition of Cases 1 and 3, $\epsilon:=\sup(D_{\alpha,f(\alpha)}\symdiff C_{\alpha,f(\alpha)})$ is $<\alpha$.
It follows that $A \cap(\epsilon,\alpha)\s C_{\alpha,f(\alpha)}$, and since $D\s\acc^+(A)$ and $C_{\alpha,f(\alpha)}$ is closed, we infer that $\sup(D\cap\alpha\setminus C_{\alpha,f(\alpha)}) \leq\epsilon <\alpha$. \qedhere
\end{enumerate}
\end{proof}

Let $D_\alpha := D_{\alpha,0}$ for all $\alpha \in \Gamma$, so that $\langle D_\alpha \mid \alpha \in \Gamma \rangle$ is a transversal for $\cvec{D}$.
All that remains is to verify Clause~(3) of the statement of the Lemma.
As a preliminary step, we prove the following:

\begin{claim}\label{claim5632} Suppose that $A$ is a cofinal subset of $\kappa$, $\Lambda' < a(\Lambda,\kappa)$, $\zeta \in(\Gamma\setminus\Omega)\cup\{0\}$,
$i<\kappa$, and $\otp(D_{\zeta,i})<\Lambda$.
Then there are stationarily many $\alpha \in \Gamma\setminus\Omega$ such that, for some $k<\kappa$:
\begin{itemize}
\item $D_{\zeta,i}\sq D_{\alpha, k}$;
\item $\nacc(D_{\alpha,k})\setminus\zeta\s A$;
\item $\otp(D_{\zeta,i}) + \Lambda' \le \otp(D_{\alpha, k})<\Lambda$.
\end{itemize}
\end{claim}
\begin{proof} If $\zeta \in \Gamma\setminus\Omega$, then fix a nonzero $\beta<\kappa$ such that $\varphi(\beta)=(\{\zeta\}, i)$.
Applying hypothesis~$(\ref{hyp-c-nacc})$ to $B:=(A\setminus\beta)\cup\{\beta\}$ and $\Lambda'$,
there are stationarily many $\alpha \in \Gamma\setminus\Omega$ such that, for some $C \in \mathcal C_\alpha$,
$\min(C)=\min(B)$, $\Lambda' \le\otp(C)<\Lambda$, and $\nacc(C)\s B$.
Consider any such $\alpha$, and let $k<\kappa$ be such that $C = C_{\alpha,k}$.
As $\min(C_{\alpha,k})=\beta$, we have $\varphi(\min(C_{\alpha,k}))=(\{\zeta\},i)$ and $\nacc((C_{\alpha,k})^\circ) \subseteq B \setminus \{\beta\} \subseteq A$.
As $\beta>0$, $\zeta \in \Gamma\setminus\Omega$, and $\otp(D_{\zeta,i})<\Lambda$, we are in Case 1.1 of the construction, and hence $D_{\alpha,k} = D_{\zeta,i}\cup\{\zeta\} \cup (C_{\alpha,k})^\circ$,
so that $\otp(D_{\zeta,i}) + \Lambda' \le\otp(D_{\alpha,k})<\Lambda$ (since $\Lambda$ is an indecomposable ordinal, and $\otp((C_{\alpha,k})^\circ) = \otp(C_{\alpha,k})$) and
$\nacc(D_{\alpha,k})\setminus\zeta = \nacc((C_{\alpha,k})^\circ) \s A$.

If $\zeta =0$, then fix a nonzero $\beta<\kappa$ such that $\varphi(\beta)=(\emptyset, i)$.
Let $\alpha \in \Gamma\setminus\Omega$ and $k<\kappa$ be arbitrary satisfying $\min(C_{\alpha,k})=\beta$, $\Lambda' \le\otp(C_{\alpha,k})<\Lambda$, and $\nacc(C_{\alpha,k})\s (A\setminus\beta)\cup\{\beta\}$.
Then we are in Case 1.2 of the construction, and hence $D_{\alpha,k}:= (C_{\alpha,k})^\circ$, so that $\emptyset\sq D_{\alpha,k}$, $\nacc(D_{\alpha,k})\s A$, and $\Lambda' \leq \otp(D_{\alpha,k}) < \Lambda$.
Since $D_{\zeta,i} = \emptyset$ in this case, we are done.
\end{proof}

\begin{claim}\label{467} Suppose that $A$ is a cofinal subset of $\kappa$. Then there exists a stationary $S\s\kappa$ for which the set
$\{\delta\in\Gamma\mid \otp(D_\delta)=\min\{\delta,\Lambda\}, \nacc(D_\delta)\s A\}$
covers the set
$$\{\delta\in \Gamma\mid \min(C_\delta)=\min(S), \otp(C_{\delta})=\cf(\min\{\delta,\Lambda\}), \nacc(C_{\delta})\s S \}.$$
\end{claim}
\begin{proof} Let $\langle M_\eta\mid \eta<\kappa \rangle$ be an $\in$-increasing and continuous chain of elementary submodels of $H_{\kappa^{+}}$,
such that $\{\Gamma,\Omega,\rho, A, \Lambda, \langle \Lambda_j \mid j<\cf(\Lambda) \rangle, \langle D_{\alpha,i} \mid \alpha,i<\kappa\rangle \}\s M_0$ and $M_\eta\cap\kappa\in\kappa$ for all $\eta<\kappa$.
Clearly, $E:=\{ \gamma<\kappa\mid M_\gamma\cap\kappa=\gamma\}$ is a club in $\kappa$,
and $G:=\{\gamma\in E\mid X_\gamma=A\cap\gamma\}$ is stationary in $\kappa$.
Let $S:=\{0\}\cup G$.

Now, suppose that $\delta\in\Gamma$ is an ordinal satisfying $\min(C_{\delta,0})=\min(S)$, $\otp(C_{\delta,0})=\cf(\min\{\delta,\Lambda\})$, and $\nacc(C_{\delta,0})\s S$.
For all $j<\otp(C_{\delta,0})$, write $\varrho_j:=\rho(C_{\delta,0}\cap (C_{\delta,0}(j)))$.
Note that if $\Lambda<\kappa$, then by $\Lambda \in M_0$, we have $\min(E)>\Lambda$, so that $\delta>\Lambda$ and $\otp(C_{\delta,0}) = \cf(\Lambda)$.
It follows that, in all cases, $\otp(C_{\delta,0}) = \cf(\delta) \leq \cf(\Lambda)$, so that
$\langle \varrho_j\mid j<\otp(C_{\delta,0})\rangle$ is a well-defined, continuous sequence, converging to $\min\{\delta,\Lambda\}$.

For every $j<\otp(C_{\delta,0})$, let $\Lambda'_j < a(\Lambda,\kappa)$ be some ordinal such that $\varrho_j + \Lambda'_j \geq \varrho_{j+1}$.\footnote{%
If $\Lambda<\kappa$, then we may simply let $\Lambda'_j := \Lambda_{ j+1}$.
If $\Lambda=\kappa$, then we may let $\Lambda'_j := \varrho_{j+1}$, recalling that $a(\Lambda,\kappa)=\kappa$ in that case.}

As $\delta \in \Gamma$ and $\otp(C_{\delta,0})\le\cf(\Lambda)$, we know that $D_{\delta,0}$ was constructed according to Case~2.
We shall now show that for all $j<\otp(C_{\delta,0})$:
$D_{\delta,0}^{j} = D_{\beta,k}$ for some pair $(\beta,k)$ with $\beta \in (\Gamma\setminus\Omega) \cup \{0\}$,
$\varrho_j\le\otp(D_{\delta,0}^j)<\Lambda$, and $\nacc(D_{\delta,0}^j)\s A$. By induction on $j$:
\begin{itemize}
\item[$\br$] For $j=0$, we have $D_{\delta,0}^0 = \emptyset = D_{0,0}$, so that $\varrho_0 = \rho(\emptyset) = 0 = \otp(D_{\delta,0}^0) < \Lambda$ and $D_{\delta,0}^0 \s A$.
\item[$\br$] Suppose that $j<\otp(C_{\delta,0})$ satisfies the induction hypothesis.
Fix a pair $(\zeta,i)$ with $\zeta \in (\Gamma\setminus\Omega) \cup \{0\}$ such that $D_{\delta,0}^j = D_{\zeta,i}$, and $i$ is the least to satisfy this.
By the induction hypothesis, $\otp(D_{\zeta,i})<\Lambda$, so that by appealing to Claim~\ref{claim5632} with $\Lambda' := \Lambda'_j$, there are stationarily many $\beta \in \Gamma\setminus\Omega$ such that, for some $l<\kappa$,
$D_{\zeta,i} \sq D_{\beta,l}$, $\nacc(D_{\beta,l}) \setminus \zeta \subseteq A$, and $\otp(D_{\zeta,i}) + \Lambda'_j \leq \otp(D_{\beta,l}) < \Lambda$.
Recall that $\zeta = \sup(D_{\zeta,i})$, and by the induction hypothesis $\nacc(D_{\zeta,i})\s A$, so that for any $\beta$ and $l$ as above we have $\nacc(D_{\beta,l}) = \nacc(D_{\zeta,i}) \cup (\nacc(D_{\beta,l}) \setminus \zeta) \subseteq A$.
Also, our choice of $\Lambda'_j$ together with the induction hypothesis gives $\otp(D_{\zeta,i}) + \Lambda'_j \geq \varrho_{j+1}$. Thus,
\[
H_{\kappa^+} \models
\sup \{\beta \in \Gamma\setminus\Omega \mid \exists l<\kappa
[D_{\zeta,i}\sq D_{\beta,l}, \nacc(D_{\beta,l})\s A, \varrho_{j+1} \le\otp(D_{\beta,l})<\Lambda] \} = \kappa.
\]

Write $\gamma := C_{\delta,0}(j+1)$.
Then $\gamma \in S\setminus\{0\} \subseteq E$, and hence $M_\gamma \cap \kappa = \gamma$.
We have $\zeta = \sup(D_{\zeta,i}) = \sup(D_{\delta,0}^j) \leq C_{\delta,0}(j) < \gamma$, so that $\zeta \in M_\gamma$ and $\mathcal D_\zeta \in M_\gamma$.
As $M_\gamma \prec H_{\kappa^+}$,  $M_\gamma\cap\kappa=\gamma$, and $|\mathcal D_\zeta| < \kappa$,  we have $\{ D_{\zeta,\tau}\mid \tau<\kappa\}=\{ D_{\zeta,\tau}\mid \tau<\gamma\}$.
In particular, $D_{\zeta,i} \in M_\gamma$.
Next, notice that $\varrho_{j+1}$ is equal to either $\sum_{\iota<j+1}\Lambda_\iota$ (in case $\Lambda<\kappa$)
or $C_{\delta,0}(j)$ (in case $\Lambda=\kappa$), which are both ordinals below $\gamma$, and hence in $M_\gamma$.
Now, since $\{\Gamma,\Omega, \varrho_{j+1},A, \Lambda, \kappa, \langle D_{\alpha,i} \mid \alpha,i<\kappa\rangle \}\s M_\gamma$, elementarity of $M_\gamma$ gives
\[
\sup\{\beta \in \Gamma\cap\gamma \setminus\Omega \mid \exists l<\gamma[D_{\zeta,i}\sq D_{\beta,l}, \nacc(D_{\beta,l})\s A, \varrho_{j+1} \le\otp(D_{\beta,l})<\Lambda] \}=\gamma.
\]

As $\gamma=C_{\delta,0}(j+1)\in S\setminus\{0\} = G$, we have $X_{C_{\delta,0}(j+1)}=A\cap\gamma$.
Altogether, there exists $\beta \in \Gamma\setminus\Omega$ and $l<\gamma$
with $C_{\delta,0}(j)<\beta<C_{\delta,0}(j+1)$ such that $D_{\delta,0}^j=D_{\zeta,i}\sq D_{\beta,l}$, $\nacc(D_{\beta,l})\s X_{C_{\delta,0}(j+1)}$ and $\varrho_{j+1} \le\otp(D_{\beta,l})<\Lambda$.
In particular, $(\beta,l)$ witnesses that $D_{\delta,0}^{j+1}$ was constructed according to the first option,
so that in fact $D_{\delta,0}^{j+1} = D_{\beta,k}$ for some pair $(\beta,k)$ with $\beta \in \Gamma\setminus\Omega$,
$\varrho_{j+1} \le \otp(D_{\delta,0}^{j+1})<\Lambda$, and $\nacc(D_{\delta,0}^{j+1}) \s A$.

\item[$\br$] Suppose that $j' \in \acc(\otp(C_{\delta,0}))$, and for every $j < j'$, $D_{\delta,0}^{j} = D_{\beta,k}$ for some pair $(\beta,k)$ with $\beta \in (\Gamma\setminus\Omega) \cup \{0\}$,
$\varrho_j\le\otp(D_{\delta,0}^j)<\Lambda$, and $\nacc(D_{\delta,0}^j)\s A$.
By Claim~\ref{claim561}(6), $D_{\delta,0}^{j'} = D_{\bar\alpha,k}$ for  $\bar\alpha := C_{\delta,0}(j')$ and some $k<\kappa$.
Furthermore, $\bar\alpha\in\Gamma\setminus\Omega$ and  $\otp(D_{\delta,0}^{j'}) < \Lambda$.
Since $D_{\delta,0}^{j'} =\bigcup_{j<j'}D_{\delta,0}^j$, it follows that $\otp(D_{\delta,0}^{j'}) = \sup_{j<j'} \otp(D_{\delta,0}^j) \geq \sup_{j<j'} \varrho_j = \varrho_{j'}$
and $\nacc(D_{\delta,0}^{j'}) = \bigcup_{j<j'} \nacc(D_{\delta,0}^j) \subseteq A$, completing the induction.
\end{itemize}

Since $\delta \in \Gamma$, $\min(C_{\delta,0}) = 0$, and $\otp(C_{\delta,0}) \leq \cf(\Lambda)$,
Case 2 of the construction gives us $D_{\delta,0} = \bigcup_{j< \otp(C_{\delta,0})} D_{\delta,0}^j$, so that $\nacc(D_{\delta,0}) \subseteq A$ and
$\min\{\delta,\Lambda\}=\sup\{\varrho_j\mid j<\otp(C_{\delta,0})\}\le\otp(D_{\delta,0})\le \min\{\Lambda,\delta\}$, completing the proof of the Claim.
\end{proof}

So $\langle D_\alpha\mid\alpha\in\Gamma\rangle$ is a transversal for $\cvec{D}$ that satisfies the desired properties.
\end{proof}

\begin{cor}\label{cor48} Suppose that $\langle (C_\alpha,Z_\alpha)\mid\alpha <\kappa\rangle$ is a sequence such that:
\begin{itemize}
\item For every $\alpha\in\acc(\kappa)$, $C_\alpha$ is a club in $\alpha$;
\item For every $\alpha\in\acc(\kappa)$ and $\bar\alpha\in\acc(C_\alpha)$, $C_\alpha\cap\bar\alpha=C_{\bar\alpha}$;
\item For every subset $Z\s\kappa$ and every club $D\s\kappa$, we have $$\sup\{\otp(C_\alpha)\mid \alpha\in\acc(\kappa), C_\alpha\s\{ \beta\in D\mid Z\cap\beta=Z_\beta\}\}=\kappa.$$
\end{itemize}

Then there exists a $\square(\kappa)$-sequence $\langle D_\alpha\mid\alpha <\kappa \rangle$ satisfying that
for every sequence $\langle A_i\mid i<\kappa\rangle$ of cofinal subsets of $\kappa$, and every $\theta\in\reg(\kappa)$, the following set is stationary:
$$\{\alpha\in E_\theta^\kappa \mid \otp(D_\alpha)=\alpha\ \&\ \forall i<\alpha[D_\alpha(i+1)\in A_i]\}.$$
\end{cor}
\begin{proof} Note that $\langle C_\alpha\mid\alpha\in\acc(\kappa)\rangle$ is a transversal for $\square(\kappa,{<}2,{\sq},V)$, and that $\langle Z_\alpha\mid\alpha<\kappa\rangle$ witnesses $\diamondsuit(\kappa)$.
Derive a $\kappa$-assignment $\mathfrak Z := \langle Z_{x,\beta}\mid x\in\mathcal K(\kappa),\beta\in\nacc(x)\rangle$ via the rule $Z_{x,\beta} := Z_\beta\cap \beta$,
and let $\Phi_{\mathfrak Z}$ be the corresponding $\acc$-preserving postprocessing function given by Lemma~\ref{phiZ}.
Then by Lemma~\ref{pp-preserves-square}, $\vec{C^\circ} := \langle \Phi_{\mathfrak Z}(C_\alpha)\mid\alpha\in\acc(\kappa)\rangle$ is a transversal for $\square(\kappa,{<}2,{\sq},V)$.

We shall show that $\vec{C^\circ}$ and $(\chi,\Lambda,\xi,\mu,\Omega):=(\aleph_0,\kappa,\kappa,2,\emptyset)$ satisfy the hypotheses of Lemma~\ref{blowup-nacc}.
Evidently, only Clause~$(\ref{hyp-c-nacc})$ of that Lemma requires an argument.
For this, consider arbitrary cofinal $B \subseteq \kappa$, $\Lambda'\in\acc(\kappa)$ and a club $C\s\kappa$.
Now, by applying the hypothesis above with $Z := B$ and $D := \acc^+(B)\cap C$, let us pick $\alpha \in \acc(\kappa)$ with $\otp(C_\alpha) > \Lambda'$ such that $C_\alpha\s\{ \beta\in \acc^+(B)\cap C\mid B\cap\beta=Z_\beta\}$.
Put $\bar\alpha:=C_\alpha(\Lambda')$.
Then ${\bar\alpha} \in C_\alpha \subseteq C$, $\otp(\Phi_{\mathfrak Z}(C_{\bar\alpha}))=\otp(C_{\bar\alpha})=\otp(C_\alpha\cap\bar\alpha)=\Lambda'$, and as explained in Example~\ref{phiZ-simpler},
it follows that $\min(\Phi_{\mathfrak Z}(C_{\bar\alpha})) = \min(B)$ and $\nacc(\Phi_{\mathfrak Z}(C_{\bar\alpha})) \subseteq B$.

Thus, let $\vec D=\langle D_\alpha \mid\alpha\in\acc(\kappa)\rangle$ be the corresponding transversal for $\square(\kappa,{<}2, {\sq}, V)$ produced by Lemma~\ref{blowup-nacc}.
By $\diamondsuit(\kappa)$, let $\Phi$ be given by Lemma~\ref{phi0}. Write $D_0^\bullet:=\emptyset$, $D^\bullet_{\alpha+1}:=\{\alpha\}$ for all $\alpha<\kappa$, and $D_\alpha^\bullet:=\Phi(D_\alpha)$ for all $\alpha\in\acc(\kappa)$.

\begin{claim}\label{4.10.1} For every sequence $\langle A_i\mid i<\kappa\rangle$ of cofinal subsets of $\kappa$ and every $\theta\in\reg(\kappa)$, the following set is stationary:
\[
\{\alpha\in E_\theta^\kappa \mid \otp(D_\alpha^\bullet)=\alpha\ \&\ \forall i<\alpha[D_\alpha^\bullet(i+1)\in A_i]\}.
\]
\end{claim}
\begin{proof} Let $\vec A=\langle A_i\mid i<\kappa\rangle$ be a sequence of cofinal subsets of $\kappa$, let $\theta\in\reg(\kappa)$ be arbitrary, and let $C$ be an arbitrary club in $\kappa$.
By our choice of $\Phi$, fix a stationary set $G \subseteq \kappa$ that encodes $\vec A$ as in Lemma~\ref{phi0}.
Then apply Clause~(3) of Lemma~\ref{blowup-nacc} with $A := G$ to obtain a corresponding stationary set $S \subseteq\kappa$.
Next, by applying the hypothesis above with $Z := S$ and $D := \acc^+(S)\cap C$, let us pick $\alpha \in \acc(\kappa)$ with $\otp(C_\alpha) > \theta$ such that $C_\alpha\s\{ \beta\in \acc^+(S)\cap C\mid S\cap\beta=Z_\beta\}$.
Put $\bar\alpha:=C_\alpha(\theta)$.
Then ${\bar\alpha} \in C\cap E^\kappa_\theta$, $\otp(\Phi_{\mathfrak Z}(C_{\bar\alpha}))=\otp(C_{\bar\alpha})=\otp(C_\alpha\cap\bar\alpha)=\theta=\cf({\bar\alpha})$, and as explained in Example~\ref{phiZ-simpler},
it follows that $\min(\Phi_{\mathfrak Z}(C_{\bar\alpha})) = \min(S)$ and $\nacc(\Phi_{\mathfrak Z}(C_{\bar\alpha})) \subseteq S$.
Thus, by our choice of $S$, it follows that $\otp(D_{\bar\alpha}) = \bar\alpha$ and $\nacc(D_{\bar\alpha}) \subseteq G$.
Since $\Phi$ is $\acc$-preserving, we have $\otp(D^\bullet_{\bar\alpha}) = \bar\alpha$, and by our choice of $G$, it follows that $D^\bullet_{\bar\alpha}(i+1) \in A_i$ for all $i<\bar\alpha$, as sought.
\end{proof}

In particular, by feeding a constant sequence into Claim~\ref{4.10.1}, we see that for every club $D \subseteq\kappa$ there is $\alpha \in E^\kappa_\omega$ such that $\nacc(D_\alpha^\bullet) \setminus\{\min(D_\alpha^\bullet)\} \subseteq D$.
Thus, by Clause~(3) of Lemma~\ref{hitting-implies-nontrivial}, $\langle D_\alpha^\bullet\mid\alpha <\kappa \rangle$ is a $\square(\kappa)$-sequence, as sought.
\end{proof}

We now prove Theorem~\ref{thm3}:

\begin{cor} The following are equivalent:
\begin{enumerate}
\item $\diamondsuit(\omega_1)$ holds;
\item There exists a $\square(\omega_1)$-sequence $\langle C_\alpha\mid\alpha<\omega_1\rangle$ such that for every sequence $\langle A_i\mid i<\omega_1\rangle$ of cofinal subsets of $\omega_1$, the following set is stationary:
$$\{\alpha<\omega_1\mid \otp(C_\alpha)=\alpha\ \&\ \forall i<\alpha[C_\alpha(i+1)\in A_i]\}.$$
\end{enumerate}
\end{cor}
\begin{proof} We focus on the forward implication. By \cite[Lemma~3.5]{paper22}, $\diamondsuit(\omega_1)$ entails the existence of a sequence  $\langle (C_\alpha,Z_\alpha)\mid \alpha<\omega_1\rangle$ such that:
\begin{itemize}
\item For every $\alpha\in\acc(\omega_1)$, $C_\alpha$ is a club in $\alpha$;
\item For every $\alpha\in\acc(\omega_1)$ and $\bar\alpha\in\acc(C_\alpha)$, $C_\alpha\cap\bar\alpha=C_{\bar\alpha}$;
\item For every subset  $Z\s\omega_1$, club $D\s\omega_1$ and $\epsilon\in\acc(\omega_1)$, there exists $\alpha\in\acc(\omega_1)$ with $\otp(C_\alpha)=\epsilon$ such that $C_\alpha\s\{ \beta\in D\mid Z\cap\beta=Z_\beta\}$.
\end{itemize}

Now, appeal to Corollary~\ref{cor48} with $\kappa := \aleph_1$.
\end{proof}

\begin{cor}\label{cor46} Assume $V=L$ and that $\kappa$ is an inaccessible cardinal that is not weakly compact.

Then there exists a $\square(\kappa)$-sequence $\langle C_\alpha\mid\alpha <\kappa \rangle$ satisfying that
for every sequence $\langle A_i\mid i<\kappa\rangle$ of cofinal subsets of $\kappa$, and every $\theta\in\reg(\kappa)$, the following set is stationary:
$$\{\alpha\in E_\theta^\kappa\mid \otp(C_\alpha)=\alpha\ \&\ \forall i<\alpha[C_\alpha(i+1)\in A_i]\}.$$
\end{cor}
\begin{proof} Work in $L$. As hinted in \cite[Theorem 3.2]{Sh:347}, the proof of \cite[\S2]{AShS:221} essentially shows that if $\kappa$ is an inaccessible cardinal that is not weakly compact,
then there exists a sequence $\langle (C_\alpha,Z_\alpha)\mid\alpha<\kappa\rangle$  such that:
\begin{itemize}
\item For every $\alpha\in\acc(\kappa)$, $C_\alpha$ is a club in $\alpha$;
\item For every $\alpha\in\acc(\kappa)$ and $\bar\alpha \in\acc(C_{\alpha})$, $C_{\bar\alpha}=C_\alpha\cap\bar\alpha$ and $Z_{\bar\alpha}=Z_\alpha\cap \bar\alpha$;
\item For every subset $Z\s\kappa$, club $D\s\kappa$, and $\epsilon\in\acc(\kappa)$, there exist stationarily many singular cardinals $\alpha<\kappa$ with $\otp(C_\alpha)=\epsilon$, $Z_\alpha=Z\cap\alpha$ and $\acc(C_\alpha)\s D$.
\end{itemize}

Let $\Phi^\Sigma$ be given by Fact~\ref{newPhiSigma} for $\Sigma := \acc(\kappa)$.
Put $C_\alpha^\circ:=\Phi^\Sigma(C_\alpha)$ for all $\alpha\in\acc(\kappa)$.
Clearly, $\langle C_\alpha^\circ\mid \alpha \in \acc(\kappa) \rangle$ is a transversal for $\square(\kappa,{<}2,{\sq},V)$.
Furthermore, for every subset  $Z\s\kappa$, club $D\s\kappa$ and a regular uncountable $\epsilon<\kappa$, we may pick $\alpha<\kappa$ with $\otp(C_\alpha)= \epsilon$, $Z_\alpha=Z\cap\alpha$ and $\acc(C_\alpha)\s D$,
so that $C_\alpha^\circ\s\acc(C_\alpha)\s\{ \beta \in D \mid Z \cap \beta = Z_\beta \}$ and $\otp(C_\alpha^\circ) =\cf(\alpha)=\epsilon$.

Now appeal to Corollary~\ref{cor48}.
\end{proof}

\begin{defn}[{\cite[Definition~1.5]{paper22}}]\label{proxy} The principle $\p^-(\kappa,2,\mathcal R_0,\theta,\mathcal S)$ asserts the existence of an $\mathcal R_0$-coherent $C$-sequence,
$\langle C_\alpha \mid \alpha < \kappa \rangle$, such that for every sequence $\langle A_i \mid i < \theta \rangle$ of cofinal subsets of $\kappa$ and every $S \in \mathcal S$,
there are stationarily many $\alpha \in S$ satisfying $\sup(\nacc(C_\alpha)\cap A_i)=\alpha$ for all $i<\min\{\theta,\alpha\}$.
\end{defn}

If we omit $\mathcal S$, then we mean that $\mathcal S=\{\kappa\}$.

\begin{cor}\label{P-from-V=L} If $V=L$, then $\p^-(\kappa,2,{\sq},\kappa,\{E^\kappa_{\ge\theta}\mid \theta\in\reg(\kappa)\ \&\ \forall\lambda<\kappa(\lambda^{<\theta}<\kappa)\})$ holds for every (regular uncountable cardinal) $\kappa$ that is not weakly compact.
\end{cor}
\begin{proof} The case that $\kappa$ is a successor cardinal was established already in \cite[Corollary~1.10(5)]{paper22}.
For $\kappa$ inaccessible that is not weakly compact, let $\vec C=\langle C_\alpha\mid\alpha<\kappa\rangle$ be given by Corollary~\ref{cor46}.
To see that $\vec C$ witnesses $\p^-(\kappa,2,{\sq},\kappa,\{E^\kappa_{\theta}\mid \theta\in\reg(\kappa)\})$,
let $\langle A_i\mid i<\kappa\rangle$ be an arbitrary sequence of cofinal subsets of $\kappa$, and let $\theta\in\reg(\kappa)$ be arbitrary.
Fix a surjection $\pi:\kappa\rightarrow\kappa$ such that the preimage of any singleton is cofinal in $\kappa$.
Consider the club $D:=\diagonal_{j<\kappa}\acc^+(\pi^{-1}\{j\})$.
By the choice of $\vec C$, the set $G:=\{\alpha\in E^\kappa_\theta\cap D\mid \otp(C_\alpha)=\alpha\ \&\ (\forall i<\alpha) C_{\alpha}(i+1)\in A_{\pi(i)}\}$ is stationary.
Let $\alpha\in G$ be arbitrary. Then $\sup(\nacc(C_\alpha)\cap A_j)=\alpha$ for all $j<\alpha$.
\end{proof}

\begin{cor}\label{cor37} Suppose that $\lambda$ is an uncountable strong-limit cardinal, $\ch_\lambda$ holds, and $\chi \in \reg(\lambda^+)$.

Then $\square(\lambda^+,{\sq_\chi})$ is equivalent to $\p^-(\lambda^+,2,{\sq_\chi},1,\{E^{\lambda^+}_{\theta}\mid \theta\in\reg(\lambda)\})$.
\end{cor}
\begin{proof} The forward implication follows from Theorem~\ref{mixing_paper24}, and the verification is left to the reader.

For the inverse implication, fix a sequence $\langle C_\alpha \mid \alpha < \lambda^+ \rangle$ witnessing $\p^-(\lambda^+,2,{\sq_\chi},1,\{E^{\lambda^+}_{\theta}\mid\theta\in\reg(\lambda)\})$.
Let $\Gamma := \{ \alpha \in \acc(\lambda^+) \mid \forall \bar\alpha \in \acc(C_\alpha)[C_\alpha \cap\bar\alpha = C_{\bar\alpha}]\}$.

Fix some standard $C$-sequence $\langle e_\alpha\mid\alpha<\kappa\rangle$, and define $\vec D := \langle D_\alpha \mid \alpha < \lambda^+ \rangle$ by stipulating:
\[D_\alpha := \begin{cases}
C_\alpha, &\text{if } \alpha \in \Gamma;\\
e_\alpha, &\text{otherwise}.
\end{cases}\]

Clearly, $\Gamma(\vec{D}) \supseteq \Gamma$.
Also, the fact that each $e_\alpha$ comes from a standard $C$-sequence ensures $\sq_\chi$-coherence of $\vec{D}$, as well as $E^{\lambda^+}_\omega \subseteq \Gamma$.
Recalling Clause~(3) of Lemma~\ref{hitting-implies-nontrivial}, we are done.
\end{proof}

Just like Theorem~\ref{transversal_of_mixing_paper24} is an analogue of Clause~(2) of Fact~\ref{clubguessing}, the next theorem is an analogue of Clause~(1).

\begin{lemma}\label{mixing_rinot11} Suppose that $\ch_\lambda$ holds for an uncountable cardinal $\lambda$, $\varepsilon<\lambda$ with $\reg(\lambda) \nsubseteq (\varepsilon+1)$,
and
$\langle C_\alpha \mid \alpha \in \Gamma \rangle$ is a $C$-sequence over some stationary subset $\Gamma\s \acc(\lambda^+)$,
satisfying $$\sup\{\theta\in\reg(\lambda)\mid \{\alpha\in E^{\lambda^+}_{\theta} \cap \Gamma \mid \otp(C_\alpha)<\alpha\}\text{ is stationary in } \lambda^+ \}=\sup(\reg(\lambda)).$$

Then there exists a postprocessing function $\Phi:\mathcal K(\lambda^+)\rightarrow\mathcal K(\lambda^+)$ such that:
\begin{enumerate}
\item For all $x\in\mathcal K(\lambda^+)$ with $\otp(x)\le\varepsilon$, we have $\Phi(x)=x$;
\item For all $x\in\mathcal K(\lambda^+)$ with $\otp(x)=\lambda$, we have $\otp(\Phi(x))=\cf(\lambda)$;
\item For cofinally many $\theta\in\reg(\lambda)$, for every cofinal $A\s\lambda^+$, there exist stationarily many $\alpha\in\Gamma$ for which $\otp(\Phi(C_\alpha))=\theta$ and $\nacc(\Phi(C_\alpha))\s A$.
\end{enumerate}
\end{lemma}
\begin{proof} Let $\Sigma$ be a club in $\lambda$ such that $\otp(\Sigma)=\cf(\lambda)$ and $\min(\Sigma)=\varepsilon$.
Let $\Phi^\Sigma$ be the conservative postprocessing function given by Fact~\ref{newPhiSigma}, and denote $C_\delta^\circ:=\Phi^\Sigma(C_\delta)$.
By the hypothesis of the Lemma together with Example~\ref{example14}, $\langle C_\delta^\circ\mid \delta\in\Gamma,\allowbreak \otp(C_\delta)<\delta\rangle$ is an amenable $C$-sequence.
By Lemma~\ref{split_amenable2}, there exist a postprocessing function $\Phi : \mathcal K(\lambda^+) \to \mathcal K(\lambda^+)$, and a cofinal subset $\Theta\s \reg(\lambda)$,
such that $S_\theta:=\{ \delta\in E^{\lambda^+}_\theta\cap\Gamma\mid \otp(C_\delta)<\delta\ \&\allowbreak\ \min(\Phi(C^\circ_\delta))=\theta\}$ is stationary for all $\theta\in\Theta$.
By thinning out, we may assume that $\min(\Theta)>\varepsilon$ and $\cf(\lambda)\notin\Theta$.

For each $\theta\in\Theta$, let $\Phi_{\theta}$ be the postprocessing function given by Lemma~\ref{phi1} when fed with $\langle \Phi(C^\circ_\delta)\mid \delta\in S_\theta\rangle$.
Define $\Phi'$ by stipulating:
$$\Phi'(x):=\begin{cases}
\Phi_\theta(\Phi(\Phi^\Sigma(x))),&\text{if }\otp(x)>\varepsilon\ \&\ \theta:=\min(\Phi(\Phi^\Sigma(x)))\text{ is in }\Theta;\\
\Phi(\Phi^\Sigma(x)), &\text{if } \otp(x) > \varepsilon \text{ but } \min(\Phi(\Phi^\Sigma(x))) \notin\Theta;\\
x, &\text{if } \otp(x) \leq \varepsilon.
\end{cases}$$

\begin{claim} $\Phi'$ is a postprocessing function.
\end{claim}
\begin{proof} Let $x \in \mathcal K(\lambda^+)$ be arbitrary. We consider several cases:
\begin{itemize}
\item[$\br$] Suppose $\otp(x)>\varepsilon$ and $\min(\Phi(\Phi^\Sigma(x)))=\theta$ for some $\theta \in \Theta$.
Then $\Phi'(x) = \Phi_\theta(\Phi(\Phi^\Sigma(x)))$ is a club in $\sup(x)$, and $\acc(\Phi'(x)) = \acc(\Phi_\theta(\Phi(\Phi^\Sigma(x)))) \subseteq \acc(\Phi(\Phi^\Sigma(x))) \subseteq \acc(\Phi^\Sigma(x)) \subseteq \acc(x)$.
Consider arbitrary $\bar\alpha \in \acc(\Phi'(x))$, and we will compare $\Phi'(x) \cap\bar\alpha$ with $\Phi'(x \cap\bar\alpha)$.
Since $\otp(x) > \varepsilon = \min(\Sigma)$, Fact~\ref{newPhiSigma}(\ref{PhiSigma9}) gives $\Phi^\Sigma(x) \subseteq x \setminus x(\varepsilon)$,
so that $\bar\alpha \in \acc(\Phi^\Sigma(x))$ entails $\otp(x \cap \bar\alpha) > \varepsilon$.
Also, $\Phi(\Phi^\Sigma(x \cap\bar\alpha)) = \Phi(\Phi^\Sigma(x)) \cap\bar\alpha$, so that $\min(\Phi(\Phi^\Sigma(x \cap\bar\alpha))) = \min(\Phi(\Phi^\Sigma(x))) = \theta$.
Thus, $\Phi'(x \cap\bar\alpha) = \Phi_\theta(\Phi(\Phi^\Sigma(x \cap\bar\alpha))) = \Phi_\theta(\Phi(\Phi^\Sigma(x))) \cap\bar\alpha = \Phi'(x) \cap\bar\alpha$, as required.
\item[$\br$] Suppose $\otp(x) > \varepsilon$ but $\min(\Phi(\Phi^\Sigma(x))) \notin \Theta$.
Then $\Phi'(x) = \Phi(\Phi^\Sigma(x))$ is a club in $\sup(x)$, and $\acc(\Phi'(x)) = \acc(\Phi(\Phi^\Sigma(x))) \subseteq \acc(\Phi^\Sigma(x)) \subseteq \acc(x)$.
Consider arbitrary $\bar\alpha \in \acc(\Phi'(x))$.
As before, since $\otp(x) > \varepsilon = \min(\Sigma)$, Fact~\ref{newPhiSigma}(\ref{PhiSigma9}) gives $\Phi^\Sigma(x) \subseteq x \setminus x(\varepsilon)$, so that $\bar\alpha \in \acc(\Phi^\Sigma(x))$ entails $\otp(x \cap \bar\alpha) > \varepsilon$.
Also, $\Phi(\Phi^\Sigma(x \cap\bar\alpha)) = \Phi(\Phi^\Sigma(x)) \cap\bar\alpha$, so that $\min(\Phi(\Phi^\Sigma(x \cap\bar\alpha))) = \min(\Phi(\Phi^\Sigma(x))) \notin \Theta$.
Thus, $\Phi'(x \cap\bar\alpha) = \Phi(\Phi^\Sigma(x \cap\bar\alpha)) = \Phi(\Phi^\Sigma(x)) \cap\bar\alpha = \Phi'(x) \cap\bar\alpha$, as required.
\item[$\br$] Suppose $\otp(x) \leq \varepsilon$. Then $\Phi'(x) = x$ is a club in $\sup(x)$, and $\acc(\Phi'(x)) = \acc(x)$.
Consider arbitrary $\bar\alpha \in \acc(\Phi'(x))$. Then $\otp(x \cap\bar\alpha) < \otp(x) \leq \varepsilon$, so that $\Phi'(x \cap\bar\alpha) = x \cap\bar\alpha = \Phi'(x) \cap\bar\alpha$, as required. \qedhere
\end{itemize}
\end{proof}

We now verify that $\Phi'$ satisfies the numbered properties in the statement of the theorem:

\begin{enumerate}
\item Straight from the definition of $\Phi'$ in this case.
\item Let $x\in\mathcal K(\lambda^+)$ with $\otp(x)=\lambda$ be arbitrary.
By $\otp(x)=\sup(\Sigma) > \varepsilon$ and Fact~\ref{newPhiSigma}(\ref{PhiSigma3}), we have
$$\cf(\lambda)=\cf(\sup(x))\le\otp(\Phi'(x))\le\otp(\Phi^\Sigma(x))=\otp(\Sigma)=\cf(\lambda).$$
\item Let $\theta \in \Theta$ be arbitrary.
Suppose that $A$ is a given cofinal subset of $\lambda^+$.
By the choice of $\Phi_\theta$, there are stationarily many $\alpha \in S_\theta$ such that $\otp(\Phi_\theta(\Phi(C_\alpha^\circ))) = \cf(\alpha)$ and
$\nacc(\Phi_\theta(\Phi(C_\alpha^\circ))) \subseteq A$.
Consider any such $\alpha$.
As $\alpha \in S_\theta$, we have $\otp(C_\alpha) \geq \otp(C_\alpha^\circ) \geq \cf(\alpha)=\theta \geq\min(\Theta) > \varepsilon$ and $\min(\Phi(C^\circ_\alpha))=\theta$.
Thus $\Phi'(C_\alpha) = \Phi_\theta(\Phi(C_\alpha^\circ))$, so that $\Phi'(C_\alpha)$ satisfies the required properties.
\end{enumerate}

This completes the proof.
\end{proof}

Recall the definition of $\acts{\Phi}{C}$ given by Notation~\ref{notationaction}.

\begin{lemma}\label{pp-preserves-square2} Suppose that $\Phi : \mathcal K(\kappa) \to \mathcal K(\kappa)$ is a postprocessing function and $\Omega\subseteq\kappa\setminus\{\omega\}$. Then:
\begin{enumerate}
\item If $\cvec{C}=\langle\mathcal C_\alpha \mid \alpha < \kappa\rangle$ is a witness to $\square_\xi(\kappa,{<}\mu, {\sq_\chi^\Omega},{\nsin})$,
then $\acts{\Phi}{C}$ is yet another witness, with some support $\Gamma^\bullet\supseteq\Gamma(\cvec{C})$;
if $\Phi$ is faithful or $\chi=\aleph_0$, then $\Gamma^\bullet=\Gamma(\cvec{C})$;
\item Suppose that $\langle C_\alpha\mid\alpha\in\Gamma\rangle$ is a transversal for $\square_\xi(\kappa,{<}\mu, {\sq_\chi^\Omega},{\nsin})$.
If $\Phi$ is faithful or $\chi=\aleph_0$, then $\langle \Phi(C_\alpha)\mid\alpha\in\Gamma\rangle$ is yet another transversal for $\square_\xi(\kappa,{<}\mu, {\sq_\chi^\Omega},{\nsin})$.
\end{enumerate}
\end{lemma}
\begin{proof} Let $\cvec{C}=\langle\mathcal C_\alpha \mid \alpha < \kappa\rangle$ be a witness to $\square_\xi(\kappa,{<}\mu, {\sq_\chi^\Omega},{\nsin})$.
Denote $\acts{\Phi}{C}$ by $\cvec{D}=\langle \mathcal D_\alpha\mid\alpha<\kappa\rangle$.

\begin{claim}\label{pp-acc-subset} For every $\alpha\in\acc(\kappa)$ and every $D \in \mathcal D_\alpha$,
there is some $C \in \mathcal C_\alpha$ such that $\acc(D\setminus\omega)  \subseteq \acc(C)$.
\end{claim}
\begin{proof} Fix arbitrary $\alpha\in\acc(\kappa)$ and $D \in \mathcal D_\alpha$, and consider two cases:
\begin{itemize}
\item[$\br$] If $\alpha\in\Gamma(\cvec{C})$, then $D = \Phi(C)$ for some $C \in \mathcal C_\alpha$. Since $\Phi$ is a postprocessing function, it follows that $\acc(D) \subseteq \acc(C)$.
\item[$\br$] If $\alpha\notin\Gamma(\cvec{C})$, then $\mathcal C_\alpha$ is a singleton, say, $\mathcal C_\alpha = \{C\}$, where by $\sq_\chi$-coherence of $\cvec{C}$ we must have $\nacc(C) \subseteq \nacc(\alpha)$.
In particular, $\sup(C\cap\nacc(\alpha)) = \alpha$, so that by definition of $\acts{\Phi}{C}$ (Notation~\ref{notationaction}), we must have $D \setminus(\omega+1) \subseteq C$, so that $\acc(D\setminus\omega)\s\acc(C)$. \qedhere
\end{itemize}
\end{proof}

As $\cvec{C}$, in particular, witnesses $\square_\xi(\kappa,{<}\mu, {\sq_\chi},V)$, Lemma~\ref{pp-preserves-square} entails that so does $\cvec{D}$.
Moreover, $\cvec{D}$ witnesses $\square_\xi(\kappa,{<}\mu, {\sq^\Omega_\chi},V)$,  as follows from the $\sq_\chi^\Omega$-coherence of $\cvec{C}$ together with Claim~\ref{pp-acc-subset}.

Finally, towards a contradiction, suppose that there exists a cofinal $A\s\kappa$ such that $A\cap\alpha\sin\mathcal D_\alpha$ for all $\alpha\in\acc^+(A)$.
Since $D := \acc^+(A\setminus \omega)$ is cofinal in $\kappa$, let us pick $\alpha\in\acc^+(D)$ such that $D\cap\alpha\nsin\mathcal C_\alpha$.
In particular, $\alpha\in\acc^+(A)$, and we may pick $D_\alpha\in\mathcal D_\alpha$ such that $A\cap\alpha\s D_\alpha$.
By Claim~\ref{pp-acc-subset}, fix $C_\alpha \in \mathcal C_\alpha$ such that $\acc(D_\alpha\setminus\omega)\s\acc(C_\alpha)$.
Then $D\cap\alpha=\acc^+(A\cap\alpha\setminus\omega)\s\acc(D_\alpha\setminus\omega)\s\acc(C_\alpha)$, contradicting the choice of $\alpha$.
\end{proof}

A combined application of Lemmas \ref{blowup-nacc} and \ref{mixing_rinot11} yields the following.

\begin{cor}\label{cor413} Suppose that $\kappa=\lambda^+$ for a given singular cardinal $\lambda$, and
\begin{enumerate}[(a)]
\item $\ch_\lambda$ holds;
\item $\cvec{C}=\langle \mathcal C_\alpha\mid \alpha<\kappa\rangle$ is a $\square_\xi(\kappa,{<}\mu,{\sq^\Omega_\chi},{\nsin})$-sequence for some fixed subset $\Omega\s\kappa\setminus\{\omega\}$;
\item $\langle C_\alpha \mid \alpha \in \Gamma \rangle$ is a transversal for $\cvec{C}$, satisfying that for cofinally many $\theta\in\reg(\lambda)$,
the set $\{\alpha\in E^{\kappa}_{\theta} \cap \Gamma \mid \otp(C_\alpha)<\alpha\}$ is stationary in $\kappa$;
\item\label{hyp-d} $\Gamma'$ is a subset of $\Gamma$ satisfying that for every stationary $A\s\kappa$, there exists $\alpha\in\Gamma'$ with $\otp(C_\alpha)=\cf(\lambda)$ such that $\nacc(C_\alpha)\s A$.
\end{enumerate}

Then there exists a witness $\cvec{D}=\langle \mathcal D_\alpha\mid\alpha<\kappa\rangle$ to $\square_\xi(\kappa,{<}\mu,{\sq^\Omega_\chi},{\nsin})$,
with a transversal $\langle D_\alpha\mid\alpha\in\Gamma\rangle$, satisfying the following properties:
\begin{enumerate}
\item $|\mathcal D_\alpha| \leq |\mathcal C_\alpha|$ for all $\alpha<\kappa$;
\item For every cofinal $A\s\kappa$, there exist stationarily many $\alpha\in\Gamma'$ such that $\otp(D_\alpha)= \lambda$ and $\nacc(D_\alpha)\s A$.
\end{enumerate}
\end{cor}
\begin{proof} We begin by feeding $\langle C_\alpha \mid \alpha\in\Gamma \rangle$ and $\varepsilon := \cf(\lambda)$ into Lemma~\ref{mixing_rinot11}, to obtain a corresponding postprocessing function $\Phi_0$.
By $\ch_\lambda$ and~\cite{Sh:922}, $\diamondsuit(\kappa)$ holds, and we may let $\Phi_1$ be the $\acc$-preserving postprocessing function given by Lemma~\ref{phi0}.
Write $\cvec{C^\circ} = \langle \mathcal C_\alpha^\circ \mid \alpha<\kappa\rangle$ for $\acts{\Phi_1\circ\Phi_0}{C}$.
Denote $C_\alpha^\circ:= \Phi_1(\Phi_0(C_\alpha))$.
By Lemma~\ref{pp-preserves-square2}, since $\Phi_1$ is faithful, $\cvec{C^\circ}$ is a witness to $\square_\xi(\kappa,{<}\mu,{\sq^\Omega_\chi},{\nsin})$, with transversal $\langle C_\alpha^\circ \mid \alpha \in \Gamma \rangle$.

\begin{claim} $\Lambda := \lambda$, $\cvec{C^\circ}$, and $\langle C_\alpha^\circ \mid \alpha \in \Gamma \rangle$ satisfy the hypotheses of Lemma~\ref{blowup-nacc}.
\end{claim}
\begin{proof} Since $\kappa=\lambda^+$, we necessarily have $\xi\ge\lambda$.
Thus, to verify Clause~$(\ref{hyp-c-nacc})$ of Lemma~\ref{blowup-nacc}, let $B$ be an arbitrary cofinal subset of $\kappa$, let $\Lambda'<a(\lambda, \kappa)$ be an arbitrary ordinal, and let $D$ be an arbitrary club in $\kappa$.
Recalling Example~\ref{examples-a}, and since $\lambda$ is a singular cardinal, we have $a(\lambda, \kappa) = \lambda$, so that $\Lambda'<\lambda$.
By slightly increasing $\Lambda'$, we may assume that it is a limit ordinal.
By the choice of $\Phi_1$, fix a stationary set $G \subseteq \kappa$ such that for all $x \in \mathcal K(\kappa)$ with $\nacc(x) \subseteq G$, we have $\min(\Phi_1(x)) = \min(B)$ and $\nacc(\Phi_1(x)) \subseteq B$.
Consider the stationary set $A:=G\cap D$.
Since $\reg(\lambda)$ is cofinal in the limit cardinal $\lambda$, and since $\Phi_0$ was given to us by Lemma~\ref{mixing_rinot11},
Clause~(3) of that Lemma allows us to pick $\theta \in \reg(\lambda)$ above $\Lambda'$ along with $\alpha \in \Gamma$ for which $\otp(\Phi_0(C_\alpha)) = \theta$ and $\nacc(\Phi_0(C_\alpha)) \subseteq A$.
Let $\bar\alpha:=(\Phi_0(C_\alpha))(\Lambda')$.
Clearly, $\bar\alpha\in D\cap \Gamma\setminus\Omega$ and $C_\alpha\cap\bar\alpha\in\mathcal C_{\bar\alpha}$.
In particular, $C:=\Phi_1(\Phi_0(C_\alpha\cap\bar\alpha))$ is in $\mathcal C_{\bar\alpha}^\circ$.
As $\Phi_1$ is $\otp$-preserving, we have that $\otp(C) = \Lambda'$.
Finally, by the choice of $G$, and by $\nacc(\Phi_0(C_\alpha\cap\bar\alpha))\s\nacc(\Phi_0(C_\alpha))\s A\s G$, we have $\min(C) = \min(B)$ and $\nacc(C) \subseteq B$.
\end{proof}

Appeal to Lemma~\ref{blowup-nacc} with $\Lambda := \lambda$, $\cvec{C^\circ}$, and $\langle C_\alpha^\circ \mid \alpha \in \Gamma \rangle$,
and let $\langle\mathcal D_\alpha\mid\alpha<\kappa\rangle$ and $\langle D_\alpha\mid\alpha \in \Gamma \rangle$ be the corresponding output.
As $\Lambda=\lambda<\kappa$, Clause~(2) of Lemma~\ref{blowup-nacc} entails that $\cvec{D}$ is a $\square_\xi(\kappa,{<}\mu,{\sq^\Omega_\chi},{\nsin})$-sequence.
We are left with verifying the numbered conclusions in the statement:

(1) By Lemma~\ref{blowup-nacc}(1) and Notation~\ref{notationaction}, $|\mathcal D_\alpha| \leq |\mathcal C_\alpha^\circ|\leq |\mathcal C_\alpha|$ for all $\alpha<\kappa$.

(2) Recalling Clause~(3) of Lemma~\ref{blowup-nacc}, let $S$ be an arbitrary stationary subset of $\kappa$, and let $D$ be an arbitrary club in $\kappa$.
We need to find $\alpha\in\Gamma'\cap D$ such that $\min(C_\alpha^\circ)=\min(S)$, $\otp(C_\alpha^\circ)=\cf(\lambda)$ and $\nacc(C_\alpha^\circ)\s S$.
By the choice of $\Phi_1$, fix a stationary set $G \subseteq \kappa$ such that for all $x \in \mathcal K(\kappa)$ with $\nacc(x) \subseteq G$, we have $\min(\Phi_1(x)) = \min(S)$ and $\nacc(\Phi_1(x)) \subseteq S$.
Consider the stationary set $A:=G\cap D$.
By Clause~$(\ref{hyp-d})$ of the hypothesis, let us pick $\alpha\in\Gamma'$ such that $\otp(C_\alpha)=\cf(\lambda)$ and $\nacc(C_\alpha)\s A$.
In particular, $\alpha\in D$.
Finally, by $\otp(C_\alpha)\le\varepsilon$ and the choice of $\Phi_0$, we have $\nacc(\Phi_0(C_\alpha))=\nacc(C_\alpha)\s A\s G$, so that $C_\alpha^\circ=\Phi_1(C_\alpha)$,
$\otp(C_\alpha^\circ)=\otp(C_\alpha)=\cf(\lambda)$, $\min(C_\alpha^\circ)=\min(S)$ and $\nacc(C_\alpha^\circ)\s S$.
\end{proof}

\begin{fact}[\cite{paper28}]\label{Phi3} For any subset $B\s\kappa$, define $\Phi:\mathcal K(\kappa)\rightarrow\mathcal K(\kappa)$ by stipulating:
$$\Phi(x):=\begin{cases}
\cl(\nacc(x)\cap B),&\text{if }\sup(\nacc(x)\cap B)=\sup(x);\\
x\setminus\sup(\nacc(x)\cap B),&\text{otherwise}.
\end{cases}$$

Then $\Phi$ is a conservative postprocessing function.
\end{fact}

\begin{fact}[\cite{paper28}]\label{paper28} If $\diamondsuit(\kappa)$ holds, then for every $\epsilon\in\acc(\kappa)$,
there exists a postprocessing function $\Phi:\mathcal K(\kappa)\rightarrow\mathcal K(\kappa)$ satisfying the following.

For every sequence $\langle A_i\mid i<\kappa\rangle$ of cofinal subsets of $\kappa$, there exists some stationary set $G$ in $\kappa$ such that for all $x\in\mathcal K(\kappa)$, if any of the following hold:
\begin{enumerate}
\item $\sup(\nacc(x)\cap G)=\sup(x)$, $\otp(x)\le\epsilon$, and $(\cf(\sup(x)))^+=\kappa$;
\item $\otp(\nacc(x)\cap G)=\sup(x)>\epsilon$;
\item $\otp(x)$ is a cardinal $\le\epsilon$ whose successor is $\kappa$, and $\nacc(x)\s G$,
\end{enumerate}
then $\sup(\nacc(\Phi(x))\cap A_i)=\sup(x)$ for all $i<\sup(x)$.
\end{fact}

The following proof invokes almost all of the machinery developed in Sections \ref{section1}, \ref{section2}, \ref{section3}, and~\ref{section4},
cumulatively applying a dozen different postprocessing functions
in order to obtain the required $C$-sequence.

\begin{thm}\label{big-thm} Suppose that $\kappa=\lambda^+$ for a singular strong-limit cardinal $\lambda$, $\Omega\s\kappa\setminus\{\omega\}$,
and $\square_\xi(\kappa,{<}\mu,{\sq^\Omega_\chi},{\nsin})+\ch_\lambda$ holds.
Then there exists a transversal $\langle C_\alpha\mid\alpha\in\Gamma\rangle$ for $\square_\xi(\kappa,{<}\mu,{\sq^\Omega_\chi},{\nsin})$
such that for every sequence $\langle A_i\mid i<\kappa\rangle$ of cofinal subsets of $\kappa$, there are stationarily many $\alpha\in\Gamma$
with $\sup(\nacc(C_\alpha) \cap A_i)=\alpha$ for all $i<\alpha$.
\end{thm}
\begin{proof} Let $\vec D=\langle D_\alpha\mid\alpha\in\Gamma\rangle$ be an arbitrary transversal for $\square_\xi(\kappa,{<}\mu,{\sq^\Omega_\chi},{\nsin})$.
By Lemma~\ref{nsintransversal}, $\vec D$ is amenable.
Since $\lambda$ is a singular strong-limit, $\Theta:=\reg(\lambda)\setminus\max\{(\cf(\lambda))^+,\chi\}$ is a subset of $\{\theta\in\reg(\lambda)\mid \mathcal D(\lambda,\theta)=\lambda\}$ satisfying $\sup(\Theta)=\lambda$.
By Lemma~\ref{Gamma-closure}(3), $\{\alpha<\kappa\mid \cf(\alpha)\in\Theta\}$ is a stationary subset of $\Gamma$,
so that $\langle D_\alpha\mid \alpha<\kappa,\cf(\alpha)\in\Theta\rangle$ is also amenable.
Thus, by appealing to Theorem~\ref{transversal_of_mixing_paper24},
we may fix a faithful postprocessing function $\Phi_0:\mathcal K(\kappa)\rightarrow\mathcal K(\kappa)$ and a cofinal subset $\Theta_0\s\Theta$
such that for every $\theta\in\Theta_0$ and every cofinal $A\s\kappa$, there exists $\alpha\in E^{\kappa}_\theta$ with $\sup(\nacc(\Phi_0(D_\alpha))\cap A)=\alpha$.

For all $\alpha\in\Gamma$, let $C_\alpha:=\Phi_0(D_\alpha)$.
By Lemma~\ref{pp-preserves-square2}(2), $\langle C_\alpha\mid\alpha\in\Gamma\rangle$ is yet another transversal for $\square_\xi(\kappa,{<}\mu,{\sq^\Omega_\chi},{\nsin})$.

Next, by $\ch_\lambda$ and \cite{Sh:922}, $\diamondsuit(\kappa)$ holds,
so let us fix a matrix $\langle A^i_\gamma\mid i,\gamma<\kappa \rangle $ as in Fact~\ref{diamond_matrix}.
We consider two cases:

\underline{Case 1.} Suppose that there exists some $\theta\in\reg(\lambda)$ such that for every cofinal $A\s\kappa$,
the following set is nonempty: $$\{\alpha\in E^{\kappa}_\theta \cap\Gamma \mid \otp(\{\gamma\in\nacc(C_\alpha)\mid \sup(A\cap\gamma)=\gamma\ \&\ A\cap\gamma=A^\theta_\gamma\})=\alpha\}.$$
Fix such a $\theta$, derive a $\kappa$-assignment $\mathfrak Z=\langle Z_{x,\beta}\mid x\in\mathcal K(\kappa), \beta\in\nacc(x)\rangle$ via the rule $Z_{x,\beta}:=A_\beta^\theta$,
and let $\Phi_{\mathfrak Z}$ be the postprocessing function given by Lemma~\ref{phiZ}.
In addition, let $\faithful$  be given by Example~\ref{faithful_correction}, and let $\Phi$ be given by Fact~\ref{paper28} for $\epsilon := \lambda$.

For all $\alpha\in\Gamma$, denote $C_\alpha^\bullet:=\faithful(\Phi(\Phi_{\mathfrak Z}(C_\alpha)))$.
Since $\faithful\circ\Phi\circ\Phi_{\mathfrak Z}\circ\Phi_0$ is faithful, Lemma~\ref{pp-preserves-square2}(2) implies that $\vec{C^\bullet}=\langle C^\bullet_\alpha\mid\alpha\in\Gamma\rangle$ is a transversal
for $\square_\xi(\kappa,{<}\mu,{\sq^\Omega_\chi},{\nsin})$. Thus, we are left with proving:

\begin{claim}\label{claim4171} For every sequence $\langle A_i\mid i<\kappa\rangle$ of cofinal subsets of $\kappa$, there exist stationarily many $\alpha\in\Gamma$
with $\sup(\nacc(C^\bullet_\alpha)\cap A_i)=\alpha$ for all $i<\alpha$.
\end{claim}
\begin{proof} Let $D\s\kappa$ be an arbitrary club, and let $\vec A=\langle A_i\mid i<\kappa\rangle$ be an arbitrary sequence of cofinal subsets of $\kappa$.
Let $G\s\kappa$ be the corresponding stationary set that encodes $\vec A$, as given by Fact~\ref{paper28}.
Consider the stationary set $A:=(G\cap D)\setminus\lambda$.
By the choice of $\theta$, we can pick some $\alpha\in E^{\kappa}_\theta \cap\Gamma$ such that
$\otp(\{ \beta \in \nacc(C_\alpha) \cap \acc^+(A) \mid A \cap \beta = A^\theta_\beta \}) = \otp(C_\alpha) = \alpha$.
In particular, $\alpha\in D$, and as explained in Example~\ref{phiZ-simpler},
$\otp(\nacc(\Phi_{\mathfrak Z}(C_\alpha))\cap A) = \otp(C_\alpha) = \alpha>\lambda$.
Thus, by Fact~\ref{paper28}(2),  we also have $\sup(\nacc(\Phi(\Phi_{\mathfrak Z}(C_\alpha)))\cap A_i)=\alpha$ for all $i<\alpha$.
As, for all $x\in\mathcal K(\kappa)$, $\faithful(x)$ differs from $x$ by at a most a single element, $\sup(\nacc(C^\bullet_\alpha)\cap A_i)=\alpha$ for all $i<\alpha$.
\end{proof}

\underline{Case 2.} Suppose that for every $\theta\in\reg(\lambda)$, there exists a cofinal $A^\theta\s\kappa$,
for which the set $$G^\theta:=\{\alpha\in E^{\kappa}_\theta \cap\Gamma \mid \otp(\{\gamma\in\nacc(C_\alpha)\mid \sup(A^\theta\cap\gamma)=\gamma\ \&\ A^\theta\cap\gamma=A^\theta_\gamma\})=\alpha\}$$ is empty.
Fix such a set $A^\theta$ for each $\theta\in\reg(\lambda)$, and let $\Phi_1$ be the postprocessing function given by Fact~\ref{Phi3}  when fed with the following set,
which is stationary by our choice of the matrix $\langle A^i_\gamma \mid i,\gamma<\kappa \rangle$:
$$B:=\{\gamma<{\kappa}\mid \forall \theta\in\reg(\lambda)[\sup(A^\theta\cap\gamma)=\gamma\ \&\ A^\theta\cap\gamma=A^\theta_\gamma]\}.$$

\begin{claim}\label{claim4212} $S_\theta:=\{\alpha\in E^{\kappa}_\theta\cap\Gamma\mid \otp(\Phi_1(C_\alpha))<\alpha\}$ is stationary for cofinally many $\theta\in\reg(\lambda)$.
\end{claim}
\begin{proof} Recall that $\Theta_0$ is a cofinal subset of $\reg(\lambda) \setminus\chi$, and that $E^\kappa_{\geq\chi} \subseteq\Gamma$.
Fix an arbitrary $\theta\in\Theta_0$ and an arbitrary club $D\s\kappa$, and we shall find $\alpha\in E^{\kappa}_\theta\cap D$ such that $\otp(\Phi_1(C_\alpha))<\alpha$.

Consider the stationary set $A:=B\cap D$. By the property that $\Phi_0$ was chosen to satisfy, pick $\alpha\in E^{\kappa}_\theta$ such that $\sup(\nacc(\Phi_0(D_\alpha))\cap A)=\alpha$.
Then $\alpha\in D$ and $\sup(\nacc(C_\alpha)\cap A)=\alpha$.
Now, by definition of $\Phi_1$, we obtain $\Phi_1(C_{\alpha}) = \cl(\nacc(C_{\alpha}) \cap B)$.
For each $\gamma\in B$, we have $\sup(A^\theta\cap\gamma)=\gamma$ and $A^\theta\cap\gamma=A^\theta_\gamma$, so that
\[
\nacc(C_{\alpha}) \cap B \subseteq
\{\gamma\in\nacc(C_{\alpha})\mid
\sup(A^\theta\cap\gamma)=\gamma\ \&\ A^\theta\cap\gamma=A^\theta_\gamma\}.
\]
Since $G^\theta$ is empty, in particular $\alpha \notin G^\theta$,
and it follows from the above that $\otp(\nacc(C_{\alpha})\cap B)<{\alpha}$.
Altogether, we infer that $\otp(\Phi_1(C_{\alpha})) = \otp(\nacc(C_{\alpha} \cap B)) <\alpha$,
so that $\alpha\in S_\theta\cap D$.
\end{proof}

For all $\alpha\in\Gamma$, put $C_\alpha^1:=\faithful(\Phi_1(C_\alpha))$, so that, in particular, $\{ \alpha\in E^\kappa_{>\cf(\lambda)} \cap\Gamma \mid \otp(C^1_\alpha)<\alpha \}$ is stationary.

By Lemma~\ref{pp-preserves-square2}(2), let us fix a $\square_\xi(\kappa,{<}\mu,{\sq^\Omega_\chi},{\nsin})$-sequence $\cvec{C}$ for which $\langle C^1_\alpha \mid \alpha\in\Gamma \rangle$ is a transversal.
By Lemma~\ref{lemma3.5}, let us fix a transversal $\langle D^1_\alpha\mid\alpha\in\Gamma\rangle$ for the very same $\cvec{C}$,
along with a faithful postprocessing function $\Phi_2$ such that for every cofinal $A\s\kappa$, there exist stationarily many $\alpha\in\Gamma$ with $\otp(\Phi_2(D_\alpha^1))=\cf(\lambda)$
and $\nacc(\Phi_2(D_\alpha^1))\s A$. For all $\alpha\in\Gamma$, let
$$C^2_\alpha:=\begin{cases}
\Phi_2(D^1_\alpha),&\text{if }\cf(\alpha)=\cf(\lambda);\\
\Phi_2(C^1_\alpha),&\text{otherwise}.
\end{cases}$$

\begin{claim} $\acts{\Phi_2}{C}$ and $\langle C^2_\alpha \mid \alpha \in \Gamma \rangle$ satisfy the hypothesis of Corollary~\ref{cor413}, with $\Gamma'=\Gamma$.
\end{claim}
\begin{proof} By Lemma~\ref{pp-preserves-square2}, $\acts{\Phi_2}{C}$ is a $\square_\xi(\kappa,{<}\mu,{\sq^\Omega_\chi},{\nsin})$-sequence,
for which $\langle \Phi_2(C^1_\alpha)\mid\alpha\in\Gamma\rangle$ and $\langle \Phi_2(D^1_\alpha)\mid\alpha\in\Gamma\rangle$ are transversals.
In particular, $\langle C^2_\alpha \mid \alpha \in \Gamma \rangle$ is a transversal for $\acts{\Phi_2}{C}$.
By Claim~\ref{claim4212}, Clause~(c) of Corollary~\ref{cor413} holds.
By the very choice of $\Phi_2$, Clause~(\ref{hyp-d}) of Corollary~\ref{cor413} holds for $\Gamma'=\Gamma$.
\end{proof}

Let $\langle D^2_\alpha\mid\alpha \in \Gamma \rangle$ be the transversal for $\square_\xi(\kappa,{<}\mu,{\sq^\Omega_\chi},{\nsin})$ produced by Corollary~\ref{cor413}.
As in Case~1, let $\Phi$ be given by Fact~\ref{paper28} for $\epsilon := \lambda$.
For all $\alpha\in\Gamma$, denote $C_\alpha^\bullet:=\faithful(\Phi(D^2_\alpha))$, so that, by Lemma~\ref{pp-preserves-square2}(2),
$\vec{C^\bullet}=\langle C^\bullet_\alpha\mid\alpha\in\Gamma\rangle$ is a transversal for $\square_\xi(\kappa,{<}\mu,{\sq^\Omega_\chi},{\nsin})$.
Thus, we are left with proving:
\begin{claim} For every sequence $\langle A_i\mid i<\kappa\rangle$ of cofinal subsets of $\kappa$, there exist stationarily many $\alpha\in\Gamma$ with $\sup(\nacc(C^\bullet_\alpha)\cap A_i)=\alpha$ for all $i<\alpha$.
\end{claim}
\begin{proof}
Let $\vec A=\langle A_i\mid i<\kappa\rangle$ be an arbitrary sequence of cofinal subsets of $\kappa$.
Let $G\s\kappa$ be the corresponding stationary set that encodes $\vec A$, as given by Fact~\ref{paper28}.
Since $G$ is cofinal in $\kappa$, $S:=\{\alpha\in\Gamma\mid \otp(D_\alpha^2)=\lambda\text{ and }\nacc(D^2_\alpha)\s G\}$ is stationary by our choice of $\langle D_\alpha^2 \mid \alpha\in\Gamma \rangle$.
It now follows from Fact~\ref{paper28}(3) that for every $\alpha\in S$, we have $\sup(\nacc(C_\alpha^\bullet)\cap A_i)=\alpha$ for all $i<\alpha$.
\end{proof}
This completes the proof.
\end{proof}

Theorem~\ref{thm4} is the special case $\chi := \aleph_0$ of the following.

\begin{cor}\label{thm-using-everything} Suppose that $\lambda$ is a singular strong-limit cardinal, $\chi \in \reg(\lambda)$, and $\square(\lambda^+,{\sq_\chi})+\ch_\lambda$ holds.
Then $\p^-(\lambda^+,2,{\sq_\chi},\lambda^+)$ holds.
\end{cor}
\begin{proof} Recalling Lemma~\ref{hitting-implies-nontrivial}(4) and appealing to Theorem~\ref{big-thm} with $(\kappa,\Omega,\xi,\mu) := (\lambda^+,\emptyset,\lambda^+,2)$,
we obtain a transversal $\vec{C} = \langle C_\alpha \mid \alpha\in\Gamma \rangle$ for $\square(\lambda^+,{\sq_\chi})$
such that for every sequence $\langle A_i \mid i<\lambda^+ \rangle$ of cofinal subsets of $\lambda^+$,
there are stationarily many $\alpha\in\Gamma$ such that $\sup(\nacc(C_\alpha) \cap A_i) =\alpha$ for all $i<\alpha$.
Let $\vec{C^\bullet}$ be an arbitrary $\sq_\chi$-coherent $C$-sequence over $\lambda^+$, satisfying $\vec{C^\bullet}\restriction\Gamma=\vec{C}$.
Then $\vec{C^\bullet}$ witnesses $\p^-(\lambda^+,2,{\sq_\chi},\lambda^+)$.
\end{proof}

Theorem~C now follows:

\begin{cor}\label{thmC} Assume either of the following:
\begin{itemize}
\item $V=L$ and $\kappa$ is a regular uncountable cardinal that is not weakly compact;
\item $\kappa=\lambda^+$, where $\lambda$ is a strong-limit singular cardinal and $\square(\lambda^+)+\ch_\lambda$ holds.
\end{itemize}

Then there exists a uniformly coherent, prolific $\kappa$-Souslin tree.
\end{cor}
\begin{proof} By Corollary~\ref{P-from-V=L} or Corollary~\ref{thm-using-everything} (with $\chi:=\aleph_0$), $\p^-(\kappa,2,{\sq},\kappa)$ holds.
By $V=L$ or $\ch_\lambda$ and \cite{Sh:922}, $\diamondsuit(\kappa)$ holds.
Now, by \cite[Proposition~2.5, Remark~2.6]{paper22}, $\p^-(\kappa,2,{\sq},\kappa) + \diamondsuit(\kappa)$ entails the existence of a uniformly coherent, prolific $\kappa$-Souslin tree.
\end{proof}

\begin{remark} We close this section by offering a simpler variation of Lemma~\ref{blowup-nacc}.
Recalling the postprocessing function provided by Lemma~\ref{phi0}, it is easy to see that Clause~$(\ref{hyp-c-nacc})$ of Lemma~\ref{blowup-nacc} may as well be relaxed to the following:
\begin{enumerate}
\item[($\ref{hyp-c-nacc}'$)] for every stationary $B\s\kappa$  and every $\Lambda' < a(\Lambda,\kappa)$, the following set is nonempty:
\[
\{\alpha\in\Gamma\setminus\Omega\mid \exists C \in \mathcal C_\alpha [\Lambda' \le\otp(C)<\Lambda,\nacc(C)\s B]\}.
\]
\end{enumerate}
In the very same way, the requirement ``$\min(C_\alpha)=\min(S)$'' in Clause~(3) may be waived.
Consequently, in the special case $\Lambda<\kappa$, we may also assume that $\min\{\alpha,\Lambda\}=\Lambda$ for all $\alpha\in S$.
In addition, in the special case $\chi=\aleph_0$, we have $\Gamma(\cvec{C})=\acc(\kappa)$, and hence there is no need to start the proof by manipulating $C_{\alpha,i}\cap(\omega+1)$ for pairs $(\alpha,i)$
in order to ensure that $\Gamma(\cvec{D})$ will not exceed $\Gamma(\cvec{C})$. Consequently, in the special case $\chi=\aleph_0$, there is no harm in allowing $\omega$ to be an element of $\Omega$.
Finally, there are cases where we shall only care about the transversals $\vec C,\vec D$, but not about the sequences $\cvec{C},\cvec{D}$.
Putting all of these together with Lemma~\ref{hitting-implies-nontrivial}(1) yields the following simpler variation of Lemma~\ref{blowup-nacc}.
\end{remark}
\begin{lemma}
Suppose that $\Lambda\le\xi<\kappa$, with $\Lambda$ an infinite indecomposable ordinal. Suppose also:
\begin{enumerate}[(a)]
\item $\diamondsuit(\kappa)$ holds;
\item $\vec C=\langle C_\alpha \mid \alpha \in \acc(\kappa)\rangle$ is a transversal for $\square_\xi(\kappa,{<}\mu,{\sq^\Omega})$, for some fixed subset $\Omega\s\kappa$;
\item For every stationary $B\s\kappa$  and every $\Lambda' < a(\Lambda,\kappa)$, there is $\alpha\in\acc(\kappa)\setminus\Omega$ with $\Lambda' \le\otp(C_\alpha)<\Lambda$ such that $\nacc(C_\alpha)\s B$.
\end{enumerate}

Then there exists a transversal $\langle D_\alpha\mid\alpha\in\acc(\kappa)\rangle$ for $\square_\xi(\kappa,{<}\mu,{\sq^\Omega})$,
satisfying that for every cofinal $A\s\kappa$, there exists a stationary $S\s\kappa$, for which
\[
\left\{\alpha\in\acc(\kappa)\Biggm|
\begin{gathered}
\nacc(C_{\alpha})\s S,\\
\otp(C_{\alpha})=\cf(\Lambda)
\end{gathered}
\right\}\s \left\{\alpha\in\acc(\kappa) \Biggm|
\begin{gathered}
\nacc(D_\alpha)\s A,\\
\otp(D_\alpha)=\Lambda
\end{gathered}
\right\}.\qed
\]
\end{lemma}

\section{A canonical nonspecial Aronszajn tree}\label{section5}

Theorem~\ref{thm23} is the special case $(\mu,\chi,\Omega) := (\kappa,\aleph_0,\emptyset)$ of the following.

\begin{thm}\label{thm5.1} Suppose that $\kappa=\lambda^+$ for a singular cardinal $\lambda$, $\ch_\lambda$ holds,
and $\cvec{C}$ is a $\square_\lambda(\kappa,{<}\mu,{\sq^\Omega_\chi})$-sequence for some fixed $\Omega\s\kappa\setminus\{\omega\}$.

Then there exists a transversal $\langle C_\alpha\mid\alpha\in\Gamma\rangle$ for $\square_\lambda(\kappa,{<}\mu,{\sq^\Omega_\chi})$,
such that for every sequence $\langle A_i\mid i<\lambda\rangle$ of cofinal subsets of $\kappa$,
there exist stationarily many $\alpha\in\Gamma$ such that $\otp(C_\alpha)=\lambda$, $\min(C_\alpha)=\min(A_0)$ and $C_\alpha(i+1)\in A_i$ for all $i<\lambda$.
\end{thm}
\begin{proof} By Lemma~\ref{lemma3.5} (with $\xi:=\lambda$), pick a transversal $\vec C=\langle C_\alpha\mid\alpha \in \Gamma\rangle$ for $\cvec{C}$
and a faithful postprocessing function $\Phi_0:\mathcal K(\kappa)\rightarrow\mathcal K(\kappa)$ such that:
\begin{enumerate}
\item $\otp(\Phi_0(C_\alpha))<\lambda$ for all $\alpha\in \Gamma$;
\item For every cofinal $A\s\kappa$, there exist stationarily many $\alpha\in \Gamma$ with $\otp(\Phi_0(C_\alpha))=\cf(\lambda)$ and $\nacc(\Phi_0(C_\alpha))\s A$.
\end{enumerate}

By Lemma~\ref{hitting-implies-nontrivial}(1), $\cvec{C}$ in fact witnesses $\square_\lambda(\kappa,{<}\mu,{\sq^\Omega_\chi},{\nsin})$.
For all $\alpha\in\Gamma$, put $C_\alpha^\circ:=\Phi_0(C_\alpha)$.
By Lemma~\ref{pp-preserves-square2}, and since $\Phi_0$ is faithful, $\acts{\Phi_0}{C}$ is a $\square_\lambda(\kappa,{<}\mu,{\sq^\Omega_\chi},{\nsin})$-sequence,
for which $\vec{C^\circ}:=\langle C_\alpha^\circ\mid\alpha\in\Gamma\rangle$ is a transversal.
Now, appealing to Corollary~\ref{cor413} with $\acts{\Phi_0}{C}$, $\vec{C^\circ}$ and $\Gamma':=\Gamma$,
we obtain a transversal $\langle D_\alpha\mid\alpha\in\Gamma\rangle$ for $\square_\lambda(\kappa,{<}\mu,{\sq^\Omega_\chi},{\nsin})$,
satisfying that for every cofinal $A\s\kappa$, there exist stationarily many $\alpha\in\Gamma$ such that $\otp(D_\alpha)=\lambda$ and $\nacc(D_\alpha)\s A$.
By $\ch_\lambda$ and \cite{Sh:922}, $\diamondsuit(\kappa)$ holds, so let $\Phi_1$ be given by Lemma~\ref{phi0}.
For all $\alpha\in\Gamma$, put $C_\alpha^\bullet:=\Phi_1(D_\alpha)$.
By Lemma~\ref{pp-preserves-square2}(2), $\vec{C^\bullet}:=\langle C_\alpha^\bullet\mid\alpha\in\Gamma\rangle$ is indeed a transversal for $\square_\lambda(\kappa,{<}\mu,{\sq^\Omega_\chi})$,
so that by the choice of $\Phi_1$, $\vec{C^\bullet}$ is as sought.
\end{proof}
\begin{remark}Compare the preceding with Corollary~\ref{cor310}.\end{remark}

The next lemma provides a postprocessing-function version of \cite[Theorem~4.11]{rinot21}.

\begin{lemma}\label{lemma5.3} Suppose that  $\diamondsuit(\kappa)$ holds, and $\langle C_\alpha\mid\alpha \in \acc(\kappa)\rangle$ is a given transversal for $\square(\kappa,{<}\kappa,{\sq},V)$.
Then there exists a $\min$-preserving, $\acc$-preserving postprocessing function $\Phi:\mathcal K(\kappa)\rightarrow\mathcal K(\kappa)$ satisfying the following.
For every function $f:\kappa\rightarrow\kappa$ and every $A\s\acc(\kappa)$, there exists a stationary $G\s\kappa$,
such that for every $\delta\in A$ with $\otp(\nacc(C_\delta)\cap G)>f(\delta)+1$, there exists  $\gamma\in \nacc(\Phi(C_\delta))\cap A$ such that
$f(\gamma)=f(\delta)$ and $\Phi(C_\delta)\cap\gamma\sqsubseteq \Phi(C_\gamma)$.
\end{lemma}
\begin{proof} Let $\langle Z_\eta\mid \eta<\kappa\rangle$ be a $\diamondsuit(\kappa)$-sequence.
In particular, $Z_\eta\s\eta$ for all $\eta<\kappa$.
For each $x\in\mathcal K(\kappa)$, define $h_x:\otp(x)\rightarrow\otp(x)$ by stipulating:
$$h_x(i):=\sup(\otp(\{\beta\in \nacc(x)\cap\acc^+(Z_{x(i)})\mid Z_\beta = Z_{x(i)} \cap\beta \})).$$

Fix a bijection $\pi:\kappa\times\kappa\leftrightarrow\kappa$.
For all $j\le\eta<\kappa$, denote $Z_\eta^j:=\{ \gamma\in\acc(\eta)\mid \exists\varepsilon\in Z_\eta[\pi(j,\gamma)\in Z_\varepsilon]\}$.

We now define a collection of functions $\langle \sigma_x:\otp(x)\rightarrow\sup(x)\mid x\in\mathcal K(\kappa)\rangle$ by recursion over $\sup(x)$ for $x\in\mathcal K(\kappa)$.
For this, fix $\alpha \in \acc(\kappa)$, and suppose that $\sigma_x$ has already been defined for all $x\in\mathcal K(\kappa)$ with $\sup(x)<\alpha$.
Fix an arbitrary $x\in\mathcal K(\kappa)$ with $\sup(x)=\alpha$.
We now define $\sigma_x:\otp(x)\rightarrow\alpha$ by recursion over $i<\otp(x)$.
Put $\sigma_x(0):=x(0)$. Next, fix a nonzero $i<\otp(x)$, and suppose that $\sigma_x\restriction i$ has already been defined.
Noting that $h_x(i) \leq i \leq x(i)$, let
$$Y_x^i:=\{\gamma\in Z_{x(i)}^{h_x(i)}\mid  \sup(x\cap x(i))<\gamma, \sigma_x\restriction i=\sigma_{C_\gamma}\restriction i\},$$
and let
$$\sigma_x(i):=\begin{cases}
\min(Y^{i}_x),&\text{if } Y^{i}_x\neq\emptyset;\\
x(i),&\text{otherwise}.
\end{cases}$$

Finally, define $\Phi:\mathcal K(\kappa)\rightarrow\mathcal K(\kappa)$ by stipulating $\Phi(x):=\im(\sigma_x)$.

\begin{claim}\label{claim51kj} $\Phi$ is a $\min$-preserving, $\acc$-preserving postprocessing function.
\end{claim}
\begin{proof} Let $x\in\mathcal K(\kappa)$ be arbitrary. For all $i<\otp(x)$, we have:
\begin{itemize}
\item[(a)] $x(i)<\sigma_x(i+1)=\Phi(x)(i+1)\le x(i+1)$ for all $i<\otp(x)$; and
\item[(b)] $\sigma_x(i) = \Phi(x)(i) = x(i)$ for all limit $i<\otp(x)$, including $i=0$.
\end{itemize}

Consequently, $\Phi(x)$ is a club in $\sup(x)$ and $\Phi$ is $\min$-preserving and $\acc$-preserving. Next, fix $\bar\alpha\in\acc(x)$.
Then a straight-forward induction over $i<\otp(x\cap\bar\alpha)$ establishes that  $h_{x\cap\bar\alpha}(i)=h_x(i)$,  $Y^{i}_{x\cap\bar\alpha}=Y^{i}_x$ and $\sigma_{x\cap\bar\alpha}(i)=\sigma_x(i)$ for all $i<\otp(x\cap\bar\alpha)$,
so that $\Phi(x)\cap\bar\alpha = \Phi(x\cap\bar\alpha)$.
\end{proof}

To see that $\Phi$ is as sought, suppose that we are given a function $f:\kappa\rightarrow\kappa$ along with an arbitrary subset $A\s \acc(\kappa)$.
For all $\alpha\in\acc(\kappa)$, denote $D_\alpha:=\Phi(C_\alpha)$.
By Lemma~\ref{pp-preserves-square}, $\langle D_\alpha\mid\alpha \in \acc(\kappa)\rangle$ is a transversal for $\square(\kappa,{<}\kappa,{\sq},V)$.
Fix a surjection $\varphi:\kappa\rightarrow\{D_\delta\cap\beta\mid \delta\in\acc(\kappa),\beta<\kappa\}$.
By Proposition~\ref{transversal-width}(2), the following set is a club in $\kappa$:
$$D := \{ \eta\in\acc(\kappa)\mid \varphi[\eta]=\{ D_\delta\cap\beta \mid \delta\in\acc(\kappa), \beta<\eta \} \}.$$

For all $j,\alpha<\kappa$, put $S_j:=A\cap f^{-1}\{j\}$ and $S_{j,\alpha}:=\{\xi\in S_j\mid \varphi(\alpha)\sq D_\xi\}$.
Then, put $Z:=\{ \pi(j,\gamma)\mid j<\kappa, \gamma\in S_j\}$.
Define a function $g:\kappa\times\kappa\times\kappa\rightarrow\kappa$ by stipulating:
$$g(\alpha,\zeta,j):=\begin{cases}
\sup (S_{j,\alpha}) +1,&\text{if } \sup (S_{j,\alpha}) <\kappa;\\
\min (S_{j,\alpha} \setminus (\zeta+1)),&\text{otherwise}.
\end{cases}$$

Finally, consider the club $E:=\{\delta\in D\mid g[\delta\times\delta\times \delta]\s\delta=\pi[\delta\times\delta]\}$, and its stationary subsets:
\begin{itemize}
\item $F:=\{\varepsilon\in E\mid Z_\varepsilon=Z\cap\varepsilon\}$, and
\item $G:=\{\eta\in\acc^+(F)\mid Z_\eta=F\cap\eta\}$.
\end{itemize}

\begin{claim}\label{c532} For every $\eta\in G$ and $j<\eta$, we have $Z_\eta^j=S_j\cap\eta$.
\end{claim}
\begin{proof} Let $\eta\in G$ be arbitrary. Then $\eta\in\acc^+(F)$ and $Z_\eta=F\cap\eta$, so that $\bigcup_{\varepsilon\in Z_\eta}Z_\varepsilon=\bigcup_{\varepsilon\in F\cap\eta}Z_\varepsilon=Z\cap\eta$.
Now, let $j<\eta$ be arbitrary. Then, $Z_\eta^j=\{\gamma\in\acc(\eta)\mid\pi(j,\gamma)\in Z\cap\eta\}$.
As $j<\eta$ and $\pi[\eta\times\eta]=\eta$, the definition of $Z$ implies that $Z_\eta^j=S_j\cap\eta$.
\end{proof}

\begin{claim} Suppose that $\delta\in A$ and $\otp(\nacc(C_\delta)\cap G)>f(\delta)+1$.

Then there exists $\gamma\in\nacc(D_\delta)\cap A$ such that $f(\gamma)=f(\delta)$ and $D_\delta\cap\gamma\sqsubseteq D_\gamma$.
\end{claim}
\begin{proof} Denote $j:=f(\delta)$.
Fix $\eta\in\nacc(C_\delta)\cap G$ such that $\otp(\nacc(C_\delta)\cap G\cap \eta)=j+1$, and then fix $i<\otp(C_\delta)$ such that $\eta=C_\delta(i+1)$.
As $C_\delta(i+1)=\eta\in G$, we have $Z_{C_\delta(i+1)}=F\cap\eta$ and
$$\{\beta\in \nacc(C_\delta)\cap\acc^+(Z_{C_\delta(i+1)})\mid Z_\beta = Z_{C_\delta(i+1)} \cap\beta \}=\nacc(C_\delta)\cap G\cap\eta,$$
so that $h_{C_\delta}(i+1)=\sup(j+1)=j$.
By $\sigma_{C_\delta}(i) \leq C_\delta(i) < C_\delta(i+1)=\eta$ and the fact that $\eta$ is a limit ordinal, it follows that $\beta:= \sigma_{C_\delta}(i) +1$ is $<\eta$.
Then, since $\eta \in D$, we may fix some $\alpha<\eta$ such that $\varphi(\alpha)=D_\delta\cap \beta$.
Write $\zeta:=C_\delta(i)$.
Since $\alpha<\eta$, $\zeta<\eta$, and $j<\otp(\nacc(C_\delta)\cap G\cap\eta)\le\eta$, the fact that $\eta\in G\s E$ entails $g(\alpha,\zeta,j)< \eta$.
By definition of $g$, if $\sup(S_{j,\alpha})<\kappa$, then $\sup(S_{j,\alpha})+1=g(\alpha,\zeta,j)<\eta<\delta$.
However, $\delta\in S_{j,\alpha}$, and hence $\sup(S_{j,\alpha})=\kappa$.
Consequently, $g(\alpha,\zeta,j)=\min(S_{j,\alpha}\setminus (\zeta+1))$. Altogether,
$$g(\alpha,\zeta,j)\in S_{j,\alpha}\cap(\zeta,\eta).$$
By Claim~\ref{c532}, $Z^j_\eta=S_j\cap\eta$.
As $h_{C_\delta}(i+1)=j$ and $\eta = C_\delta(i+1)$, the definition of $Y_{C_\delta}^{i+1}$ yields:
\begin{align*}Y_{C_\delta}^{i+1} &=
\{\gamma\in S_j\cap\eta\mid \sup(C_\delta\cap C_\delta(i+1))<\gamma, \sigma_{C_\delta}\restriction (i+1)=\sigma_{C_\gamma}\restriction(i+1)\}\\
&=\{\gamma\in S_j\cap\eta\mid C_\delta(i)<\gamma, D_\delta\cap\beta=D_\gamma\cap\beta\}\\
&=\{\gamma\in S_j\cap\eta\mid \zeta<\gamma, \varphi(\alpha)\sq D_\gamma\} \\
&= S_{j,\alpha}\cap(\zeta,\eta).
\end{align*}
In particular, $Y^{i+1}_{C_\delta}$ is a nonempty subset of $S_{j,\alpha}$. Write $\gamma:=\sigma_{C_\delta}(i+1)$.
Then $\gamma\in Y^{i+1}_{C_\delta}\cap\nacc(D_\delta) \s S_{j,\alpha}$, so that $\gamma \in A$, $f(\gamma)=j = f(\delta)$, and $D_\delta\cap\gamma = D_\delta\cap\beta = \varphi(\alpha) \sq D_\gamma$, as sought.
\end{proof}
This completes the proof.
\end{proof}

\begin{thm}\label{thm5.4} Suppose that $\lambda$ is an uncountable cardinal and $\ch_\lambda$ holds. Then $(1)\implies(2)$:
\begin{enumerate}
\item There exists a transversal $\vec C=\langle C_\alpha\mid\alpha \in \acc(\lambda^+)\rangle$ for $\square^*_\lambda$ such that for every stationary $G\s\lambda^+$
there exists $\alpha \in \acc(\lambda^+)$ for which $\otp(\nacc(C_\alpha)\cap G)=\lambda$.
\item There exists a $C$-sequence $\vec{D}=\langle D_\alpha\mid\alpha\in\acc(\lambda^+)\rangle$ such that:
\begin{itemize}
\item $(\mathcal T(\rho_0^{\vec D}),{\stree})$ is a special $\lambda^+$-Aronszajn tree.
\item $(\mathcal T(\rho_1^{\vec D}),{\stree})$ is a normal nonspecial $\lambda^+$-tree, which is $\lambda^+$-Aronszajn in any $\lambda$-distributive forcing extension.
\end{itemize}
\end{enumerate}
\end{thm}
\begin{proof} Denote $\kappa:=\lambda^+$. By $\ch_\lambda$ and \cite{Sh:922}, $\diamondsuit(\kappa)$ holds.
Let $\vec C$ be as in Clause~(1). By Lemma~\ref{split_amenable}, there exists a postprocessing function $\Phi:\mathcal K(\kappa)\rightarrow\mathcal K(\kappa)$
and an injection $h:\kappa\rightarrow\kappa$ such that, for every $\iota<\kappa$, $\{ \delta\in\acc(\kappa)\mid \min(\Phi(C_\delta))=h(\iota)\}$ is stationary in $\kappa$.
As made clear by the proof of that lemma, we also have $\Phi(x)=^* x$ for all $x\in\mathcal K(\kappa)$.
Then, by a manipulation along the lines of the proof of Lemma~\ref{split_amenable2}, there exists a postprocessing function $\Phi':\mathcal K(\kappa)\rightarrow\mathcal K(\kappa)$ satisfying:
\begin{itemize}
\item For every $\iota<\kappa,$ $\{ \delta\in\acc(\kappa)\mid \min(\Phi'(C_\delta))=\iota\}$ is stationary in $\kappa$;
\item $\Phi'(x)=^* x$ for all $x\in\mathcal K(\kappa)$.
\end{itemize}

Put $\vec C^\bullet:=\langle \Phi'(C_\alpha)\mid \alpha\in\acc(\kappa)\rangle$.
Now, by appealing to Lemma~\ref{lemma5.3} with $\vec{C^\bullet}$, we infer from Lemma~\ref{pp-preserves-square} the existence of a transversal $\vec D=\langle D_\alpha\mid\alpha \in \acc(\kappa)\rangle$ for $\square_\lambda^*$ such that:
\begin{enumerate}
\item[(a)] For every $\iota<\lambda^+$, $\{\delta\in\acc(\lambda^+)\mid \min(D_\delta)=\iota\}$ is cofinal (and even stationary) in $\lambda^+$;
\item[(b)] For every function $f:\lambda^+\rightarrow\lambda$, there exist a limit ordinal $\delta\in[\lambda,\lambda^+)$ and a limit ordinal $\gamma\in\nacc(D_\delta)$ such that $f(\gamma)=f(\delta)$ and $D_\delta\cap\gamma\sq D_\gamma$.
\end{enumerate}

Since $\vec D$ is a transversal for $\square^*_\lambda$, the first two lines of the proof of \cite[Lemma~6.1.14]{MR2355670} with $\theta := \lambda^+$ show that $(\mathcal T(\rho_0^{\vec D}),{\stree})$ is a special $\lambda^+$-Aronszajn tree.
As $\mathcal T(\rho_1^{\vec D})$ is a projection of $\mathcal T(\rho_0^{\vec D})$ (indeed, under the map $\sigma\mapsto\max(\im(\sigma))$),
we know that $(\mathcal T(\rho_1^{\vec D}),{\stree})$ is a $\lambda^+$-tree.
By Clause~(a), $(\mathcal T(\rho_1^{\vec D}),{\stree})$ is normal.
By Clause~(b) and the proof of \cite[Theorem~3]{MR2194042},  $(\mathcal T(\rho_1^{\vec D}),{\stree})$ is nonspecial.

Finally, let $\mathbb P$ be an arbitrary $\lambda$-distributive notion of forcing.
Work in $V^{\mathbb P}$, so that $\lambda^+=(\lambda^+)^V$ and $\mathcal T(\rho_1^{\vec D})=(\mathcal T(\rho_1^{\vec D}))^V$.
Towards a contradiction, suppose that $b:\lambda^+\rightarrow\lambda$ is such that $\{b\restriction\alpha\mid\alpha<\lambda^+\}$ is a cofinal branch through $(\mathcal T(\rho_1^{\vec D}),{\stree})$.
By a standard argument (see, e.g., \cite[Corollary~2.6]{rinot18}), there exists a cofinal subset $X\s\lambda^+$ such that for all $\beta<\alpha$ both from $X$, we have $\rho_1^{\vec D}(\beta,\alpha)=b(\beta)$.
Then there must exist a cofinal subset $Y\s X$ on which $b$ is constant, say, $b[Y] = \{\nu\}$ for some $\nu < \lambda$.
Choosing $\alpha\in Y$ such that $|Y \cap \alpha| = \lambda$ gives
\[
| \{ \beta<\alpha \mid \rho_1^{\vec D}(\beta,\alpha) = \nu \} | = \lambda > \max\{|\nu|,\aleph_0\},
\]
contradicting \cite[Lemma~6.2.1]{MR2355670}.
\end{proof}

Finally, we derive Theorem~B:

\begin{cor}\label{thm51} Suppose that $\lambda$ is a singular cardinal and $\square^*_\lambda+\ch_\lambda$ holds.

Then there exists a $C$-sequence $\vec{C}=\langle C_\alpha\mid\alpha \in \acc(\lambda^+) \rangle$ such that:
\begin{itemize}
\item $(\mathcal T(\rho_0^{\vec C}),{\stree})$ is a special $\lambda^+$-Aronszajn tree.
\item $(\mathcal T(\rho_1^{\vec C}),{\stree})$ is a nonspecial $\lambda^+$-Aronszajn tree which is normal but not $\lambda$-distributive.
\end{itemize}
\end{cor}
\begin{proof} Appeal to Theorem~\ref{thm5.1} with $(\mu,\chi,\Omega) := (\lambda^+,\aleph_0,\emptyset)$ to obtain a $C$-sequence satisfying clause~(1) of Theorem~\ref{thm5.4}.
Then, let $\vec D$ be the $C$-sequence provided by Clause~(2) of  Theorem~\ref{thm5.4}.
Since forcing with the normal $\lambda^+$-tree $\mathcal T(\rho_1^{\vec D})$ would introduce a cofinal branch through the tree, it follows in particular that $(\mathcal T(\rho_1^{\vec D}),{\stree})$ cannot be $\lambda$-distributive.
\end{proof}

\section*{Acknowledgments}
We thank Yair Hayut for illuminating discussions on $\square(\kappa,{<}\mu)$-sequences.

The main results of this paper were presented by the first author at the Toronto Set-Theory Seminar, April 2017,
and at the \emph{6th European Set-Theory Conference}, Budapest, July 2017,
and by the second author at the \emph{14th International Workshop on Set Theory}, Luminy, October 2017.
We thank the organizers of the respective meetings for the invitations.

\end{document}